\documentclass[a4paper, 12pt]{amsart}
\usepackage[utf8]{inputenc}
\usepackage[english]{babel}

\usepackage[toc,page]{appendix}
\usepackage{vmargin}
\setpapersize{A4}
\setmarginsrb{18mm}{15mm}{18mm}{15mm}{0mm}{8mm}{0mm}{8mm}
\usepackage[centertags]{amsmath}
\usepackage{amsthm}
\usepackage[bookmarks=false]{hyperref}
\usepackage{hyperref}
\usepackage[pass]{geometry}
\usepackage{amssymb}
\usepackage{enumitem}
\usepackage{mathtools}
\usepackage{mathrsfs}
\usepackage{color}
\mathtoolsset{showonlyrefs}
\DeclarePairedDelimiterX{\scal}[2]{\langle}{\rangle}{#1, #2}
\DeclarePairedDelimiterX{\norm}[1]{\lVert}{\rVert}{#1}
\DeclarePairedDelimiterX{\normi}[1]{\lVert}{\rVert_\infty}{#1}
\DeclareMathOperator{\dive}{div}
\DeclareMathOperator{\Id}{Id}
\DeclareMathOperator{\supp}{supp}
\DeclareMathOperator{\vol}{Vol}
\DeclareMathOperator{\loc}{loc}
\DeclareMathOperator{\tang}{Tan}

\DeclareMathOperator*{\per}{Per}
\newcommand{\rom}[1]{\uppercase\expandafter{\romannumeral #1\relax}}
\renewcommand{\u}{3}
\newcommand{\gk}{m}

\newcommand{\breg}{C^{2,\alpha}}%
\newcommand{\des}{\partial^*}
\newcommand{\td}{\Theta_2}

\newcommand{\as}{A_{\Sigma}}

\renewcommand{\phi}{\varphi}
\newcommand{\g}{\Gamma}
\newcommand{\om}{\Omega}
\newcommand{\e}{\varepsilon}

\newcommand{\NN}{\mathcal{N}}
\newcommand{\s}{\Sigma}
\renewcommand{\l}{\ell}
\newcommand{\psiij}{\psi_{i,j}}
\newcommand{\ts}[1]{\tilde{\s}_{#1}}
\newcommand{\tg}[1]{\tilde{\g}_{#1}}

\renewcommand{\t}{\tau}
\newcommand{\opm}{\mathscr{O}(\M)}
\newcommand{\sw}{\mathscr{S}}
\newcommand{\anul}[2]{\mathcal{AN}_{#1}({#2})}
\newcommand{\co}{\mathcal{C}}
\newcommand{\hpp}{\mathcal{HP}\big(\ck{k},\an\big)}
\newcommand{\cu}{\V_2(\M) \times \V_2(\pM)}
\newcommand{\am}{A^{2\m}}
\newcommand{\ami}{A_\infty}
\newcommand{\gnm}{G_2(\M)}
\newcommand{\gnpm}{G_2(\pM)}
\newcommand{\pas}{(\s,\g)}

\newcommand{\upair}[1]{(U_{1,{#1}},U_{2,{#1}})}

\newcommand{\restr}[1]{\llcorner_{#1}}
\DeclareMathOperator*{\an}{An}
\DeclareMathOperator{\inter}{Int}
\newcommand{\intm}{\interior{\M}}
\newcommand{\bi}{\partial_i}
\newcommand{\bb}{\partial_b}

\newcommand{\bio}{\bi \om}
\newcommand{\bbo}{\bb \om}
\newcommand{\biot}{\bi \om_t}
\newcommand{\bbot}{\bb \om_t}
\newcommand{\db}{\gamma}
\newcommand{\dbo}{\db(\om)}
\newcommand{\interior}[1]{%
  {\kern0pt#1}^{\mathrm{o}}%
}
\newcommand{\clos}[1]{\mkern 1mu\overline{\mkern-1mu#1\mkern-1mu}\mkern 1mu}

\newcommand{\pM}{\partial \M}
\newcommand{\de}{\partial}%
\newcommand{\oswo}{\{\om_t\}_{t \in [0,1]}} 
\newcommand{\swo}{\{\s_t\}_{t \in [0,1]}} 
\newcommand{\I}{[0,1]}
\renewcommand{\c}{(\s,\g)}
\newcommand{\ct}[1]{(\s_{#1},\g_{#1})}
\newcommand{\tc}{(\tilde{\s},\tilde{\g})}
\renewcommand{\tt}[1]{(\tilde{\s}_{#1},\tilde{\g}_{#1})}
\newcommand{\ck}[1]{(\s^{#1},\g^{#1})}
\newcommand{\ctk}[2]{(\s_{#1}^{#2},\g_{#1}^{#2})}
\newcommand{\tk}[1]{(\tilde{\s}^{#1},\tilde{\g}^{#1})}
\newcommand{\ttk}[2]{(\tilde{\s}_{#1}^{#2},\tilde{\g}_{#1}^{#2})}
\newcommand{\ft}[1]{(\Phi_{#1},\Psi_{#1})}

\newcommand{\cswo}{(\s_t, \g_t)_{t \in \I}}
\newcommand{\ctswo}{(\tilde{\s}_t, \tilde{\g}_t)_{t \in \I}}
\newcommand{\cswok}{(\s_t^k, \g_t^k)}
\newcommand{\ctswok}{(\tilde{\s}_t^k, \tilde{\g}_t^k)}
\newcommand{\csw}{(\s_t,\g_t)}
\newcommand{\ctsw}{(\tilde{\s}_t,\tilde{\g}_t)}
\newcommand{\cktk}{(\s_{t_k}^k,\g_{t_k}^k)}
\newcommand{\st}{\s_t}
\newcommand{\gt}{\g_t}
\newcommand{\ot}[1]{\om_{#1}}
\newcommand{\tom}{\tilde{\om}}
\newcommand{\tot}[1]{\tilde{\om}_{#1}}
\newcommand{\tok}[1]{\tilde{\om}^{#1}}
\newcommand{\con}{\eta}
\newcommand{\z}{\zeta}%
\newcommand{\fa}{F_a}%
\newcommand{\dfa}{\delta_{\fa}}
\newcommand{\dtfa}{\delta^2_{\fa}}
\newcommand{\ms}{\max_{t \in [0,1]}}
\newcommand{\m}{{m_0}}
\newcommand{\nv}{\norm{V}}
\newcommand{\nw}{\norm{W}}

\newcommand{\nds}[1]{|\delta^s {#1}|}

\newcommand{\tv}{\tilde{V}}
\newcommand{\tw}{\tilde{W}}
\newcommand{\ntv}{\norm{\tv}}
\newcommand{\ntw}{\norm{\tw}}
\newcommand{\pair}{(V,W)}
\newcommand{\tpair}{(\tv,\tw)}

\newcommand{\pairij}{(V_{i}^j,W_{i}^j)}

\newcommand{\rij}{r_{i,j}}
\newcommand{\bij}{\B_{i,j}}
\newcommand{\tbij}{\tilde{\B}_{i,j}}
\newcommand{\pairt}{(V_t,W_t)}
\newcommand{\tpairk}[1]{(\tilde{V}_{#1},\tilde{W}_{#1})}
\newcommand{\pairk}[1]{(V^{#1},W^{#1})}
\newcommand{\paira}[1]{(V_{#1},W_{#1})}
\newcommand{\cpair}{(C,K)}
\newcommand{\tcpair}{(\tilde{C},\tilde{K})}
\newcommand{\sumpair}{\norm{V}+a\norm{W}}
\newcommand{\ypair}{Y_{\pair}}
\newcommand{\dil}[1]{D_{{#1}}}
\newcommand{\me}{\mathscr{M}}
\newcommand{\mep}{\mathscr{M}^+}

\renewcommand{\-}{\setminus}
\newcommand{\Xt}{\mathfrak{X}_t(\M)}

\newcommand{\Xc}{\mathfrak{X}_c(\M)}
\newcommand{\Xm}{\mathfrak{X}(\M)}
\newcommand{\X}{\mathfrak{X}}
\newcommand{\Xo}{\mathfrak{X}_0(\M)}
\newcommand{\Xp}{\mathfrak{X}_\perp(\M)}
\newcommand{\cc}{\subset\joinrel\subset}
\newcommand{\ind}{\mathbf{1}}
\newcommand{\D}{\mathcal{D}}
\newcommand{\B}{\mathcal{B}}
\newcommand{\T}{\Theta}

\newcommand{\tum}{\Theta^{*}_{m}}
\newcommand{\tlm}{\Theta_{*m}}
\newcommand{\tm}{\Theta_m}

\newcommand{\muv}{\sigma_V}

\newcommand{\muj}{\sigma_j}

\newcommand{\wto}{\stackrel{\ast}{\rightharpoonup}}

\newcommand{\tto}[2]{\xrightarrow[{#2}]{{#1}}}
\newcommand*\dif{\mathop{}\!\mathrm{d}}

\DeclareMathOperator*{\Div}{div}

\DeclareMathOperator*{\diam}{diam}

\DeclareMathOperator*{\dist}{dist}

\newcommand{\R}{\mathbb{R}}
\newcommand{\ru}{\mathbb{R}^{3}}

\newcommand{\N}{\mathbb{N}}

\newcommand{\M}{\mathcal{M}}

\newcommand{\F}{\mathcal{F}}

\newcommand{\V}{\mathcal{V}}

\renewcommand{\T}{\Theta}

\newcommand{\haus}[1]{\mathcal{H}^{#1}}

\newcommand{\hn}{\mathcal{H}^2}
\newcommand{\hu}{\mathcal{H}^{1}}

\numberwithin{equation}{section}
\newtheorem{theorem}{Theorem}[section]
\newtheorem*{theorem*}{Theorem}
\newtheorem{lemma}[theorem]{Lemma}

\newtheorem{proposition}[theorem]{Proposition}
\theoremstyle{definition}
\newtheorem{defi}{Definition}[section]

\theoremstyle{remark}
\newtheorem{rem}{Remark}[section]
\newlist{steps}{enumerate}{1}
\setlist[steps, 1]{label = \textbf{Step \arabic*:}}
\newcounter{count}

\title{Min-max construction of minimal surfaces with a fixed angle at the boundary}

\author[L.~De Masi]{Luigi De Masi}
\address{\textit{L.~De Masi:} SISSA Via Bonomea 265, 34136 Trieste, Italy}
\email{ldemasi@sissa.it}

\author[G.~De Philippis]{Guido De Philippis}
\address{\textit{G.~De Philippis:} Courant Institute of Mathematical Sciences, New York University, 251 Mercer St., New York, NY 10012, USA.}
\email{guido@cims.nyu.edu}

\thanks{{\bf Acknowledgments.} 
	G.D.P has been partially supported by the NSF grant DMS 2055686 and by the Simons Foundation.}

\subjclass[2010]{}


\begin{document}
\maketitle

\begin{abstract}
We  prove the existence of  minimal surfaces in a bounded convex subset of $\R^3$ $\M$  intersecting  the boundary of $\M$ with a fixed contact angle. The proof is based on a   min-max construction in the spirit of Almgren-Pitts for the capillarity functional.
\end{abstract}

\section{Introduction}\label{sec:introduction}
\subsection{Background and main result} Minimal surfaces, i.e.\ critical points of the area functional,  are a fundamental object in Geometric Analysis and the problem of finding a minimal surface (possibly with given side constraints) in a given Riemannian manifold \(\mathcal M\) has been a fundamental topic of research in the last 50 years.

When the topology of the ambient manifold \(\mathcal M\) is non-trivial, existence result can be often  obtained via minimization of the area functional in suitable homology classes. On the other hand, when the topology is trivial, it is easy to see that non-trivial minimizers do not exist and one has to rely on different methods.

Since the seminal work of Birkhoff about the existence of closed geodesic in Riemannian manifolds,  the use  of min-max techniques has been a successful tool to show existence of critical points for a variety of functionals. 

The implementation of min-max techniques for the construction of minimal surfaces has been started in the seminal work of Almgren and Pitts and deeply relies on techniques of Geometric Measure Theory. One of the fundamental outcomes of the theory has been to show the existence of a \(d\) dimensional minimal surface in every \((d+1)\)-dimensional manifold, \cite{pitts2014existence,colding2003min,delellis2009existence}. (To be precise the minimal surface is allowed to have a ``small'' singular set when \(d\ge 7\)) . Since their seminal work, the strategy has been suitably modified  to show the existence of surfaces with prescribed curvature \cite{zhou2019min}, given boundary \cite{delellis2017minmax} as well as free boundary, \cite{delellis2017minmax}. Furthermore these techniques have been as well used to show the existence of \emph{infinitely many} minimal surfaces in a given Riemannian manifold \cite{marques2016existence,IrieMarquesNeves18,song2018existence}

%
The main result of this paper is to show how Almgren-Pitts techniques (and more precisely their revised version by Colding-De Lellis, \cite{colding2003min}) allows the existence of minimal surfaces with 	\emph{prescribed contact angle at the boundary} in convex subsets of \(\ru\).

\begin{theorem}\label{thm:main}
Let \(\M\subset \R^3\) be a  convex bounded open set with \(C^{2,\alpha}\) boundary. Then, for all \(\theta \in (0,\pi/2)\)  there exists a minimal surface  \(\Sigma\subset \M\) with \(\partial \Sigma \subset \partial \M\) and 
\[
\eta_{\Sigma} \cdot N=\sin \theta \qquad \text{on \(\partial \Sigma\)}
\]
where \(\eta_{\Sigma}\)  is the exterior co-normal field to  \(\Sigma\) at \(\partial \Sigma\)  and \(N\) is the exterior unit normal to \(\partial \M\).
\end{theorem}

\subsection{Sketch of the argument and structure of the paper}

The main idea of the proof, which has  already appeared in the Master Thesis of the first author \cite{demasi2018}, where some  preliminary steps of this program were carried out, is to apply  Almgren-Pitts/Colding-De Lellis techniques to the \emph{capillarity energy}:
\begin{equation}\label{e:cap}
\fa(\om) = \hn(\bi{\om}) + \cos \theta \hn(\bb \om),
\end{equation}
where \(\Omega\subset \mathcal M\) and \(\bi{\om}=\partial \Omega\cap \interior{\M}\), \(\bb \om=\partial \Omega\cap \partial \M\). Note indeed that a (smooth) stationary point of \eqref{e:cap} shall  satisfy
\[
H_{\Sigma }=0  \text{ in \(\interior{\M}\),} \qquad \eta_{\Sigma} \cdot N=\sin \theta \quad \text{on \(\partial \Sigma\cap \partial \M\)}
\]
where \(\Sigma=\bi{\om}\). With respect to the case of the pure area functional, or to the free boundary one (when \(\theta=\pi/2\))  two main new difficulties arise. First,  a priori, the capillarity functional is only defined for surfaces which are boundary of a set and it is not clear how to extend it to class of rectifiable varifold, a technical step which is instrumental to the proof of Theorem \ref{thm:main}. The key idea here is to ``decouple'' the functional and initially look at the interior and boundary part as independent. Once this has been done, the pull-tight procedure can be  carried on and one can show the existence of a stationary varifold which satisfies a suitable notion of ``angle condition''  which a slight modification on the on proposed by Kagaya and Tonegawa in  \cite{kagaya2017}. This will be done in Section \ref{sec3} and \ref{sec:4}.

The second key step is to show regularity of the obtained varifold; here we  start by following  the stategy of \cite{pitts2014existence,colding2003min,delellis2009existence}, and showing that the stationary varifold found in the previous step is \(\e\)-minimizing in annuli and thus it admits replacements there, see Section \ref{sec5} and \ref{sec7}. Existence of replacements then follows by the the regularity theory of  \cite{DePMag14,dephilippis2014dimensional} together with the  curvature estimate proved for \(2\)-dimensional  stable critical point of the capillarity functional. These estimates, recalled in Section \ref{sec6} can be easily obtained if one assumes an a priori area growth on the surface (a condition which is met on our construction) and have been recently established in full generality in \cite{hong2021capillary}. Existence of replacements implies that the obtained varifold is induced by a smooth surface outside a finite set of points and the proof follows the one in~\cite{delellis2017minmax}.  In order to remove the final singular points we instead rely on a new argument which strongly uses that \(\theta<\pi/2\). Under this assumption we can show that stable surface is indeed graphical and satisfies a suitable free boundary problem, from which we can deduce the desired regularity. This is done in Section \ref{sec8}. 
%
%
%
%

We conclude the  introduction by commenting the restriction to  open convex subsets of \(\R^3\). First of all, it will be clear from the proof that we can endow \(\mathcal M\) with any (smooth) Riemaniann metric for which the boundary is still convex  and obtain the same result. Convexity on the other hand  is used in showing that the ``interior'' part of \(\Sigma\) does not touch the boundary of \(\M\) and it looks plausible that by following the argument of \cite{Li2015}, the assumption can be dispensed. The  restriction to dimension \(3\)  appears instead only in one (crucial) point of our proof, namely in showing  curvature estimates for stable capillarity surfaces in Section \ref{sec6}. We  conjecture that similar estimates should hold true in higher dimension as well and if this is the case, the same arguments of the paper will show the existence of \(d\)-dimensional surfaces into convex sets of \((d+1)\)-dimensional manifolds. We refer the reader to the forthcoming PhD thesis of the first author  for some of the mentioned generalizations, \cite{demasi2022}.

\bigskip

While completing  this paper, we learned that  Chao Li, Xin Zhou and Jonathan Zu \cite{li2021}  had just obtained a similar result  with related  techniques.

\section{Notations}\label{sec:notations and basic results}
\subsection{Basic notations}
We work in $\ru$.
If $x \in \ru$, we denote by $B_r(x)$ the closed ball with center $x$ and radius $r$ 
Moreover we set $B_r := B_r(0)$. If $0<s<r$ we call $\an(x,s,r)$ the open annulus $\interior{B_r(x)} \- B_s(x)$ and $\anul{\rho}{x}$ the set of open annuli $\an(x,s,r)$ such that $r< \rho$.
For any $\lambda>0$ and $x \in \ru$ we consider the dilation map
\begin{equation}\label{eq:dilation_map}
\dil{x,r}(y)= \frac{1}{r}(y-x).
\end{equation}

For every $A \subset \ru$, we denote by $\ind_A$ the indicator function of $A$, by $\clos{A}$ and $\interior{A}$ respectively the closure and the interior of $A$ in the euclidean topology.
We denote by $S$ a generic $\gk$-dimensional linear subspace (or $\gk$-plane) of $\ru$ and we write $S^\perp$ for the orthogonal complement of $S$ in $\ru$. We denote by $P_S$ the orthogonal projection on $S$. If $X$ is a $C^1$ vector field, we call $\dive_S X$ the scalar product $P_S \cdot DX$. If $\tau_1, \dots, \tau_\gk$ is an orthonormal basis of $S$, by simple computations one has
\begin{equation*}\label{eq:tangential_divergence_defi}
\dive_S X(x)=
\sum_{i=1}^\gk  {D_{\tau_i} X(x)}\cdot{\tau_i}.
\end{equation*}
If $\Gamma$ is a $C^1$ $\gk$-dimensional sub-manifold in $\ru$ and $x \in \Gamma$, we write $T_x \Gamma$ for the tangent space to $\Gamma$ at $x$.

We work on a compact domain $\M \subset \ru$ with $\breg$ boundary $\pM$ for some $\alpha \in (0,1)$. We will assume also that $\pM$ is uniformly  convex, that is the principal curvatures of $\pM$ satisfy an uniform positive lower bound.
We write $N(x)$ for the exterior unit normal vector to $\pM$ at $x$. If $\om \subset \M$ is open in the relative topology of $\M$ and has finite perimeter, we call $\bi\om$ and $\bb\om$ respectively the internal and the external reduced boundary of $\om$, that is
\begin{equation*}
\bi \om = \de^* \om \cap \intm,
\qquad
\bb \om = \de^* \om \cap \pM,
\end{equation*}
where $\de^* \om$ denotes the reduced boundary of $\om$.
If $\bi \om$ and $\bb \om$ are sufficiently regular $2$-submanifolds of $\ru$, we denote by $\dbo$ their common boundary (as manifolds).

We work with several classes of vector fields on $\M$, which we denote with the letter $\X$ with subscripts based on their behavior on $\pM$:
\begin{gather*}\label{eq:vector_fields}
\Xm = C^1(\clos{\M},\R^{3}),
\qquad
\Xt = \{X \in \Xm \mid X(x) \in T_x \pM, \,\, \forall x \in \pM\},
\notag
\\
\Xp = \{X \in \Xm \mid X(x) \in (T_x \pM)^\perp ,\,\, \forall x \in \pM\}
\\
\Xo = \{X \in \Xm \mid X(x)=0 , \,\,\forall x \in \pM\}
\quad
\Xc = \{X \in \Xm \mid \supp X \cc \interior{\M}\}. \notag
\end{gather*}

\subsection{Measures}
If $A \subset \ru$, $\me(A,\R^{m})$ is the space of $\R^m$-valued Radon measures on $A$ and $\mep(A)$ is the space of positive Radon measures on $A$. If $\mu \in \me(A,\R^m)$ we denote by $|\mu|$ the total variation measure of $\mu$. If $B \subset \ru$ is Borel, we write $\mu \llcorner B$ for the restriction of the measure $\mu$ to $B$.
We endow $\me(A,\R^m)$ with the weak*-topology: i.e. we say that a sequence of Radon measures $\{\muj\}_j$ converges to $\mu$ ($\muj \wto \mu$) if
\begin{equation*}
\lim_{j \to \infty}
\int f \dif \muj
=
\int f \dif \mu
\qquad
\forall f \in C_c(A,\R^m).
\end{equation*}
If $\mu \in \mep(\ru)$, $x \in \ru$, $\gk \in \N$, we define the upper and the lower $\gk$-densities of $\mu$ at $x$:
\begin{equation*}
\tum(\mu,x)
=
\limsup_{r \to 0} \frac{\mu\big(B_r(x)\big)}{r^m}
\qquad
\tlm(\mu,x)
=
\liminf_{r \to 0} \frac{\mu\big(B_r(x)\big)}{r^m}.
\end{equation*}
If the above limits coincide, then we define the $m$-density of $\mu$ at $x$ as their common value, which we denote by $\tm(\mu,x)$.
For any $s>0$, we denote by $\haus{s}$ the $s$-dimensional Hausdorff measure.
\subsection{Varifolds}
If $1 \leq m \leq \u$ we call $G(m,\u)$ the Grassmannian of the un-oriented $m$-dimensional linear subspaces (or $m$-planes) of $\ru$. If $A\subset \ru$ we denote by $G_m(A) := A \times G(m,\u)$ the trivial Grassmannian bundle over $A$.


An \emph{$m$-varifold} on $A$ is a positive Radon measure on $G_m(A)$. We denote by $\V_m(A)$ the set of all $m$-varifolds on $A$ and we endow $\V_m(A)$ with the topology of the weak*-convergence of Radon measures, i.e. we say that $V_j \wto V$ if
\begin{equation*}
\lim_{j \to \infty} \int_{G_m(A)} \varphi(x,S) \dif V_j(x,S)
=
\int_{G_m(A)} \varphi(x,S) \dif V(x,S)
\qquad
\forall \varphi \in C_c(G_m(A)).
\end{equation*}
An $m$-rectifiable measure $\mu= \theta \haus{m}\llcorner M$ in $\ru$ (where $M$ is $m$-rectifiable and $\theta \geq 0$ belongs to $L^1_{\loc}(M,\haus{m})$) induces the $m$-varifold
\begin{equation*}
V = \theta \haus{m}\llcorner M \otimes \delta_{T_xM},
\end{equation*}
where $T_xM$ is the approximate tangent space of $M$ at $x$. A varifold that is induced by a rectifiable measure is called a \emph{rectifiable varifold}.
If $V$ is an $m$-varifold, the \emph{mass} $\norm{V}$ (or \emph{total variation}) of $V$ is the positive Radon measure defined as
\begin{equation*}
\norm{V}(A) = V(G_m(A))
\qquad
\forall A \subset \om \mbox{ Borel}.
\end{equation*}
If $V$ is the $m$-varifold induced by the rectifiable measure $\mu= \theta \haus{m}\llcorner M$, then
\begin{equation*}
\norm{V}(B)
=
\int_B \theta(x) \dif \haus{m}(x).
\end{equation*}
By slight abuse of notation, we often denote $\supp \nv$ by $\supp V$.
If $V \in \V_m(\M)$ and if $\psi:\M \to \M$ is a diffeomorphism, the \emph{push forward} $\psi_{\sharp}V$ of $V$ through $\psi$ is the varifold in $\V_m(\M)$ such that, $\forall \varphi \in C_c\big(G_m(\M)\big)$,
\begin{equation}\label{eq:push-forward_varifolds}
\int_{G_m\left(\M\right)} \varphi(y,T) \dif \psi_{\sharp}V(y,T)
=
\int_{G_m(\M)} J_S\psi(x) \varphi\big(\psi(x),\dif \psi_x(S)\big) \dif V(x,S),
\end{equation}
where $J_S \psi(x)$ is the Jacobian of $\psi$ relative to the $m$-plane $S$, i.e.
\begin{equation*}
J_S \psi(x) = \sqrt{\det\big( (\dif \psi_x)_{|S}^* \circ \big( (\dif \psi_x)_{|S} \big)}.
\end{equation*}
We notice that this \emph{is not} the push forward of measures previous defined (which is denoted by the different symbol $f_\# \mu$). In fact, the push forward of varifolds is defined in this way in order to ensure the validity of the area formula: indeed if $V$ is induced by a rectifiable set $M$, then $\psi_\sharp V$ is induced by $\psi(M)$.
If $V \in \V_m(\R^n)$, we say that $C \in \V_m(\R^n)$ is a \emph{blow-up} of $V$ at $x$ or a \emph{tangent varifold} to $V$ at $x$ if there exists a sequence $r_j \downarrow 0$ such that
\begin{equation*}
(\dil{x,r_j})_\sharp V
\wto
C.
\end{equation*}
We write $\tang(V,x)$ for the set of tangent varifold to $V$ at the point $x$.

If $V \in \V_m(\M)$ and if $X \in \Xm$ the \emph{first variation} $\delta V(X)$ of $V$ with respect to $X$ is
\begin{equation*}
\delta V [X] = \left. \frac{\dif}{\dif t} \Big( \norm*{(\psi_t)_{\sharp} V}(\ru) \Big) \right|_{t=0}
\end{equation*}
where $\psi_t$ is the flow map of $X$ at the time $t$.
The following \emph{first variation formula} holds:
\begin{equation*}
\delta V [X] = \int_{G_k(\M)} \dive_S X(x) \, \dif V(x,S).
\end{equation*}
We now define the class of varifolds with \emph{bounded first variation}:
\begin{defi}
We say that a varifold $V \in \V_m(\M)$ has \emph{bounded first variation} in $\M$ if
\begin{equation}\label{eq:bounded_firsv_variation}
\sup \{|\delta V[X]| \mid X \in \Xm \,,
\max|X| \leq 1
\} < + \infty.
\end{equation}
If \eqref{eq:bounded_firsv_variation} holds with a proper subset of $\Xm$(e.g. $\Xc$, $\Xo$...) in place of $\Xm$, we say that $V$ has bounded first variation with respect to this subset.
\end{defi}
Therefore $V$ has bounded first variation if $\delta V \in \me(\M,\ru)$.
If $V$ has bounded first variation, then by Lebesgue decomposition there exist $\nds{V} \in \mep(\M)$, a $\nds{V}$-measurable function $\eta: \M \to \R^3$ and a $\norm{V}$-measurable function $H: \M \to \R^3$ such that
\begin{equation*}
\delta V[X]= - \int_\M {H}\cdot{X} \, \dif \norm{V} +
\int_\M {X}\cdot{\eta} \dif \nds{V}
\qquad
\forall X \in \Xm
\end{equation*}
where $\nds{V}$ is the singular part of $|\delta V|$ with respect to $\norm{V}$:
\begin{equation*}
\nds{V} = |\delta V|\llcorner Z
\qquad
Z=
\Big\{
x \in \M \mid \limsup_{r \to 0} \frac{|\delta V|(B_r(x))}{\norm{V}(B_r(x))} = +\infty
\Big\}.
\end{equation*}
Since the previous formula is similar to the corresponding one for smooth surfaces, we call $H$ the \emph{generalized mean curvature} of $V$, $\nds{V}$ the \emph{boundary measure} of $V$, the set $Z$ is the \emph{boundary} of $V$ and $\eta$ is the \emph{unit co-normal} of $V$.

\begin{defi}\label{def:varifold_generalized_curvature}
We say that $V \in \V_m(\M)$ has \emph{generalized mean curvature} $H$ with respect to $\Xc$ (respectively $\Xo$, $\Xt$) if for each $X \in \Xc$ (respectively $\Xo$, $\Xt$) the following formula holds:
\begin{equation}\label{eq:generalized_mean_curvature}
\int_{G_k(\M)} \dive_S X(x) \, \dif V(x,S)
=
- \int_\M {H}\cdot{X} \, \dif \norm{V}.
\end{equation}
\end{defi}
Thus $V$ has \emph{generalized mean curvature} with respect to $\Xc$ (respectively $\Xo$, $\Xt$) if has bounded variation with respect to $\Xc$ (respectively $\Xo$, $\Xt$) and $\delta V$ has no singular part with respect to $\nv$ when we test with vector fields in $\Xc$ (respectively $\Xo$, $\Xt$).
\begin{defi}[Varifold with free boundary]\label{def:varifold_free_boundary}
If $V \in \V_k(\M)$ has generalized mean curvature with respect to $\Xt$, we say that $V$ \emph{has free boundary} at $\pM$.
\end{defi}
\subsection{Capillarity free energy}
\begin{defi}[Capillarity free energy]
If $\om \subset \M$ is relatively open with finite perimeter and if $a \in [0,1)$, the \emph{$a$-capillarity energy} $\fa$ applied to $\om$ is defined as
\begin{equation*}
\fa(\om) = \hn(\bi{\om}) + a \hn(\bb \om),
\end{equation*}
\end{defi}

We now define $\opm$ as the class of all subsets of $\M$ which are open in the relative topology of $\M$ and whose internal and external boundary are $2$-dimensional submanifolds of class $C^2$ with $C^1$ boundary, that is
\begin{equation*}
\opm
=
\{
\om \subset \M \mid
\om \text{ is open in } \M, \,
\bi(\om) \text{ and } \bb(\om) \text{ are of class } C^2 \text{ with } C^1 \text{ boundary } \dbo
\},
\end{equation*}
where we recall that $\bio$ and $\bbo$ are respectively the internal and the external boundary of $\om$ previous defined and $\dbo$ denotes the common boundary of $\bio$ and $\bbo$ as manifolds.

If $\om \in \opm$, standard computations yield the first variation of $\fa$ with respect to $X \in \Xt$:
\begin{equation*}\label{eq:first_variation_formula_cap}
\dfa \om [X]
=
- \int_{\bio} {X}\cdot{H} \dif \hn
+ \int_{\dbo} {\con + a \z}\cdot{X} \dif \hu,
\end{equation*}
where $H$ is the mean curvature vector of $\bi \om$, $\con$ is the exterior unit conormal vector to $\bi \om$ and $\z$ is the exterior unit conormal vector to $\bb \om$.

In the following we need to extend our definition of capillarity functional to the set of pairs of varifolds $\cu$.
\begin{defi}[Capillarity functional for pairs of varifolds]
If $\pair \in \cu$ and if $a \in [0,1)$, the \emph{$a$-capillarity functional} $\fa$ applied to $\pair$ is defined as
\begin{equation*}
\fa\pair = \nv(\M) + a \nw(\pM).
\end{equation*}
\end{defi}

\section{Contact angle condition for varifolds}\label{sec3}
We now  discuss a notion of  contact angle condition for varifolds. Our condition follows the one given by Kagaya and Tonegawa in  \cite{kagaya2017}, but it is weaker since it allows the ``boundary part" to be just a $2$-varifold on $\pM$ (that is a positive Radon measure on $\pM$), whereas the one in \cite{kagaya2017} is more restrictive, since it fixes the boundary part to be a subset of $\pM$. We state our condition for pairs $(V,W)$ where $V \in \V_2(\M)$ and $W \in \V_2(\pM)$.

\begin{defi}[Contact angle condition]\label{def:contact_angle_condition}
Given $\theta \in [0,\frac{\pi}{2})$ we say that the pair $(V,W) \in \cu$ satisfies the contact angle condition $\theta$ if there exists a $\nv$-measurable vector field $H \in L^1(\M,\nv)$ such that, for every $X \in \Xt$, it holds
\begin{equation}\label{eq:def_contact_angle}
\begin{split}
\dfa \pair [X]
& =
\int_{\gnm} \dive_S X(x) \dif V(x,S)
+ a \int_{\gnpm} \dive_{\pM} X(x) \dif W(x,T_x \pM)
\\
& =
- \int_{\M} {X(x)}\cdot{H(x)} \dif \nv(x),
\end{split}
\end{equation}
where $a = \cos \theta$. Since we test only with vector fields in $\Xt$, we assume that $H(x) \in T_x \pM$ for $\nv$-a.e. $x \in \pM$.
\begin{defi}[Stationary pairs]
If $\theta \in [0,\frac{\pi}{2})$ and $a = \cos \theta$, we say that the pair $\pair \in \cu$ is \emph{stationary} for $\fa$ if it satisfies the contact angle condition \eqref{eq:def_contact_angle} with $H=0$. If $U \subset \M$ is relatively open in $\M$, we say that $\pair$ is stationary in $U$ if it satisfies \eqref{eq:def_contact_angle} with $H=0$ for every vector field $X \in \X_t$ that is compactly supported in $U$.
\end{defi}

We also give the definition of stability for pairs with respect to $\fa$.
\begin{defi}[Stable pairs]
If $\theta \in [0,\frac{\pi}{2})$ and $a = \cos \theta$, we say that the pair $\pair \in \cu$ is \emph{stable} for $\fa$ if it satisfies
\begin{equation}\label{eq:def_stability_fa}
\dtfa \pair [X] \geq 0
\qquad
\forall X \in \Xt,
\end{equation}
where $\dtfa$ is the second variation of $\fa$. If $U \subset \M$ is relatively open in $\M$, we say that $\pair$ is stable in $U$ if it satisfies \eqref{eq:def_stability_fa} for every vector field $X \in \X_t$ that is compactly supported in $U$.
\end{defi}

\end{defi}
\begin{rem}
To understand the intuition behind the previous definitions, we  recall that that a $C^1$-surface $\s \subset \M$ with $C^1$-boundary $\partial \s \subset \pM$ has fixed contact angle $\theta$ at $\pM$ if its unit exterior conormal $\eta_\s$ satisfies ${\con(x)}\cdot{N(x)} \equiv \sin \theta$ for every $x \in \partial \s$ (where $N$ is the exterior unit normal vector to $\pM$).

If $\om \in \opm$, then it easily seen that the surface $\bio \subset \M$ has fixed contact angle $\theta$ if and only if either $\con + a\z \perp \pM$ or $\con - a\z \perp \pM$, where $a = \cos \theta$, and $\eta$ and $\z$ are the unit conormal respectively of $\bi \om$ and $\bb \om$; thus $\bio$ has fixed contact angle $\theta$ if and only if
\begin{equation}
\int_{\bio} \dive_{\bio} X \dif \hn
\pm a \int_{\bbo} \dive_{\pM} X \dif \hn
=
- \int_{\bio} {X}\cdot{H} \dif \hn,
\end{equation}
where $\pm$ must be constantly $+$ or $-$ and $H$ is the mean curvature of $\bio$. Moreover, if $H \equiv 0$, then $\bi \om$ is a minimal surface which meets $\pM$ with angle $\theta$.
\end{rem}

\subsection{Bounded first variation and Monotonicity formula}\label{subsec:monotonicity}
All the results of this section are proved in Appendix \ref{sec:appendix_monotonicity}. We first state that, if $\pair \in \cu$ has contact angle $\theta$, then $V+aW$ has bounded first variation in $\M$.

\begin{proposition}\label{prop:bv_angle_varifolds}
Let $\pair \in \cu$ have contact angle $\theta$. Then $V + aW$ has bounded first variation. More precisely, there exist a positive Radon measure $\muv$ on $\pM$ and a continuous vector field $\tilde{H}$ such that
\begin{equation}\label{eq:total_fvf_varifold_angle}
\begin{split}
& \int_{G_2(\M)} \dive_S X(x) \dif \big( V(x,S) + aW(x,S) \big)
=
- \int_\M {X}\cdot{H} \dif \nv
\\
& \qquad
- \int_\M {X}\cdot{\tilde{H}} \dif (\nv + a \nw)
+ \int_{\pM} {X}\cdot{N} \dif \muv
\qquad
\forall X \in \Xm,
\end{split}
\end{equation}
where $\tilde{H}(x)= -N(x) \dive_{\pM}N(x)$ for every $x \in \pM$. Moreover, the following global and local estimates hold:
\begin{equation}\label{eq:estimate_muv_global}
\muv (\pM)
\leq
c \nv(\M) + \int_\M |H| \dif \nv;
\end{equation}
\begin{equation}\label{eq:estimate_measure_boundary}
\muv \big(B_{r/2}(x_0)\big)
\leq
\frac{c}{r} \nv \big( B_r(x_0) \big)
+
\int_{B_r(x_0)} |H| \dif \nv
\qquad
\forall x_0 \in \pM,
\,
\forall r \leq R(\M),
\end{equation}
where the constant $c=c(\M)$ depends only on the second fundamental form of $\pM$ and $R(\M)$ is the minimum radius of curvature of $\pM$.
\end{proposition}

As in \cite{kagaya2017}, a monotonicity formula at points on $\pM$ holds for a pair $(V,W)$ that satisfies the contact angle condition. The following monotonicity is slightly different from the one in \cite{kagaya2017}, since it does not involve reflections over $\pM$.

\begin{proposition}\label{prop:monotonicity_formula_p}
Suppose $(V,W) \in \cu$ satisfies the contact angle condition with $H \in L^p(\M,\nv)$ for some $p \in (2,+\infty)$. Then there exists $\Lambda = \Lambda\big(p,\M, (\nv+a\nw)(\M), \norm{H}_{L^p}\big)>0$ such that, for all $x_0 \in \pM$ the function

\begin{equation}\label{eq:monotonicity_mass_p}
\rho
\longmapsto
e^{\Lambda \rho}
\left[
\bigg(
\frac{\nv\big( B_\rho(x_0)\big) + a\nw \big( B_\rho(x_0)\big)}{\rho^2}
\bigg)^{\frac{1}{p}}
+ \Lambda \rho^{1- \frac{2}{p}}
\right]
\end{equation}
is monotone increasing.
\end{proposition}
The proof is a straightforward modification of the one of \cite[Corollary 4.8]{demasi2021rectifiability} and is reported in Appendix \ref{sec:appendix_monotonicity}.
If $H \in L^\infty(\M,\nv)$,  we have the following result.
\begin{proposition}\label{prop:monotonicity_formula_infty}
Suppose $(V,W) \in \cu$ satisfies the contact angle condition with $H \in L^\infty(\M,\nv)$. Then there exists $\Lambda = \Lambda (\M,\norm{H}_{L^\infty})>0$ such that, for all $x_0 \in \pM$ the function
\begin{equation}\label{eq:monotonicity_mass_infty}
\rho
\longmapsto
e^{\Lambda \rho}
\frac{\nv\big( B_\rho(x)\big) + a\nw \big( B_\rho(x_0)\big)}{\rho^2}
\end{equation}
is monotone increasing.
\end{proposition}
As above, the proof is very similar to the one of \cite[Corollary 4.9]{demasi2021rectifiability} and is sketched in Appendix \ref{sec:appendix_monotonicity}.

\section{Existence of non-trivial stationary pairs}\label{sec:4}
\subsection{Sweepouts}\label{subsec:sweepouts}

In this section we define the basic elements of the min-max procedure.

\begin{defi}\label{def:paramfamily}
A family $\swo$ is a \emph{parametrized smooth family of surfaces} if it satisfies the following properties:
\begin{enumerate}
\item For every $t \in [0,1]$ $\s_t$ is a closed subset of $\M$ of finite $\hn$-measure;
\item \label{pnt:smoothhypersurface} For every $t \in [0,1]$ there exists a finite set $S_t$ such that $\s_t \setminus S_t$ is a smooth surface with boundary $\partial \s_t \subset \partial \M$ which is smooth outside of $S_t$
\item \label{pnt:hausdorffdistance}if $t \to s$, then
$\sup_{x \in \s_t} \dist(x, \s_s) \to 0$;
\item \label{pnt:smoothcompact}if $U \cc \M \setminus S_s$, then $\s_t$ converges in the smooth sense to $\s_s$ in $U$ as $t \to s$;
\item the function $t \mapsto \hn(\s_t)$ is continuous.
\end{enumerate}
\end{defi}

A first elementary consequence of the definition of parametrized family is that the area of the surfaces converge also locally. This is useful in the following, then we record it in the following lemma.
\begin{lemma}\label{lmm:convergence_area_any_open}
Let $\{\s_t\}_{t \in \I}$ be a smooth parametrized family of surfaces. Then, for every open set $U \subset \M$, it holds
\begin{equation}
\lim_{t \to s} \hn(\s_t \cap U)
=
\hn(\s_s \cap U).
\end{equation}
\begin{proof}
For every $\e >0$, there exists $W \cc \M \- S_s$ such that $\hn(\s_s \cap W) > \hn(\s_s) - \e$. By smooth convergence in $W$, we have that
\begin{equation}
\lim_{t \to s} \hn(\s_t \cap U \cap W)
=
\hn(\s_s \cap U \cap W),
\qquad
\lim_{t \to s} \hn(\s_t \cap (W \- U))
=
\hn(\s_s \cap (W \- U))
\end{equation}
Since $\hn(\s_t) \to \hn(\s_s)$ by the choice of $W$ we have also $|\hn(\s_t \cap U \- W) - \hn(\s_s \cap U \- W)| < \e$ for $|t-s|$ sufficiently small. This implies that $|\hn(\s_t \cap U) - \hn(\s_s \cap U)| < \e$ for $|t-s|$ sufficiently small.
\end{proof}

\end{lemma}

We often refer to a parametrized smooth family of surfaces simply as a family of surfaces.

\begin{defi}[Sweepout]\label{def:sweepout}
A pair $\cswo$ of parametrized family of surfaces  with (finite) singular sets $S_t$ and $G_t$ is called a \emph{sweepout} if there exists a family of open sets $\oswo$ in the relative topology of $\M$ that satisfy the following properties:
\begin{enumerate}
\item\label{item:sw_contained} for every $t \in \I$ $\s_t \subset \M$ and $\g_t \subset \pM$;
\item\label{item:sw_boundary} for every $t \in \I$ $\s_t \cap \g_t = \de \s_t = \de \g_t$ (here the boundary is intended in the sense of differential topology);
\item\label{item:sw_internal} for every $t \in [0,1]$, $\biot \- \s_t \subset S_t$;
\item\label{item:sw_coincide} for every $t \in \I$, $\big(\bbot \- (\s_t \cup \g_t) \big) \cup \big( (\s_t \cap \pM) \- \bbot \big) \cup \big( \g_t \- \bbot \big) \subset S_t \cup G_t$;
\item\label{item:sw_open_fill} $\om_0=\emptyset$ and $\om_1= \M$;
\item\label{item:sw_vol_continuous} if $t \to s$, then $\vol(\om_t \setminus \om_s)+\vol (\om_s \setminus \om_t) \to 0$.
\end{enumerate}
\end{defi}
As it is stated in \cite[Proposition 0.4]{delellis2009existence} sweepouts always exist.

\begin{defi}[Pair compatible with open set]\label{def:compatible_pair_open}
A pair of surfaces $\pas$  and an open set $\om$ are said to be \emph{compatible} if the properties \ref{item:sw_contained}, \ref{item:sw_boundary}, \ref{item:sw_internal},  \ref{item:sw_coincide} and \ref{item:sw_vol_continuous} of Definition \ref{def:sweepout} holds with $\pas$, $\om$, $S$, $G$ in place of $\cswo$, $\oswo$, $S_t$, $G_t$.
\end{defi}

\begin{defi}
Two sweepouts $\cswo$ and $(\tilde{\s}_t, \tilde{\g}_t)_{t \in \I}$ are \emph{homotopic} if there exists a parametrized family of surfaces $(\s_{(t,s)},\g_{(t,s)})_{(t,s) \in [0,1]^2}$ such that $\s_{(t,0)} = \s_t$, $\g_{(t,0)}= \g_t$, $\s_{(t,1)}= \tilde{\s}_t$ and $\g_{(t,1)}= \tilde{\g}_t$ for every $t \in [0,1]$.
\end{defi}

\begin{defi}
A non-empty set of sweepouts $\sw$ is \emph{homotopically closed} if it contains the homotopy class of all its elements.
\end{defi}
From now on, with $\sw$ we always denote a homotopically closed set of sweepouts.


\begin{defi}
If $\sw$ is an homotopically closed set of sweepouts, the \emph{min-max value of $\sw$} is
\begin{equation}
\m(\sw) = \inf \{\ms \fa \ct{t} \mid \cswo \in \sw \}
\end{equation}
\end{defi}

\begin{defi}
A sequence $\{(\s_t^i,\g_t^i)_{t \in \I}\}^{i \in \N}$ of sweepouts in $\sw$ is a \emph{minimizing sequence} if
\begin{equation}
\lim_{i \to \infty} \ms \fa(\s_t^i,\g_t^i) = \m(\sw).
\end{equation}
If $(\s_t^i,\g_t^i)_{t \in \I}^{i \in \N}$ is a minimizing sequence of sweepouts, a sequence $(\s^i,\g^i):=(\s_{t_i}^i, \g_{t_i}^i)$ (where $t_i \in \I$ for every $i \in \N$) is a \emph{min-max sequence} if
\begin{equation}
\lim_{i \to \infty} \fa \big( \s^i, \g^i\big)
=
\m(\sw),
\end{equation}
\end{defi}

\subsection{Pull-tight procedure} We now want to prove that there exists a minimizing sequence of sweepouts such that every min-max sequence converges to a pair which is stationary for the capillarity functional. From now on, we consider fixed the homotopically closed set of sweepouts $\sw$ and we call $\m=\m(\sw)$.

\begin{theorem}[Pull-tight procedure]\label{thm:pull-tight_procedure}
There exists a minimizing sequence of sweepouts $\{\cswok\}^{k \in \N}$ such that every min-max sequence converges, up to subsequences, to a pair $(V,W) \in \cu$ which is stationary for the capillarity functional.
\end{theorem}

\begin{proof}
We follow the construction in \cite{colding2003min}.
Let us fix a minimizing sequence of sweepouts $\{\cswok\}^{k \in \N}$.
The proof is based on the construction of a deformation of the minimizing sequence; this deformation reduce the mass of the surfaces of a quantity that depends on how much the pair is ``far" from being stationary. This deformed minimizing sequence has the desired property.
We call
\begin{equation}
\am = \{(V,W) \in \cu \mid \nv(\M) \leq 2\m, \nw(\pM)\leq \hn(\pM) \}
\end{equation}
and
\begin{equation}
\ami
=
\{(V,W) \in \am \mid (V,W) \text{ is stationary for } \fa \}.
\end{equation}
Since the weak*-topology is metrizable on bounded sets, we choose a distance which induces it on $\{V \in \V_2(\M) \mid \nv(\M)\leq 2\m\}$ and $\{W \in \V_2(\pM) \mid \nw(\pM) \leq \hn(\pM)\}$ and we call $\D$ the distance on the product space $\am$ which induces the product topology. Up to scaling $\D$, we can assume that $\diam \am = 1$. Moreover $\am$ is compact since  it is a product of compact spaces. Since the property of being stationary is preserved under weak*-convergence of varifolds, it follows that $\ami$ is also a compact subset of $\am$.  We call $\B_r\pair$ the ball of radius $r$ and center $\pair$ with respect the distance $\D$.

For every $i \in \N$, we call
\begin{equation}
A_i
=
\Big\{
(V,W) \in \am \mid \frac{1}{2^{i+1}} \leq \D\big( (V,W), \ami \big) \leq \frac{1}{2^i}
\Big\}.
\end{equation}

Since every $A_i$ is closed, it is compact.
\begin{steps}[wide,%
labelindent=5pt]
\item
We claim that for every $i \in \N$ there exists a constant $c_i>0$ such that, for every $\pair \in A_i$, there exists $X _{\pair}\in \Xt$ such that $|X_{\pair}(x)|\leq 1$ for every $x \in \M$ and
\begin{equation}\label{eq:first_variation_negative_pull-tight}
\begin{split}
\dfa \pair [X_{\pair}]
= &
\int_{\gnm} \dive_S X_{\pair}(x) \dif V(x,S)
+ a \int_{\gnpm} \dive_{\pM} X_{\pair}(x) \dif W(x,T_x \pM)
\\
\leq &
- c_i.
\end{split}
\end{equation}
Indeed, suppose by contradiction that the claim is false for some $i \in \N$; thus there exists a sequence of pairs $(V_j,W_j)_j$ in  $A_i$ such that for every $j$ and for every $X \in \Xt$ such that $|X| \leq 1$ we have
\begin{equation}
\Big|
\int_{\gnm} \dive_S X(x) \dif V_j(x,S)
+ a \int_{\gnpm} \dive_{\pM} X(x) \dif W_j(x,T_x \pM)
\Big|
\leq\frac{1}{j}.
\end{equation}
By compactness of $A_i$, $(V_j,W_j)$ converges, up to subsequences, to $\pair \in A_i$. By the previous equation, it follows that $\pair$ is stationary for $\fa$, which contradicts $\pair \in A_i$. This proves the claim.

\item
The map $\pair \in \am \- \ami \mapsto X_{\pair} \in \Xt$ is not necessarily continuous. We want now to use it to construct a continuous one with respect the $C^1$-topology in $\Xt$ that satisfies an inequality on the first variation of $\fa$ similar to \eqref{eq:first_variation_negative_pull-tight} in a neighborhood of $\pair$.

We first observe that, by continuity of the first variation under varifold convergence, for every $i \in \N$ and every $\pair \in A_i$ there exists $r_{\pair} \in (0,\frac{1}{2^{i+2}})$ such that
\begin{equation}\label{eq:fv_remains_negative}
\int_{\gnm} \dive_S X_{\pair}(x) \dif\tv(x,S)
+ a \int_{\gnpm} \dive_{\pM} X_{\pair}(x) \dif \tw(x,T_x \pM)
\leq
- \frac{c_i}{2}
\end{equation}
whenever $\D \big(\pair, \tpair \big) \leq r_{\pair}$.
Thus the family
\begin{equation}
\F_i
=
\{
\B_{r_{\pair}} \pair \mid \pair \in A_i
\}
\end{equation}
(we recall that $\B_r \pair$ is the ball of center $\pair$ and radius $r$ with respect the distance $\D$)
is an open cover of $A_i$ and by compactness there exists a finite number $J_i$ of pairs $\{\pairij\}^{j=1,\dots,J_i}$ such that
\begin{equation}
A_i
\subset
\bigcup_{j = 1}^{J_i}
\B_{\frac{\rij}{2}}\pairij
\end{equation}
(where we call $\rij := r_{\pairij}$). For simplicity of notation, we call $\bij= \B_{\rij}\pairij$ and $\tbij=\B_{\frac{\rij}{2}}\pairij$.
We underline that by the choice of $r_{\pair}$ it follows that
\begin{equation}
A_h \cap \bij
= \emptyset
\qquad
\text{whenever } |i-h| \geq 2.
\end{equation}
By the assumptions, the family
\begin{equation}
\F
=
\{
\tbij \mid i \in \N, j=1, \dots, J_i
\}
\end{equation}
is a locally finite covering of $\am \- \ami$. We can construct a partition of the unity $\psiij$ with respect to this covering and define
\begin{equation}
\ypair
=
\D\big(\pair, \ami \big)
\sum_{\substack{i \in \N \\ j=1,\dots,J_i}} \psiij\pair X_{\pairij}
\qquad
\forall
\pair \in \am.
\end{equation}
The map $\pair \mapsto \ypair$ is continuous with respect the $C^1$ topology of $\Xt$ and $\ypair = 0$ if $\pair \in \ami$.
If we define for every $\pair \in \am$
\begin{equation}
\rho\pair
=
\min \Big\{
\frac{\rij}{2} \mid \pair \in \tbij
\Big\},
\end{equation}
by definition we have $\rho\pair \leq \frac{1}{2^{i+3}}$ whenever $\pair \in A_i$. Hence, if $\pair \in A_h$, for every $\tpair \in \B_{\rho\pair}\pair$ we have that
\begin{equation}
\begin{split}
\dfa \tpair (\ypair)
& =
\D\big(\tpair, \ami \big)
\sum_{\substack{|i-h| \leq 1 \\ j=1,\dots,J_i}} \psiij\tpair \dfa \tpair (X_{\pairij})
\\
& \leq
- \frac{1}{2^{h+2}} \min\{c_{h-1},c_h, c_{h+1}\},
\end{split}
\end{equation}
by \eqref{eq:fv_remains_negative}, since $\B_{\rho\pair}\pair \subset \bij$ for every $i,j$ such that $\pair \in \tbij$.

By compactness again, there exist two continuous functions $d \colon [0,1] \to (0,1]$ and $f \colon (0,1] \to (0,+\infty)$ such that
\begin{gather}
d(s) \leq 2^{-(i+3)} 	\label{eq:definition_distance_deformation_pulltight}
\qquad
\forall s \in [2^{-(i+1)},2^{-i}],
\\
\label{eq:deformation_pulltight_estimates_first_variation}
\dfa \tpair (\ypair)
\leq
- f \big( \D(\pair,\ami) \big)
\quad
\text{whenever}
\quad
\D\big( \pair,\tpair \big) < d\big( \D(\pair, \ami) \big).
\end{gather}
\item We now want to use the vector field-valued map constructed to deform the minimizing sequence of sweepouts $\cswok^{k \in \N}$ in a new minimizing sequence $\ctswok^{k \in \N}$ in $\sw$.

If $\pair \in \am$, we define $\pairt = \big( (\psi_{\pair , t})_\sharp V, (\psi_{\pair , t})_\sharp W \big)$, where $\psi_{\pair, t}$ is the flow map of $\ypair$ at the time $t$. Thus \eqref{eq:deformation_pulltight_estimates_first_variation} yields the existence of a maximum $t_0\pair>0$ such that, for any $t \in [0,t_0\pair]$ we have
\begin{equation}
\dfa \pairt (\ypair) \leq - f \big( \D(\pair,\ami) \big)
\quad
\text{and}
\quad
\D(\pairt,\pair) \leq d\big(\D(\pair,\ami) \big).
\end{equation}
Again by compactness of the sets $A_i$
there exists a continuous function $T \colon (0,1] \to (0,1]$ such that, for every $\pair \in \am$ and for every $t \in [0,T\big(\D(\pair,\ami) \big)]$, it holds
\begin{gather}
\dfa \pairt (\ypair) \leq - f \big( \D(\pair,\ami) \big),
\\
\D(\pairt,\pair)
\leq
d\big(\D(\pair, \ami) \big)
\stackrel{\eqref{eq:definition_distance_deformation_pulltight}}{<}
\frac{\D(\pair, \ami)}{4}.
\end{gather}
Hence, setting $\beta=\D(\pair,\ami)$
\begin{equation}
\fa (V_{T(\beta)},W_{T(\beta)}) - \fa \pair
\leq
\int_0^{T(\beta)} \dfa \pairt (\ypair) \dif t
\leq
- T(\beta) f \big( \D(\pair,\ami) \big).
\end{equation}
We can re-normalize the flows $\psi_{\pair, t}$ by defining $\Psi_{\pair, t}(x) = \psi_{\pair, T(\beta)t}$, where again $\beta= \D(\pair,\ami)$.  Thus, re-defining $\pairt = \big( (\Psi_{\pair , t})_\sharp V, \Psi_{\pair , t})_\sharp W \big)$ we have that there exists $g \colon (0,1] \to (0,+\infty)$ such that $\lim_{s \to 0}g(s)=0$ and
\begin{equation}
\fa (V_1,W_1) - \fa \pair
\leq
- g \big( \D(\pair,\ami) \big),
\qquad
\D\big( (V_1,W_1), \pair \big)
<
\frac{\D(\pair, \ami)}{4}.
\end{equation}
Now we choose a minimizing sequence of sweepouts $\{\cswok\}^{k \in \N}$ in $\sw$ such that
\begin{equation}
\ms \fa \cswok \leq \m(\sw) + \frac{1}{k}
\qquad
\forall k \in \N.
\end{equation}
Let us fix $k \in \N$. We would like to deform every $\cswok$ by the flow $\Psi_{\cswok, 1}$, but since the map
\begin{equation}
Y_k \colon t \in [0,1] \mapsto Y_{\cswok} \in \Xt
\end{equation}
is only continuous, the family of pairs
$t \mapsto (\Psi_{\cswok, 1})_\sharp \cswok$ may not belong to $\sw$.
Therefore we have to smooth the map $Y_k$. To this aim let us consider a smooth map $Z_k \colon [0,1] \to \Xt$. If $\sup_{t \in \I}\norm{Y_k - Z_k}_{C^1}$ is sufficiently small and if $Z_k(t)=0$ when $Y_k(t)=0$, then for every $t \in \I$ we have
\begin{equation}\label{eq:deforming_sweepout_pull-tight}
\fa \ctswok - \fa \cswok
\leq
- \frac{g\big( \D(\cswok,\ami) \big)}{2},
\qquad
\D\big( \ctswok, \cswok \big)
\leq
\frac{\D\big(\cswok, \ami\big)}{4}.
\end{equation}
where $\ctswok = \Phi_1(t)_\sharp \cswok$ and $\Phi_1(t)$ is the flow map of $Z_k(t)$ evaluated at the time $1$. Hence $\ctswok_{t \in I} \in \sw$.
By repeating the same procedure for every $k \in \N$ we obtain a new sequence of sweepouts $\ctswok^{k \in \N}$. By \eqref{eq:deforming_sweepout_pull-tight} we have that this sequence is again a minimizing sequence.
\item
We claim that $\ctswok^{k \in \N}$ has the desired property, that is, every min-max sequence converges to a stationary pair for $\fa$. To prove this, suppose by contradiction that the claim is false and let $(\ts{t_k}^k,\tg{t_k}^k)^{k \in \N}$ a min-max sequence such that
\begin{equation}
\limsup_{k \to \infty} \D((\ts{t_k}^k,\tg{t_k}^k), \ami) =\l >0.
\end{equation}
Up to passing to a subsequence we can assume that the above upper limit is in fact a limit. By the second relation in \eqref{eq:deforming_sweepout_pull-tight}, we have that, for $k$ sufficiently large,
\begin{equation}
\D\big( (\s_{t_k}^k, \g_{t_k}^k), \ami\big)
\ge
\frac{\l}{2}.
\end{equation}
By the first inequality in \eqref{eq:deforming_sweepout_pull-tight} we get
\begin{equation}
\fa (\s_{t_k}^k, \g_{t_k}^k)
\ge
\fa (\ts{t_k}^k,\tg{t_k}^k) +
\frac{1}{2} g\Big( \frac{\l}{2}\Big)
\xrightarrow{k \to \infty}
\m + \frac{1}{2} g\Big( \frac{\l}{2}\Big)
>
\m.
\end{equation}
Thus
\begin{equation}
\liminf_{k \to \infty} \ms \fa\cswok
\ge
\liminf_{k \to \infty} \fa (\s_{t_k}^k, \g_{t_k}^k)
\geq
\m + \frac{1}{2} g\Big( \frac{\l}{2}\Big)
\end{equation}
which contradicts the fact that the sequence $\cswok$ was minimizing in $\sw$.
\end{steps}
\end{proof}

In order to prove that the limit pair of every min-max sequence is non-trivial we need the following proposition.
\begin{proposition}\label{prop:lower_bound_min-max_value}
If $\sw$ is an homotopically closed set of sweepouts, then $a \hn(\pM) < \m(\sw)$.
\end{proposition}
\begin{proof}
We want to show that there exists $\eta>0$  such that, for every sweepout $\cswo \in \sw$, it holds $\ms \fa \csw \geq a \hn(\pM) + \eta$.

To do so, we are going to prove that there exists $\e>0$ and a differentiable function $f \colon [0,\e) \to \R^+$ such that, for every sweepout $\csw$, $f(0) = 0$, $f'(s)>0$ for $s \in (0,\e)$ and
\begin{equation}
\fa \csw
\geq
a\hn(\pM) + f(|\M \- \om_t|)
\qquad
\forall t \in \I \text{ s.t. } |\M \- \om_t| \leq \e.
\end{equation}
This clearly implies the result. To prove the existence of such $f$, we write $\fa$ in a different way, and next we use the relative isoperimetric inequality to obtain a lower bound on $\fa \csw$.

%

We fix a vector field $X \in \Xm$ such that $|X(x)| \leq 1$ for all $x \in \M$ and $X(x)=N(x)$ (the exterior unit normal to $\M$) for all $x \in \pM$; the $C^1$-norm of $X$ depends only on $\M$. Using $X$ we can write $\fa \csw$ as
\begin{equation}\label{eq:rewriting_Fa}
\begin{split}
\fa \csw
= &
\hn(\st) + a \hn(\gt)
\\
= &
\hn(\st) + a \int_{\st \cup \gt} {X}\cdot{\nu_\s} \dif \hn
- a \int_{\st} {X}\cdot{\nu_\s} \dif \hn
\\
= &
\int_{\st} (1 - a {X}\cdot{\nu_\s}) \dif \hn
+ a \int_{\om_t} \dive X(x) \dif x.
\end{split}
\end{equation}
Since up to sets of $\hn$-measure $0$ it holds $\st \supset \de \om_t \cap \interior{\M}$, by the relative isoperimetric inequality there exists a constant $c(\M)>0$ such that
\begin{equation}
\hn(\st)
\geq
\hn(\de \om_t \cap \interior{\M})
\geq
c(\M) \min \{|\om_t|, |\M \- \om_t| \}^{\frac{2}{3}}.
\end{equation}
Thus, by \eqref{eq:rewriting_Fa} and $1-aX \cdot \nu_\s \geq 1-a$, for every $t \in \I$ we have
\begin{equation}
\begin{split}
\fa \csw
\geq &
c(\M)(1-a) \min \{|\om_t|, |\M \- \om_t| \}^{\frac{2}{3}}
+
a \int_{\M} \dive X(x) \dif x
- a \int_{\M \- \om_t} \dive X(x) \dif x
\\
= &
c(\M)(1-a) \min \{|\om_t|, |\M \- \om_t| \}^{\frac{2}{3}}
+ a \hn(\pM)
- a \int_{\M \- \om_t} \dive X(x) \dif x.
\end{split}
\end{equation}
If $\e >0$ is sufficently small, for all $t \in \I$ such that $|\M \- \om_t| \leq \e$, the previous inequality becomes
\begin{equation}
\begin{split}
\fa \csw
\geq &
a \hn(\pM) +
c_1 |\M \- \om_t|^{\frac{2}{3}} - c_2 |\M \- \om_t|
\\
= &
a \hn(\pM) + f(|\M \- \om_t|),
\end{split}
\end{equation}
where $c_1, c_2$ are positive constants depending on $\M, a$. Clearly  $f \in C^1\big((0,\e)\big)$, $f(0)=0$ and, up to choosing $\e > 0$ even smaller, we have and $f'(s) > 0$ for all $s \in (0,\e)$.
\end{proof}

\section{Existence of almost minimizing min-max sequences} \label{sec5}
We now introduce the notion of \emph{almost minimizing} pair $\pas$.

\begin{defi}\label{def:almost_minimizing}
Let $\e>0$, $U \subset \M$ be open and let $\pas$  be a pair  of surfaces that is compatible with an open set $\om \subset \M$ (in the sense of Definition \ref{def:compatible_pair_open}). We say that $\pas$ is \emph{$\e$-almost minimizing} in $U$ if there exists no parametrized pair of surfaces $\ft{t}_{t \in \I}$ and no family of open sets $\oswo$ such that:
\begin{enumerate}
\item $\ft{t}_{t \in \I}$ and $\oswo$
satisfy properties \ref{item:sw_contained}, \ref{item:sw_boundary}, \ref{item:sw_internal},  \ref{item:sw_coincide} and \ref{item:sw_vol_continuous} of Definition \ref{def:sweepout};
\item\label{item:am_contained_U} $\om_0= \om$ and $\om_t \- U = \om \- U$ for every $t \in \I$;
\item
$\Phi_0= \s$ and $\Phi_t \- U = \s \- U$ for every $t \in \I$;
\item
$\Psi_0 = \g$ and $\Psi_t \- U = \g \- U$ for every $t \in \I$;
\item $\fa \ft{t} \leq \fa \pas + \frac{\e}{8}$ for every $t \in \I$;
\item $\fa \ft{1} \leq \fa \pas - \e$.
\end{enumerate}
A sequence of pairs $\{\ck{k}\}^k$ is \emph{almost minimizing} in $U$ if there exists a sequence $\e_k \downarrow 0$ such that each $\ck{k}$ is $\e_k$-almost minimizing in $U$.
\end{defi}
Roughly speaking, a pair is almost minimizing in $U$ if it cannot be continuously deformed in $U$ into one other pair with $\fa$ smaller than $\e$ without passing through a pair with $\fa$ larger than $\e/8$.
The main result of this section is the following.
\begin{proposition}\label{prop:almost_minimizing_min-max_sequences}
Let $\sw$ a homotopically closed set of sweepouts; then there exist a min-max sequence $\ck{k}^{k \in \N} = \cktk^{k \in \N}$ and a function $r \colon \M \to (0,+\infty)$ that satisfy:
\begin{enumerate}
\item $\ck{k}$ converges, up to subsequences, to a pair $\pair \in \cu$ which is stationary for $\fa$;
\item There exists $\e_k \downarrow 0$ such that for every $x \in \M$ and for every annulus $\an \in \anul{r(x)}{x}$, there exists $N \in \N$ such that the sequence $\ck{k}^{k > N}$ is $\e_k$-almost minimizing in $\an$.
\end{enumerate}
%
\end{proposition}

The strategy of the proof is as follows: we have an homotopically closed set of sweepouts and a minimizing sequence $\{\cswok\}^{k \in \N}$ such that any min-max sequence converges to a stationary varifold, which is given by Theorem \ref{thm:pull-tight_procedure}. We prove that, if there were no almost minimizing min-max sequence, then the minimizing sequence can be deformed in another one $\ctswok^{k \in \N}$ such that $\lim_k \ms \fa \ctswok < \m$ and this would be a contradiction.

\subsection{Deforming a sweepout near a non-almost minimizing slice}

If a slice $\ct{t_0}$ of a sweepout is not $\varepsilon$-almost minimizing in some open set $U$, it can be deformed in another pair $\tt{}$, with $\fa \tt{} \leq \fa \ct{t_0} - \e$ by a parametrized family of pairs $\ft{t}_{t \in \I}$ such that $\ft{0}=\ct{t_0}$ and $\ft{1}=\tt{}$.

We cannot simply replace $\ct{t_0}$ with $\tt{}$, first of all because in this way we no longer have a smooth parametrized family of pairs and because we need to deform also the slices for times near $t_0$ to build the contradiction argument stated above. Hence we need the following lemma, which is similar to \cite[Lemma 3.1]{delellis2009existence} and \cite[Lemma 5.1]{delellis2017minmax}.

\begin{lemma}\label{lmm:freezing_procedure}
Let $\cswo$ be a sweepout and suppose $t_0 \in (0,1)$ is such that $\ct{t_0}$ is not $\e$-almost minimizing in $U \subset \M$ for some $\e>0$ and some open set $U$. Then for every open set $V \subset \M$ with $U \cc V$, there exists $\bar{\eta}>0$ such that, whenever we choose $\eta \leq \bar{\eta}$ and $0\leq t_0-\eta < a < a''< b'' < b < t_0+\eta \leq 1$, there exists a sweepout $\ctswo$ that is homotopic to $\cswo$ and satisfies the following properties:
\begin{enumerate}
\item $\tt{t}=\ct{t}$ for every $t \in [0,a] \cup [b,1]$;
\item $\tt{t}  \- V = \csw \- V$ for every $t \in \I$;
\item $\fa \ctsw < \fa \csw + \frac{\e}{4}$ for every $t \in \I$;
\item $\fa \ctsw < \fa \csw - \frac{\e}{2}$ for every $t \in [a'',b'']$.
\end{enumerate}
\end{lemma}
\begin{proof}
The number $\bar{\eta}$ will be chosen at the end of the proof.
\begin{steps}[wide,%
labelindent=5pt]
\item
The first step in the proof is to build the new sweepout $\tt{t}$; this is achieved first of all by ``deforming the time inside $U$", to have that the slice is frozen on $\ct{t_0}$ in an open neighborhood of $t_0$, whereas in $\M \- V$ nothing changes. While $\ct{t_0}$ remains frozen in $U$, we deform it therein by the family of $\ft{t}$ given by Definition \ref{def:almost_minimizing} which contradicts the $\e$-almost minimality of $\ct{t_0}$. Of course, this requires some technicalities.

Let $\{\om_{t}\}_{t \in \I}$ be the family of open sets relative to $\csw$.
Let us fix $V$ and choose two open sets $A,B$ such that
\begin{equation}
U \cc A \cc B \cc V
\end{equation}
and such that both $\s_{t_0}$ and $\g_{t_0}$ are smooth surfaces (with eventually non-empty smooth boundary) in $C := B \- A$ (this is possibile since the sets of singularities $S_{t_0}, G_{t_0}$ of $\s_{t_0}$ and $\g_{t_0}$ are finite).
Let us fix two non-negative smooth functions $\chi_A, \chi_B$ such that
\begin{itemize}
\item $\chi_A \in C_c^{\infty}(B)$, $\chi_B \in C_c^{\infty}(\M \- \clos{A})$;
\item $\chi_A(x) + \chi_B(x) \equiv 1$ for all $x \in \M$.
\end{itemize}
Since $\pM$ is $\breg$, there exists $r>0$ such that the distance function $\dist(x, \pM)$ is $\breg$ if $\dist(x, \pM) \leq 2r$. Let $d(\cdot)$ a smooth approximation of $\dist(\cdot, \pM)$ that coincides with $\dist(\cdot,\pM)$ whenever $\dist(x,\pM) \leq r$. By the regularity of $\pM$ and $d$, there exists a smooth function $F \colon x \in \M \mapsto \varphi_x \in \breg(B_r,B_r)$ such that each $\phi_x$ is a diffeomorphism and
\begin{equation}
\varphi_x\big(B_r \cap \dil{x,1}(\M) \big)  = B_r \cap \{x_1 \geq - d(x)\}
\qquad
\forall x \in \M,
\end{equation}
where $\dil{x,\lambda}$ is the dilation map defined in \eqref{eq:dilation_map}. Essentially, $\phi_x$ is a local chart that depends smoothly on $x$ and maps $\pM$ in $\{x_1 = -d(x)\}$.

By a suitable choice of $A, B, \bar{\eta}$ and by smooth convergence of $\s_t \to \s_{t_0}$ in $C$, for every $t \in (t_0 - \bar{\eta}, t_0 + \bar{\eta})$ $\s_t \cap C$ can be written as $\psi_t(\s_{t_0})$ for a smooth 1-parameter family of diffeomorphisms such that $\psi_{t_0}=\Id \restr{\s_{t_0}}$.
Moreover, if $\bar{\eta}$ is sufficiently small, then $\psi_t(x) \in B_r(x)$ for all $t \in (t_0-\bar{\eta}, t_0+\bar{\eta})$ and for all $x \in \M$.

All the previous arguments show that, for all $s,t \in (t_0-\bar{\eta}, t_0+\bar{\eta})$ and all $\lambda \in \I$, it is well defined
\begin{equation}
f_{t,s,\lambda}(x)
=
\phi_x^{-1} \Big(
\chi_B(x) \phi_x(\psi_t(x))
+
\chi_A (x) \big[ (1-\lambda)\phi_x( \psi_t(x))
+ \lambda \phi_x ( \psi_s(x)) \big]
\Big)
\quad
\forall x \in \M.
\end{equation}
In other words, $f_{t,s,\lambda}(\s_{t_0} \cap C)$ coincides with $\s_t$ in a neighborhood of $B$ and changes smoothly in a convex combination of $\s_{t}$ and $\s_{s}$ in a neighborhood $A$.
This definition ensures that $\de [f_{t,s,\lambda}(\s_{t_0} \cap C) ]\subset \pM$ for all $\s,t \in (t_0-\bar{\eta}, t_0+\bar{\eta})$ and all $\lambda \in \I$. Moreover, since $f_{t,s,\lambda}(\s_{t_0})$ is essentially a convex combination of $\s_t$ and $\s_s$ and by smooth convergence of the $\s_t$ in $C$, $\bar{\eta}$ can be chosen so that every $f_{t,s,\lambda}(\s_{t_0} \cap C)$ is diffeomorphic to $\s_{t_0} \cap C$ and
\begin{equation}\label{eq:estimate_convex_combination_a.m.}
\hn(f_{t,s,\lambda}(\s_{t_0} \cap C))
<
\hn(\s_t \cap C) + \frac{\e}{32}
\qquad
\forall t,s, \in (t_0-\bar{\eta}, t_0+\bar{\eta}),
\,
\forall \lambda \in \I.
\end{equation}
In a similar way we can define $g_{t,s,\lambda}$ which satisfies the same properties of $f_{t,s,\lambda}$ with $\g_t$ in place of $\s_t$, in particular
\begin{equation}\label{eq:estimate_boundary_convex_combination_a.m.}
\hn(g_{t,s,\lambda}(\g_{t_0} \cap C)) < \hn(\g_t \cap C) + \frac{\e}{32}
\qquad
\forall t,s, \in (t_0-\bar{\eta}, t_0+\bar{\eta}),
\,
\forall \lambda \in \I.
\end{equation}
Moreover we have that $g_{t,s,\lambda}(\g_{t_0} \cap C) \subset \pM$ for all $s,t \in (t_0-\bar{\eta}, t_0+\bar{\eta})$ and all $\lambda \in \I$. We underline that by the construction of $f_{t,s,\lambda}(\s_{t_0} \cap C)$ and $g_{t,s,\lambda}(\g_{t_0} \cap C)$, it follows that their boundaries coincide and are contained in $\pM$.
Again in a similar way, we can define the family of open sets $\Xi_{t,s,\lambda}$ relative to $f_{t,s,\lambda}(\s_{t_0})$ and $g_{t,s,\lambda}(\s_{t_0})$.

We now fix $a,a',a'',b'',b,b' \in (0,1)$ such that $t_0 - \bar{\eta} < a < a' < a'' < b'' < b' < b < t_0 + \bar{\eta}$ and two smooth functions $\beta, \gamma \colon \I \to \I$ such that
\begin{itemize}
\item $\beta(t)= 0$ for every $t \in [0,a] \cup [b,1]$ and $\beta(t)\equiv 1$ for every $t \in [a',b']$;
\item $\gamma(t) = t$ for every $t \in [0,a] \cup [b,1]$, $\gamma(t)= t_0$ for every $t \in [a',b']$ and $\gamma$ is monotone non-decreasing.
\end{itemize}
We now use the hypotheses that $\ct{t_0}$ is not $\e$-almost minimizing in $U$, that is there exists a parametrized pair of surfaces $\ft{t}_{t \in \I}$ which satisfies the properties of Definition \ref{def:almost_minimizing}.
Let us consider another smooth function $\omega \colon \I \to \I$ such that $\omega(t)=0$ for every $t \in [0,a']\cup [b',1]$ and $\omega(t)=1$ for every $t \in [a'',b'']$. We now define the new sweepout $\tt{t}_{t \in \I}$ as follows:
\begin{itemize}
\item $\tt{t}=\ct{t}$ for every $t \in [0,a] \cup [b,1]$;
\item $\tt{t} \- B =\ct{t} \- B$ for every $t \in \I$;
\item $\tt{t} \cap C = (f_{t,\gamma(t),\beta(t)}(\s_{t_0} \cap C),g_{t,\gamma(t),\beta(t)}(\g_{t_0} \cap C))$ for every $t \in (a,b)$;
\item $\tt{t} \cap A = \ct{\gamma(t)}$ for every $t \in [a,a'] \cup [b',b]$;
\item $\tt{t} \cap A = \ft{\omega(t)}$ for every $t \in (a',b')$.
\end{itemize}
In a similar way are defined the open sets $\tilde{\om}_t$ relative to $\tt{t}$.
By definition it is clear that $\tt{t}_{t \in \I}$ is homotopic to $\ct{t}_{t \in \I}$ and satisfies the first two conditions of the lemma.
\item The next step is to estimate $\fa \tt{t}$ to complete the proof of the lemma. Since $\tt{t}= \ct{t}$ for every $t \in [0,a] \cup [b,1]$, we have to estimate $\fa \tt{t}$ only for $t \in (a,b)$.
\begin{itemize}
\item Let us consider $t \in (a,a'] \cup [b',b)$; We have that $\tt{t} \- B = \ct{t} \- B$. In $C$ we have
\begin{gather}
\hn(\ts{t} \cap C)
=
\hn(f_{t,\gamma(t),\beta(t)}(\s_{t_0} \cap C))
<
\hn(\s_t \cap C) + \frac{\e}{32},
\\
\hn(\tg{t} \cap C)
=
\hn(g_{t,\gamma(t),\beta(t)}(\g_{t_0} \cap C))
<
\hn(\g_t \cap C) + \frac{\e}{32},
\end{gather}
where the inequalities follow by \eqref{eq:estimate_convex_combination_a.m.} and \eqref{eq:estimate_boundary_convex_combination_a.m.} and by $\gamma(t) \in (a,b)$ for $t \in (a,b)$. In $A$ we have
\begin{gather}
|\hn(\ts{t} \cap A) - \hn(\s_{t} \cap A)|
=
|\hn(\s_{\gamma(t)} \cap A) - \hn(\s_{t} \cap A)|
<
\frac{\e}{32}
\\
|\hn(\tg{t} \cap A) - \hn(\g_{t} \cap A)|
=
|\hn(\g_{\gamma(t)} \cap A) - \hn(\g_{t} \cap A)|
<
\frac{\e}{32}
\end{gather}
by Lemma \ref{lmm:convergence_area_any_open} if $\bar{\eta}$ is chosen sufficiently small. Gathering these estimates we get
\begin{equation}
\fa \tt{t}
<
\fa \ct{t} + \frac{\e}{8}
\qquad
\forall t \in (a,a'] \cup [b',b).
\end{equation}
\item
Now we consider $t \in (a',a'') \cup (b'',b')$. The estimates in $\M \- B$ and in $C$ are the same of the case $t \in (a,a'] \cup [b',b)$, the only difference being in $A$. Thus we have to estimate only in $A$. Therein we have
\begin{equation}
\begin{split}
\hn(\ts{t} \cap A)
+
a \hn(\tg{t} \cap A)
& =
\hn(\Phi_{\omega(t)} \cap A)
+
a \hn(\Psi_{\omega(t)} \cap A)
\\
& <
\hn(\s_{t_0} \cap A) +  a \hn(\g_{t_0} \cap A) + \frac{\e}{8}
\\
& <
\hn(\s_{t} \cap A) + a\hn(\s_{t} \cap A) + \frac{\e}{8} + \frac{\e}{16}
\end{split}
\end{equation}
where the first inequality follows by the fact that $\ft{s}$ contradicts the $\e$-almost minimality of $\ct{t_0}$ in $U$ (recall that $\ft{s} \- U = \ct{t_0} \- U$), whereas the second inequality is true by Lemma \ref{lmm:convergence_area_any_open} if we choose $\bar{\eta}$ sufficiently small.

By this and the estimate in $C$ we have
\begin{equation}
\fa \tt{t}
<
\fa \ct{t} + \frac{\e}{4}
\qquad
\forall t \in (a',a'') \cup (b'',b').
\end{equation}
\item
Now consider the case $t \in [a'',b'']$. Again, the estimates in $\M \- B$ and in $C$ as the same as the other cases. We racall that, by definition, $\ts{t} \cap A= \Phi_1 \cap A$ and $\tg{t} \cap A = \Psi_1 \cap A$. Thus
\begin{equation}
\begin{split}
\hn(\ts{t} \cap A) + a\hn(\tg{t} \cap A)
& =
\hn(\Phi_{1} \cap A) + a \hn(\Psi_1 \cap A)
\\
& <
\hn(\s_{t_0} \cap A) + a \hn(\g_{t_0} \cap A)- \frac{\e}{2}
\\
& <
\hn(\s_t \cap A) + a \hn(\g_{t} \cap A) - \frac{3 \e}{8}.
\end{split}
\end{equation}
where the first inequality follows by the definition of $\ft{t}$ and the second is true if we choose $\bar{\eta}$ sufficiently small. Gathering this estimate with the one in $C$, we get
\begin{equation}
\fa \tt{t}
<
\fa\ct{t} - \frac{\e}{4}
\qquad
\forall t \in [a'',b''].
\end{equation}
\end{itemize}
If we choose $\bar{\eta}>0$ so small so that all the previous estimate are valid, $\tt{t}$ satisfies all the conclusion of the lemma. It is clear that all the previous arguments remain true whenever $0< \eta \leq \bar{\eta}$.
\end{steps}
\end{proof}

\subsection{Almgren-Pitts combinatorial argument}
The aim of this section is to show that there exists a min-max sequence of pairs which is almost minimizing in at least one of any \emph{admissible pair} of open sets, that is a pair of open sets which are at a sufficiently large distance compared to their diameter. This is used in \ref{subsec:proof_almost_minimizing_min-max} to complete the proof of Proposition \ref{prop:almost_minimizing_min-max_sequences}.

\begin{defi}[Admissible pair of sets]
If $U_1, U_2 \subset \M$ are relatively open and non-empty, we say that $(U_1, U_2)$ is an \emph{admissible pair of open sets} if
\begin{equation}
\dist(U_1,U_2) \geq 4 \min\{\diam U_1,\diam U_2 \}.
\end{equation}
We call $\co$ the family of admissible pairs of open sets.
\end{defi}

\begin{defi}
For any $\e>0$ and $(U_1,U_2) \in \co$, we say that $\ct{}$ is $\varepsilon$-almost minimizing in $(U_1,U_2)$ if it is $\e$-a.m. in at least one of $U_1, U_2$.
\end{defi}

We can now state the fundamental result of this subsection.

\begin{lemma}[Almgren-Pitts combinatorial lemma] \label{lmm:A-P_comb_lemma}
Let $\sw$ be a homotopically closed family of sweepouts; there exists a minimizing sequence of sweepouts $\ctk{t}{k}_{t \in \I}^{k \in \N}$ and a min-max sequence $\ck{k}^{k \in \N}:= \ctk{t_k}{k}^{k \in \N}$ that satisfies, for any $k \in \N$, the following properties:
\begin{enumerate}
\item $\fa \ck{k} \geq \m - \frac{1}{k}$;
\item $\fa \ck{k}$ is $\frac{1}{k}$-a.m. in every $(U_1,U_2) \in \co$;
\item $\ck{k}$ converges to a pair in $\cu$ which is stationary for $\fa$.
\end{enumerate}
\end{lemma}

Before the proof of Lemma \ref{lmm:A-P_comb_lemma} we state the following simple and useful property.

\begin{lemma}\label{lmm:admisible_pairs_disjoints}
If $(U_1,U_2), (V_1,V_2)$ are pairs of non-empty open sets such that
\begin{gather}
\dist(U_1,U_2) \geq 2 \min\{\diam U_1,\diam U_2 \}
\\
\dist(V_1,V_2) \geq 2 \min\{\diam V_1,\diam V_2 \},
\end{gather}
then there exist two indices $i,j \in \{1,2\}$ such that $U_i \cap V_j = \emptyset$
\end{lemma}
\begin{proof}
Without loss of generality we can suppose that $\diam U_1 \leq \diam U_2$ and $\diam V_1 \leq \diam V_2$. Moreover we can suppose $\diam U_1 \leq \diam V_1$.

If $U_1 \cap V_1 = \emptyset$ there is nothing to prove. Otherwise, there exists $\xi \in U_1 \cap V_1$. Then, for every $x \in U_1$ and $y \in V_2$
\begin{equation}
\begin{split}
|x-y|
& \geq
|y - \xi| - |\xi - x| \\
& \geq
2 \diam V_1 - \diam U_1 \\
& \geq
\diam V_1
\\
& \geq
\diam U_1.
\end{split}
\end{equation}
Since $U_1$ is relatively open and $\pM$ is of class $\breg$, if $x,y$ both belonged to $U_1$, then one would have $|x-y|< \diam U_1$.
This shows that $U_1 \cap V_2 = \emptyset$.
\end{proof}
We can now pass to the proof of Lemma \ref{lmm:A-P_comb_lemma}.
\begin{proof}[Proof of Lemma \ref{lmm:A-P_comb_lemma}]
Let $\{\ctk{t}{k}\}^k$ be the minimizing sequence of sweepouts given by Theorem \ref{thm:pull-tight_procedure};
%
Up to passing to a subsequence, we can assume that for every $k \in \N$
\begin{equation}
\max_{t \in [0,1]} \fa\ctk{t}{k}
\leq
\m + \frac{1}{8k}.
\end{equation}
This suitable subsequence (not relabeled) is clearly still minimizing and we are going to prove that is the minimizing sequence stated in the theorem. Since by Theorem \ref{thm:pull-tight_procedure} every min-max sequence converges to a stationary pair up to subsequences, we have only to find a min-max sequence with the almost minimizing property.
For any $k \in \N$ let us define the set of times
\begin{equation}
K_k = \Big\{
t \in [0,1] \mid \fa\ctk{t}{k} \geq \m - \frac{1}{k}
\Big\}.
\end{equation}
$K_k$ is clearly compact.
To prove the existence of the min-max sequence with the desired properties, we argue by contradiction: if there exists $k \in \N$ such that for every $t \in K_k$ $\ctk{t}{k}$ is neither $\frac{1}{k}$-almost in $U_{1,t}$ nor in $U_{2,t}$ for some $\upair{t} \in \co$, we construct a sweepout $\tt{t}_{t \in \I}$ that is homotopic to $\ttk{t}{k}_{t \in \I}$ and satisfies $\ms \fa \tt{t} < \m$, which is absurd. Thus we fix such a $k \in \N$. Since $k$ is fixed, we drop the subscripts and superscripts $k$.

Let us assume, by contradiction, that for any $t \in K$
there exists $\upair{t} \in \co$ such that $\ct{t}$ is neither $\frac{1}{k}$-almost in $U_{1,t}$ nor in $U_{2,t}$.
\begin{steps}[wide,%
labelindent=5pt]
\item As first step, we want to show that there exists a covering $\{J_j\}_{j=1}^M$ of $K$ made of closed intervals with the following properties:
\begin{itemize}
\item
For every $j=1,\dots, M$, there exist $(a_j,b_j) \subset J_j$, an open set $W_j \subset \M$ and a sweepout $\ttk{t}{j}_{t \in \I}$ that is homotopic to $\ct{t}_{t \in \I}$ and satisfies:
\begin{enumerate}
\item $\ttk{t}{j}=\ct{t}$ for every $t \in \I \- J_j$;
\item $\ttk{t}{j} \- W_j = \ct{t} \- W_j$ for every $t \in \I$;
\item \label{pnt:c_comb_lemma_competitor} $\fa \ttk{t}{j} < \fa \ct{t} + \frac{1}{4k}$ for every $t \in [0,1]$;
\item \label{pnt:d_comb_lemma_competitor}$\fa \ttk{t}{j}  < \fa \ct{t} - \frac{1}{2k}$ for every $t \in [a_j,b_j]$.
\end{enumerate}
\item
\begin{equation}
K \subset \bigcup_{j = 1}^M (a_j,b_j)
\end{equation}
\item For every $j=1,\dots,M$, $J_j$ intersects at most $J_{{j-1}}$ and $J_{{j+1}}$;
\item If $i \neq j$ and $J_i \cap J_j \neq \emptyset$, then $W_i \cap W_j = \emptyset$.
\end{itemize}

By the assumptions, for every $t \in K$, $i=1,2$ and for any open set $V_{i,t} \supset\supset U_{i,t}$, let $\bar{\eta_t}$ the number given by Lemma \ref{lmm:freezing_procedure} and for every $0< \eta_t \leq \bar{\eta_t}$ let us define
$I_{t,\eta_t}= [t- \eta_{t}, t+ \eta_{t}]$.
We choose the sets $V_{i,t}$ such that
\begin{equation}
\dist(V_{1,t},V_{2,t}) \geq 2 \min\{\diam V_{1,t}, \diam V_{2,t}\}.
\end{equation}
By compactness of $K$, there exists a finite set of points $t_1 < \dots < t_m$ such that
\begin{equation}\label{eq:K_subset_interior_Ij}
K \subset \bigcup_{j=1}^m \interior{(I_{t_j,{\bar{\eta}}_{t_j}})}.
\end{equation}
Since for every $t_j$ we can choose $\eta_j \leq \bar{\eta}_{t_j}$, we can obtain that each $I_{t_j,\eta_j}$ intersects at most $I_{t_{j-1},\eta_{j-1}}$ and $I_{t_{j+1},\eta_{j+1}}$ and still
\begin{equation}\label{eq:K_subset_interior_Ij2}
K \subset \bigcup_{j=1}^m \interior{(I_{t_j,\eta_j})}.
\end{equation}
For every $j=1, \dots, m$ and $i=1,2$, we denote $V_j^i := V_{i,t_j}$ and $I_j := I_{t_j,\eta_j}$.

Starting from this covering, we want to construct the covering $\{J_j\}_{j=1}^M$ of $K$.
We begin by defining $J_1$ and $W_1$:
\begin{itemize}
\item If $I_1 \cap I_2 = \emptyset$, then we set $J_1 = I_1$, $W_1= V_1^1$;
\item If $I_1 \cap I_2 \neq \emptyset$, up to relabeling the sets $V_j^i$, by Lemma \ref{lmm:admisible_pairs_disjoints} we can suppose that $V_1^1 \cap V_2^1 = \emptyset$. Then we set $J_1=I_1$ and $W_1 = V_1^1$.
\end{itemize}
Now, in any of the two above cases, by Lemma \ref{lmm:admisible_pairs_disjoints}, we can assume that  $V_1^1 \cap V_2^1 = \emptyset$; to choose $J_2, W_2$,
we have to consider several cases:
\begin{enumerate}
\item\label{item:1_comb_lemma} If $I_2 \cap I_3 = \emptyset$, then we set $J_2= I_2$ and $W_2 = V_2^1$;
\item\label{item:2_comb_lemma} If $I_2 \cap I_3 \neq \emptyset$, by Lemma \ref{lmm:admisible_pairs_disjoints} there exists two indices $i,\l$ such that $V_2^i \cap V_3^\l = \emptyset$. Without loss of generality we can suppose $\l=1$.
\begin{enumerate}

\item\label{item:2a_comb_lemma} If $i=1$, we set $J_2=I_2$, $W_2= V_2^1$;

\item\label{item:2b_comb_lemma} If $i=2$, we can find two closed intervals $L,H$ such that
\begin{equation}
\interior{L} \cup \interior{H} = \interior{I_2}
\qquad
L \cap I_3 = \emptyset
\qquad
H \cap I_1 = \emptyset
\end{equation}
(the first condition implies that $L,H$ cover $I_2$ and that their interiors overlap in the middle of $I_2$).
We then set $J_2=L$, $W_2=V_2^1$, $J_3=H$, $W_3=V_2^2$.

\end{enumerate}
\end{enumerate}
To choose the next interval and open set, we proceed as follows:
\begin{itemize}
\item In the case \ref{item:1_comb_lemma}, we choose $J_3,W_3$ as we have done for $J_1, W_1$ (since $I_3 \cap I_2 = \emptyset$);
\item In the case \ref{item:2a_comb_lemma} we can choose $J_3,W_3$ as we have done for $J_2, W_2$;
\item In the case \ref{item:2b_comb_lemma} we can choose $J_4,W_4$ as we have done for $J_2,W_2$.
\end{itemize}
Thus we can continue this procedure until we complete the choice of all the intervals $I_j$; it is clear that, by construction, each each $J_j$ intersects at most $J_{j-1}$ and $J_{j+1}$ and that, if $J_j \cap J_\l \neq \emptyset$, then $W_j \cap W_\l = \emptyset$. Moreover, by \eqref{eq:K_subset_interior_Ij2} and the above construction, we have $K \subset \bigcup_{j=1}^M \interior{J_j}$. Hence for every $j=1, \dots, M$ we can choose, $a_j,b_j \in \interior{J_j}$  such that
\begin{equation}
K \subset \bigcup_{j=1}^M (a_j,b_j).
\end{equation}
This complete the definition of the covering $\{J_j\}_{j=1}^M$ and the open sets $W_j$; we have now to define the sweepouts $\ttk{t}{j}_{t \in \I}$ for $j=1,\dots, M$.

For every $j=1,\dots, M$ there exists $\l \in \{1,\dots,m\}$ and $i \in \{0,1\}$ such that $(a_j,b_j) \subset \interior{J_j} \subset I_\l=I_{t_\l,\eta_\l}$, $W_j = V_{\l}^i$ and $\ct{t_\l}$ is not $\frac{1}{k}$-almost minimizing in $W_j$. Thus, by Lemma \ref{lmm:freezing_procedure} there exists a sweepout $\ttk{t}{j}_{t \in \I}$ that is homotopic to $\ct{t}_{t \in \I}$ and such that
\begin{enumerate}
\item $\ttk{t}{j}=\ct{t}$ for every $t \in \I \- J_j$;
\item $\ttk{t}{j} \- W_j = \ct{t} \- W_j$ for every $t \in \I$;
\item \label{pnt:c_comb_lemma_competitor} $\fa \ttk{t}{j} < \fa \ct{t} + \frac{1}{4k}$ for every $t \in [0,1]$;
\item \label{pnt:d_comb_lemma_competitor}$\fa \ttk{t}{j}  < \fa \ct{t} - \frac{1}{2k}$ for every $t \in [a_j,b_j]$,
\end{enumerate}
which are the required conditions to complete the proof of the first step.

\item
We now want to ``glue together" the deformations $\ttk{t}{j}$ in an unique sweepout $\tt{t}_{t \in \I}$ that is homotopic to $\ct{t}_{t \in \I}$ and satisfies $\ms \fa \tt{t} < \m$, which is a contradiction.

We define the competitor sweeoput $\tt{t}_{t \in \I}$ as follows:
\begin{itemize}
\item If $t \notin \bigcup_{j=1}^M J_j$, then $\tt{t} = \ct{t}$;
\item If $t \in J_j \setminus(J_{j+1}\cup J_{j-1})$ then $\tt{t} = \ttk{t}{j}$;
\item If $t \in J_j \cap J_{j+1}$, then
\begin{gather}
\ts{t} \cap W_j = \ts{t}^j \cap W_j,
\quad
\ts{t} \cap W_{j+1} = \ts{t}^{j+1} \cap W_{j+1},
\quad
\ts{t} \- (W_j \cup W_{j+1}) = \s_t \setminus(W_j \cup W_{j+1});
\\
\tg{t} \cap W_j = \tg{t}^j \cap W_j,
\quad
\tg{t} \cap W_{j+1} = \tg{t}^{j+1} \cap W_{j+1},
\quad
\tg{t} \- (W_j \cup W_{j+1}) = \g_t \setminus(W_j \cup W_{j+1});
\end{gather}
\end{itemize}
These are all the possible cases, since $t$ can belong to at most two consecutive intervals. Clearly $\tt{t}_{t \in \I}$ is well-defined, because by construction of the intervals $J_j$ and open sets $W_j$, if $J_j \cap J_\l \neq \emptyset$, then $W_j \cap W_\l = \emptyset$; Since every $\ttk{t}{j}_{t \in \I}$ is homotopic to $\ct{t}_{t \in \I}$, also the sweepout $\tt{t}_{t \in \I}$ is homotopic to $\ct{t}_{t \in \I}$. The open sets $\tot{t}$ relative to the pairs $\tt{t}$ are defined in the same way, starting from the open sets $\tot{t}^j$ relative to $\ttk{t}{j}$.

We have now to estimate $\fa \tt{t}$ to get the contradiction:
\begin{enumerate}
\item If $t \notin K$, then $t$ can be in at most two of the intervals $J_j$; thus $\tt{t}$ coincides with $\ttk{t}{j}$ and $\ttk{t}{j+1}$ respectively in the two open sets $W_j, W_{j+1}$ and coincides with $\ct{t}$ in $\M \- (W_j \cup W_{j+1})$. Since for every $j$ it holds $\fa\ttk{s}{j} < \fa\ct{s} + \frac{1}{4k}$ for every $s \in \I$, we have
\begin{equation} \label{eq:estimate_a_competitor_comb}
\fa \tt{t}
<
\fa\ct{t} + \frac{2}{4k}
<
m_0 - \frac{1}{k} + \frac{1}{2k}
=
m_0 - \frac{1}{2k},
\end{equation}
where the second inequality is given by $t \notin K$, thus $\fa\ct{t} < \m - \frac{1}{k}$ .
\item If $t \in K$, then there exists $j$ such that $t$ belongs to at least one $(a_j,b_j)$ and at most in another $J_\l$ for $\l=j-1$ or $\l=j+1$. Hence, by the properties of the sweepouts $\ttk{s}{h}_{s \in \I}$ previously defined one has
\begin{gather}
\hn(\ts{t}^j \cap W_j) + a\hn(\tg{t}^j \cap W_j)
<
\hn(\s_{t} \cap W_j) + a\hn(\g_{t} \cap W_j) - \frac{1}{2k}
\\
\hn(\ts{t}^\l \cap W_\l) + a\hn(\tg{t}^\l \cap W_\l)
<
\hn(\s_{t} \cap W_\l) + a\hn(\g_{t} \cap W_\l) + \frac{1}{4k}.
\end{gather}
Thus
\begin{equation} \label{eq:estimate_b_competitor_comb}
\fa \tt{t}
\leq
\fa \ct{t} - \frac{1}{4k}
\leq
\m + \frac{1}{8k} - \frac{1}{4k}
=
\m - \frac{1}{8k},
\end{equation}
where we used $\fa \ct{t} \leq \m + \frac{1}{8k}$.
\end{enumerate}
The estimates \eqref{eq:estimate_a_competitor_comb} and \eqref{eq:estimate_b_competitor_comb} show that
\begin{equation}
\max_{t \in [0,1]} \fa \tt{t}
<
\m ,
\end{equation}
which is absurd since $\tt{t}_{t \in \I}$ is homotopic to $\ct{t}_{t \in \I}$.
\end{steps}
\end{proof}

\subsection{Proof of Proposition \ref{prop:almost_minimizing_min-max_sequences}}\label{subsec:proof_almost_minimizing_min-max}
We use the min-max sequence $\ck{k}^{k \in \N}$ given by Lemma \ref{lmm:A-P_comb_lemma} to prove Proposition \ref{prop:almost_minimizing_min-max_sequences}.

\begin{proof}[Proof of Proposition \ref{prop:almost_minimizing_min-max_sequences}]
Let us consider the min-max sequence $\ck{k}^{k \in \N}$ given by Lemma \ref{lmm:A-P_comb_lemma}. It converges to a stationary pair and, for every $k \in \N$, it is $\frac{1}{k}$-almost minimizing in any pair of open sets $(U_1,U_2) \in \co$.

We begin with a simple useful remark: if a surface is $\varepsilon$-a.m. in a set $U$ and $V \subset U$, then it is $\varepsilon$-a.m. also in $V$.
For every $x \in \M$ and $r>0$, we have $\bigl(B_r(x), \M \setminus B_{9r}(x)\bigr) \in \co$, provided that $r$ is small enough to have $\M \setminus B_{9r}(x) \neq \emptyset$.
\begin{itemize}

\item If for every $x \in \M$ there exists $\rho(x)>0$ and $N(x) \in \N$ such that $\ck{k}$ is $\frac{1}{k}$-almost minimizing in $B_{\rho(x)}(x)$ for every $k \geq N(x)$, then the proof is complete if we set $\e_k = \frac{1}{k}$, $r(x)=\rho(x)$, since $\ck{k}$ is $\frac{1}{k}$-a.m. in every $\an \in \anul{r(x)}{x}$ for $k \geq N(x)$.

\item If there exists $x \in \M$ and a subsequence $\ck{k_j}_{j \in \N}$ such that every $\ck{k_j}$ is not $\frac{1}{k_j}$-a.m. in $B_{1/j}(x)$, then for every $j \in \N$ the pair $\ck{k_j}$ is $\frac{1}{k_j}$-a.m. in $\M \- B_{9/j}(x)$. For every $j \in \N$ we define
\begin{equation}
\e_j = \frac{1}{k_j};
\quad
\tk{j}= \ck{k_j};
\quad
r(y)
=
\begin{cases}
\diam \M & \text{if } y=x, \\
\dist(y,x) & \text{if } y \neq x.
\end{cases}
\end{equation}
We claim that the sequence $\tk{j}_{j \in \N}$ and the function $r$ satisfy the conclusion. In fact for every $\an(x,s,t) \in \anul{r(x)}{x}$ there exists $N \in \N$ such that if $j > N$ then $\an(x,s,t) \subset \M \- B_{1/9j}(x)$ and by the initial remark $\tk{j}$ is $\frac{1}{k_j}$-a.m. in $\an(x,s,t)$. If $y \neq x$ and $\an(y,s,t) \in \anul{r(y)}{y}$, then $t < \dist(y,x)$ and there exists $N \in \N$ such that if $j > N$ then $\dist(y,x) - \frac{1}{9j} < t$. Thus $\an(y,s,t) \subset \M \- B_{1/9j}(x)$ and  by the initial remark $\tk{j}$ is $\frac{1}{k_j}$-a.m. in $\an(x,s,t)$.
\end{itemize}
\end{proof}

\section{Bernstein's Theorem, curvature estimates and compactness}\label{sec6}
In this section we report some properties of (smooth) stable pairs for the capillarity functional.
The second variation of the capillarity functional $\fa$ on a pair of surfaces $\c$ in $\M$ induces the so called stability operator $Q$, which is a quadratic form on the set of smooth functions with compact support.

\begin{proposition}[Stability operator]
Let $\c$ be a pair of surfaces in $\M$. Then for any $X \in \Xt$ we have
\begin{equation}
\dtfa \pair [X]
=
Q(u),
\end{equation}
where $u(x) = X(x) \cdot \nu_\s(x)$ for every $x \in \s$, and
\begin{equation}\label{eq:definition_stability_operator}
Q(u)
=
\int_\s |\nabla u|^2 - |\as|^2 u^2 \dif \haus{2}
-
\int_{\de \s} q u^2 \dif \haus{1},
\end{equation}
where
\begin{equation}
q
=
\frac{H_{\pM}}{\sin \theta} + \cot \theta H_\s - \kappa_{\de \s}.
\end{equation}
Here $\kappa_{\de \s}= {D_\tau \eta_\s}\cdot{\tau}$, where $\tau$ is a unit tangent vector field on $\de \s$.
In particular, $\c$ is stable for $\fa$ if and only if the operator $Q$ is non-negative definite on $C_c^\infty(\M)$.
\end{proposition}
For a proof we refer to \cite[Appendix]{ros97} and to \cite[section 2.1]{hong2021capillary}.
If $\c$ is stationary for $\fa$ in a half space, then $H_{\pM}=0$ and $H_\s=0$. Moreover, if the half space is defined by \(\{x_1\ge 0\}\), by testing the stability inequality with the function 
\[
 u=\z \Bigl(\frac{1-\cos \theta {\nu_\s}\cdot{e_1} }{\sin \theta}\Bigr)
\]
we obtain the following proposition, compare with the proof of \cite[Theorem 1.5]{hong2021capillary}.
%
\begin{proposition}[Stability inequality for Capillarity functional]\label{prop:stability_inequality_capillarity}
Assume that $\NN$ is a half-space. Then there exists an universal constant $c(a)>0$ such that, for every pair $\c$ which is stationary and stable for $\fa$ and for every $\z \in C_c^\infty(\ru)$ we have
\begin{equation}\label{eq:stability_inequality_cap}
\int_\s \z^2 |\as|^2 \dif \hn
\leq
c \int_\s  |\nabla \z|^2 \dif \hn,
\end{equation}
where $|\as|^2$ is the norm squared of the second fundamental form of $\s$.
\end{proposition}
\begin{proof}
Without loss of generality, we can assume that $\NN=\{x_1 \geq 0\}$. Let us define
\begin{equation}
\varphi(x)
=
\frac{1-\cos \theta {\nu_\s(x)}\cdot{e_1} }{\sin \theta}
\end{equation}
and let $\z \in C_c^\infty(\ru)$ be any test function.
The computations in the first part of the proof of \cite[Theorem 1.5]{hong2021capillary} yield (see \cite[(4.26)]{hong2021capillary})
\begin{equation}
Q(\z \varphi)
=
\int_{\s} \varphi^2 |\nabla \z|^2 - \z^2 \varphi \frac{|\as|^2}{\sin \theta}.
\end{equation}
Since
\begin{equation}
0 < \frac{1-\cos \theta}{\sin \theta} \leq \varphi \leq \frac{1+ \cos \theta}{\sin \theta},
\end{equation}
the result easily follows.
\end{proof}

The stability inequality is used to obtain a Bernstein-type theorem for stationary and stable surfaces with respect to the capillarity energy  in a three-dimensional half space.

\begin{theorem}[Bernstein Theorem]\label{thm:Bernstein_cap}
Let us assume that $\pas$ is stationary and stable for $\fa$ in $\NN \- \{0\}$, where $\NN := \{x_1 \geq 0\} \subset \R^3$, and that there exists a constant $M>0$ such that
\begin{equation}\label{eq:volume_ratio_bernstein}
\sup_{R \in (0,\infty)} \frac{\haus{2}(\s \cap B_R(0))}{R^2}
\leq M;
\end{equation}
Then $\s$ and $\g$ are half planes and $\s$ forms an angle $\theta$ with $\de \NN$ such that $a = \cos \theta$.
\end{theorem}
\begin{proof}
The proof is classical and we just point out that it is enough to test \eqref{eq:stability_inequality_cap} with the following logarithmic cut-off function, for $N \in \N$:
\begin{equation}
\z_N(x)
=
\begin{cases}
0 & \text{if }|x| \leq e^{-2N} \\
2 + \frac{\log(|x|)}{N} & \text{if } e^{-2N} \leq |x| \leq e^{-N}\\
1 & \text{if } e^{-N} \leq |x| \leq e^N\\
2 - \frac{\log(|x|)}{N} & \text{if } e^N \leq |x| \leq e^{2N}\\
0 & \text{if } |x| \geq e^{2N}.
\end{cases}
\end{equation}
This, together with the area bound \eqref{eq:volume_ratio_bernstein}, gives
\begin{equation}
\int_{\s \cap \big( B_{e^N} \- B_{e^{-N}} \big)} |\as|^2 \dif \haus{2}
\leq
\frac{2c M e^2}{N}.
\end{equation}
By sending $N \to \infty$, we obtain $\as \equiv 0$, hence the conclusion.
\end{proof}

As it is well known, Bernstein theorem as above implies  local curvature estimates, see \cite[Theorem 1.6]{hong2021capillary} for the (classical) argument.
\begin{theorem}[Curvature estimates for stable surfaces]\label{thm:curvature_estimates}
There exists a constant $\Lambda = \Lambda(a) >0$ such that, for every $\pas$ stationary stable pair for $\fa$ in $\M$, it holds
\begin{equation}
|\as|(x) \dist\nolimits_{\s}(x, \de \s \- \pM) \leq \Lambda
\qquad
\forall x \in \s.
\end{equation}
\end{theorem}
A straightforward consequence of the above curvature estimates is the following compactness result.
\begin{theorem}[Compactness for stable pairs]\label{thm:compactness_stable_capillarity}
Let $U\subset \M$ be open and let $\ck{k}^k$ be a sequence of stable stationary pairs for $\fa$ in $\M$ satisfying
\begin{enumerate}
\item For every $k \in \N$, $\s_k$ and $\g_k$ are smooth up to the boundary;
\item every $\ck{k}$ is stationary and stable for $\fa$ in $U$;
\item $\sup_k \fa\ck{k} < +\infty$;
\end{enumerate}
Then, up to subsequences, $\ck{k}$ converges in the sense of varifolds to a stationary stable pair $\pair$ of varifolds for $\fa$ such that:
\begin{enumerate}
\item $W = \g$ where $\g \subset \pM$ is a surface with smooth boundary;
\item $V$ is, up to multiplicity, a smooth stable minimal surface $\s$ in $U$ with boundary $\de \s = \de \g$ and meets $\pM$ with contact angle $\theta$.
\item For every open set $U' \cc U$, the convergences $\s_k \to \s$ and $\g_k \to \g$ in $U'$ are smooth.
\end{enumerate}
\end{theorem}

\section{Replacements}\label{sec7}
To prove regularity for  $\pair$  we follow Almgren and Pitts and introduce  the notion of replacement

\begin{defi}[Replacements]
Let $\pair$ be stationary for $\fa$ in $\M$ and let $U \subset \M$ be relatively open; a pair $\tpair$ is a \emph{replacement} of $\pair$ in $U$ if
\begin{enumerate}
\item $\tpair$ is stationary for $\fa$ in $\M$ and $\tv = V$ and $\tw = W$ in $\M \- U$;
\item $\ntv(\M) = \nv(\M)$ and $\ntw(\M) = \nw(\M)$;
\item $\tv$ and $\tw$ are induced in $U$ by smooth surfaces $\ts{}$ and $\tg{}$ with boundary such that the pair $\tt{}$ is stationary and stable for $\fa$ in $U$.
\end{enumerate}
\end{defi}
The key result of this section is the existence of replacements in every annulus $\an$ in $\anul{r(x)}{x}$.
\begin{proposition}[Existence of replacements]\label{prop:existence_replacements}
Let $\ck{k}^k, \pair, r$ be respectively the min-max sequence, the pair of varifolds and the function $r$ given by Proposition \ref{prop:almost_minimizing_min-max_sequences}. Let $x \in \M$ and let $\an:= \an(x,r,s) \in \anul{r(x)}{x}$; then there exists a min-max sequence $\tk{k}^k$ and a function $\tilde{r}: \M \to (0,+\infty)$ such that
\begin{itemize}
\item $\tk{k}$ converges, up to subsequences, to a pair $\tpair$ which is a replacement for $\pair$ in $\an=\an(x,r,s)$;
\item For every $y \in \M$, $\tk{k}$ is almost minimizing in every annulus in $\anul{\tilde{r}(y)}{y}$;
\item $\tilde{r}(x)=r(x)$.
\end{itemize}
\end{proposition}
Since the proposition is already proved for points $x \in \interior{\M}$ in \cite[Proposition 2.6]{delellis2009existence}, throughout the section we assume that $x \in \pM$.

The idea to prove the existence of replacements is the following: we fix an annulus $\an = \an(x,r,s) \in \anul{r(x)}{x}$; since every $\ck{k}$ is, for $k$ large enough, $\varepsilon_k$-a.m. in $\an$ for the sequence $\e_k \downarrow 0$ stated by Proposition \ref{prop:almost_minimizing_min-max_sequences}, we fix such a $k$ and consider all homotopic deformations $\ttk{t}{k}_{t \in \I}$ of $\ck{k}$ supported in $\an$ which do not ``pass" through pairs for which $\fa$ is larger than $\fa \ck{k} + \frac{\varepsilon_k}{8}$, that is
\begin{equation}\label{eq:intro_replac_area}
\fa \ttk{t}{k}
\leq
\fa \ck{k} + \frac{\varepsilon_k}{8}
\qquad
\forall t \in \I
\end{equation}
We can consider a minimizing sequence $\ctk{j}{k}_j$ for this class of deformations that converges, up to subsequences, to a pair $\tpairk{k}$, which is stationary for $\fa$ in $\an$ and coincides with $\ck{k}$ in $\M \- \an$, since all deformations are supported in $\an$.
%
The main step is to prove that $\tpairk{k}$ is induced by regular surfaces in $\an$. This is based on the fact that if we consider a sufficiently small ball $B \subset \an$, the sequence $\ctk{j}{k}_j$ is actually minimizing $\fa$ in $B$ \emph{without} the restriction \eqref{eq:intro_replac_area}. This allows us to use the regularity theory for sets which minimize the capillarity functional $\fa$, see \cite[Theorem 1.2]{DePMag14} and \cite[Theorem 1.5]{dephilippis2014dimensional}, to obtain that $\tpairk{k}$ is regular inside $\an$.

Up to subsequences, $\tpairk{k}$ converges as $k \to \infty$, to a pair of varifolds $\tpair$.
$\tpair$ coincides with $\pair$ in $\M \- \an$ and,
By Theorem \ref{thm:compactness_stable_capillarity}, coincides in $\an$ with a regular pair $\tk{}$ which is stationary and stable for $\fa$.
Thus $\tpair$ is the desired replacement.

We fix $x \in \pM$ and an annulus $\an = \an(x,r,s) \in \anul{r(x)}{x}$. These are fixed throughout the remainder of the section, and when we simply write $\an$ we refer to $\an(x,r,s)$.

\subsection{$\frac{\e_k}{8}$-homotopic Plateau problem}

Let $\ck{k}^k$ be the almost minimizing min-max sequence given by Proposition \ref{prop:almost_minimizing_min-max_sequences} and for every $k \in \N$ let $\om_k$ be the open subset of $\M$ which is compatible with $\ck{k}$. We fix $k \in \N$ such that $\ck{k}$ is ${\e_k}$-a.m. in $\an$ and we define the class of pairs $\hpp$.

\begin{defi}\label{def:hpp}
A pair $\c$ belongs to $\hpp$ if there exists a parametrized family of pairs of surfaces $\ct{t}_{t \in [0,1]}$ and a family of open sets $(\ot{t})_{t \in [0,1]}$ compatible with $\ct{t}$ that satisfy the following properties:
\begin{enumerate}
\item \label{pnt:1_hpp} $\ct{0}=\ck{k}$ and $\ct{1}=\c$;
\item \label{pnt:1.5_hpp} $\ct{t} = \ck{k}$ in $\M \- \an$ for every $t \in [0,1]$;
\item \label{pnt:4_hpp} $\fa \ct{t} \leq \fa\ck{k} +\frac{\e_k}{8}$ for every $t \in [0,1]$.
\end{enumerate}
\end{defi}
Roughly speaking, $\c$ belongs to $\hpp$ if it can be obtained deforming $\ck{k}$ in $\an$ without passing through a pair for which $\fa$ is larger than $\fa\ck{k} +\frac{\e_k}{8}$. We now consider the minimizing problem for $\fa$ in the class $\hpp$.

\begin{proposition}[$\frac{\e_k}{8}$-Homotopic Plateau problem]\label{prop:homotopic_plateau_problem}
Let $\ctk{j}{k}_j$ be a minimizing sequence of pairs in $\hpp$, that is
\begin{equation}
\lim_{j \to \infty} \fa \ctk{j}{k}
=
\inf \Big\{ \fa \c \mid \c \in \hpp \Big\}
\end{equation}
and let $\om_j^k$ be open sets compatible with $\ctk{j}{k}$.
Then, up to subsequences, the following conclusions hold:
\begin{enumerate}
\item $\ctk{j}{k}$ converge, as $j \to \infty$, to a pair $\pairk{k}$ which is stationary and stable for $\fa$ in $\an$;
\item $\pairk{k}$ is induced, in $\an$, by a pair of surfaces $\tk{k}$ which are smooth and minimize $\fa$ on small balls contained in $\an$;
\item $\om_j^k$ converge, as $j \to \infty$, to a set $\tok{k}$ of finite perimeter which is compatible with $\tk{k}$ in $\an$.
\end{enumerate}
\end{proposition}

\subsection{Proof of proposition \ref{prop:homotopic_plateau_problem}}

The proof is quite long and we divide it in several steps. Since $k$ is fixed, throughout the proof we drop superscripts $k$, thus we write $\c$ for $\ck{k}$, $\ct{j}$ for $\ctk{j}{k}$, $\om_j$ for $\om_j^k$, $\pair$ for $\pairk{k}$, $\tc$ for $\tk{k}$ and $\tom$ for $\tok{k}$.

First of all we notice that of course there exists a subsequence, not relabeled, $\ct{j}$ which converges to a pair $\pair$ in the sense of varifolds, since masses are bounded; moreover the open sets $\om_j$ compatible with $\ct{j}$ converge to $\tom$, which is of finite perimeter since the masses of the pairs $\ct{j}$ are uniformly bounded. We then have to prove that these objects have the properties stated in the proposition.

\subsubsection{Step 1: $\pairk{k}$ is stationary and stable for $\fa$ in $\an$}
We first prove that $\pair$ is stationary in $\an$ for $\fa$; indeed, let us assume by contradiction the existence of a vector field $X \in \Xt$ with compact support in $\an$ such that
\begin{equation}
\dfa \pair [X] = - \beta <0.
\end{equation}
Since $\ct{j}$ converge to $\pair$, by continuity of first variation there exists $J \in \N$ such that
\begin{equation}
\dfa \ct{j}[X] \leq - \frac{\beta}{2}
\qquad
\forall j \geq J.
\end{equation}
Since the first variation with respect to $X$ is continuous with respect to the varifold convergence, there exists $T>0$ such that
\begin{equation}
\dfa (\psi_t)_\sharp \ct{j}[X]
\leq
-\frac{\beta}{4}
\qquad
\forall t \in [0,T],
\forall j \geq J.
\end{equation}
Thus, for every $j \geq J$ we have
\begin{equation}
\fa (\psi_T)_\sharp \ct{j}
\leq
\fa \ct{j} + \int_0^T \dfa (\psi_t)_\sharp \ct{j}[X] \dif t
\leq
\fa \ct{j}
- \frac{T \beta}{4}
\end{equation}
which contradicts the minimality of the sequence $\ct{j}_j$ in $\hpp$.

By a similar argument, $\pair$ is also stable for $\fa$ in $\an$.

\subsubsection{Step 2: In small balls pairs with small $\fa$ are in $\hpp$}

We show that, if we consider a sufficiently small ball $B$ in $\an$ and a sufficiently large $j$, if $\tc$ coincides with $\ct{j}$ outside of $B$ and it is such that $\fa \tc \leq \fa \ct{j}$, then $\tc$ belongs to $\hpp$.

\begin{lemma}\label{lmm:local_balls_competitors}
Let us assume $y \in \an$; there exists $j_0 \in \N$ and $\rho_0>0$ such that $B_{\rho_0}(y) \subset \an$ with the following property: if $j> j_0$, $\rho< \rho_0$ and if $\tc$ is a smooth pair of surfaces such that
\begin{enumerate}
\item $\tc \- B_\rho(y) = \ct{j} \- B_\rho(y)$;
\item \label{pnt:3_local_balls_competitor} $\fa \tc < \fa \ct{j}$,
\end{enumerate}
then $\tc \in \hpp$.
\end{lemma}

\begin{proof}
Since the case $y \in \interior{\M}$ is covered by \cite[Lemma 4.1]{delellis2009existence}, we fix $y \in \an \cap \pM$.
Up to translating, we can assume that $y$ coincides with the origin $0$.

We first assume that there exists $R>0$ such that $B_R \cap \pM$ is flat, that is, up to rotating the coordinates, we assume that $B_R \cap \pM = \{(z_1,z_2,z_3) \in B_R \mid z_{3}\geq 0\}$.

\begin{enumerate}
\item
We begin by considering $j_0 \in \N$ and $\rho_0>0$ such that $4\rho_0 \leq R$ and $B_{4 \rho_0}\subset \an$. The exact values of $j_0$ and $\rho_0$ will be chosen at the end of the proof.

We choose $\rho< \rho_0$, $j>j_0$; since the sets of singularities $S_j, G_j$ of $\ct{j}$ are finite, by Sard's Lemma there exists $\tau \in (\rho,2\rho)$ such that each surface of the pair $\ct{j}$ intersects $\de B_\tau$ transversally and smoothly.

We now choose $\delta>0$ so that $2\delta < \min \{\tau-\rho, 2\rho -\tau\}$, and we consider a monotone non-decreasing smooth function $\gamma: [0,+\infty) \to [0,+\infty)$ such that
\begin{itemize}
\item $\gamma(t) = t$ if $t \in [0,\tau- 2\delta] \cup [\tau+2\delta, +\infty)$;
\item $\gamma(t) = \tau$ if $t \in [\tau-\delta, \tau+\delta]$.
\end{itemize}

We recall, by \eqref{eq:dilation_map}, that
\begin{equation}
\dil{0,\kappa}: z \in \R^{3} \longmapsto \frac{1}{\kappa} z \in \R^{3},
\end{equation}
and, for every $t \in [0,1]$ and every $\lambda \geq 0$, we define $\ct{j,t}$ as follows:
\begin{equation}
\ct{j,t} \cap \partial B_\lambda = \dil{0,[t\gamma(\lambda) + (1-t)\lambda]/\lambda} \Bigl(\ct{j} \cap \partial B_{t\gamma(\lambda) + (1-t)\lambda} \Bigr).
\end{equation}
Roughly speaking, we are stretching $\ct{j}$ in a neighbourhood of $\de B_\tau$ and deforming it in a cone.

Whatever is $\e>0$, if $\delta$ is chosen sufficiently small, we have that
\begin{gather}
\fa \ct{j,t} \leq \fa \ct{j} + \varepsilon
\qquad
\forall t \in [0,1].
\end{gather}
Since $\tc = \ct{j}$ outside $B_\rho$, the same modification can be done on $\tc$ with the same estimates on $\fa$.
This allows us to assume, without loss of generality, that $\ct{j}$ (and thus $\tc$, since it coincides with $\ct{j}$ outside $B_\rho$) is a cone in $\an(0,\tau-\delta,\tau+\delta)$.

\item
We now define $C, K$ as the cones with base $\s_j \cap \de B_\tau$ and $\g_j \cap \de B_\tau$ and vertex $0$, that is:
\begin{equation}
\begin{gathered}
C \cap \partial B_\lambda
=
\dil{0,\t/\lambda} \Bigl( \s_j \cap \de B_\t \Bigr)
\qquad
\forall \lambda > 0,
\\
K \cap \partial B_\lambda
=
\dil{0,\t/\lambda} \Bigl( \g_j \cap \de B_\t \Bigr)
\qquad
\forall \lambda > 0.
\end{gathered}
\end{equation}

We can now construct the homotopy beetween $\ct{j}$ and $\tc$, that is the family of pairs $\{\ct{j,t}\}_{t \in [0,1]}$ such that $\ct{j,0}= \ct{j}$ and $\ct{j,1} = \tc$; we define it in the following way:
\begin{itemize}
\item $\ct{j,t} \- B_\t = \ct{j} \- B_\t$ for every $t \in [0,1]$;
\item $\ct{j,t} \cap \an(0,|1-2t|\t,\t) = (C,K) \cap \an(0,|1-2t|\t,\t)$ for every $t \in [0,1]$;
\item $\ct{j,t} \cap B_{(1-2t)\t} = \dil{0,1/(1-2t)}\Bigl(\ct{j} \cap B_\t \Bigr)$ for $t \in [0,\frac{1}{2}]$;
\item $\ct{j,t} \cap B_{(2t-1)\t} = \dil{0,1/(2t-1)}\Bigl(\tc \cap B_\t \Bigr)$ for $t \in [\frac{1}{2},1]$.
\end{itemize}
The family of open sets compatible with $\ct{j,t}$ is defined in the same way.

It is clear that this family of pair of surfaces satisfies conditions \ref{pnt:1_hpp} and \ref{pnt:1.5_hpp} of definition \ref{def:hpp}; we have only to check that condition \ref{pnt:4_hpp} is satisfied.

\item
By coarea formula, we have
\begin{equation}
\begin{gathered}
\hn(\s_j \cap \an(0,\rho,2\rho))
=
\int_\rho^{2\rho} \haus{1} \big(\s_j \cap \partial B_\lambda \big) \dif \lambda,
\\
\hn(\g_j \cap \an(0,\rho,2\rho))
=
\int_\rho^{2\rho} \haus{1} \big(\g_j \cap \partial B_\lambda \big) \dif \lambda.
\end{gathered}
\end{equation}

By Chebyshev inequality we have
\begin{equation}
\begin{gathered}
\haus{1}\Bigl(
\Bigl\{
\lambda \in [\rho,2\rho] \mid
\haus{1}(\s_j \cap \partial B_\lambda)
\geq
\frac{3 \hn(\s_j \cap \an(0,\rho,2\rho))}{\rho}
\Bigr\}
\Bigr)
\leq
\frac{\rho}{3},
\\
\haus{1}\Bigl(
\Bigl\{
\lambda \in [\rho,2\rho] \mid
\haus{1}(\g_j \cap \partial B_\lambda)
\geq
\frac{3 \hn(\g_j \cap \an(0,\rho,2\rho))}{\rho}
\Bigr\}
\Bigr)
\leq
\frac{\rho}{3}.
\end{gathered}
\end{equation}
Since by Sard's Lemma the set of $\lambda \in [\rho,2\rho]$ such that $\s_j$  and $\g_j$ intersect $\de B_\lambda$ transversally has full measure, we can choose $\t \in [\rho,2\rho]$ such that $\s_j \cap \de B_\t$ and $\g_j \cap \de B_\t$ are smooth $(1)$-manifolds and
\begin{equation}
\begin{gathered}
\haus{1}(\s_j \cap \de B_\tau)
\leq
\frac{3 \hn(\s_j \cap \an(0,\rho,2\rho))}{\rho},
\\
\haus{1}(\g_j \cap \de B_\tau)
\leq
\frac{3 \hn(\g_j \cap \an(0,\rho,2\rho))}{\rho}.
\end{gathered}
\end{equation}
Thus we have an estimate for the area of $C \cap \an(y,|1-2t|\t,\t)$ and $K \cap \an(y,|1-2t|\t,\t)$:
\begin{equation} \label{eq:estimate_1_local_balls}
\begin{split}
\hn(C \cap \an(y,|1-2t|\t,\t))
&=
\int_{|1-2t|\t}^\t \haus{1}(K \cap \de B_\lambda)  \dif \lambda
\\
&=
\haus{1}(\s_j \cap \partial B_\t)
\int_{|1-2t|\t}^\t \Bigl(\frac{\lambda}{\t}\Bigr) \dif \lambda
\\
&\leq
6 \hn(\s_j \cap B_{2\rho}).
\end{split}
\end{equation}
In a similar way we get
\begin{equation}
\hn(K \cap \an(y,|1-2t|\t,\t))
\leq
6 \hn(\g_j \cap B_{2\rho}).
\end{equation}
If $t \in [0,\frac{1}{2}]$ we have
\begin{equation}\label{eq:estimate_2_local_balls}
\begin{gathered}
\hn(\s_{j,t} \cap B(0,|1-2t|\tau))
=
|1-2t|^2 \hn(\s_j \cap B_\t)
\leq
\hn(\s_j \cap B_{2\rho})
\\
\hn(\g_{j,t} \cap B(0,|1-2t|\tau))
=
|1-2t|^2 \hn(\g_j \cap B_\t)
\leq
\hn(\g_j \cap B_{2\rho})
\end{gathered}
\end{equation}
If $t \in [\frac{1}{2},1]$ we similarly have
\begin{equation}
\begin{gathered}
\hn(\s_{j,t} \cap B(0,|1-2t|\tau))
=
|1-2t|^2 \hn(\ts{} \cap B_\t)
\leq
\hn(\ts{} \cap B_{2\rho})
\\
\hn(\g_{j,t} \cap B(0,|1-2t|\tau))
=
|1-2t|^2 \hn(\tg{} \cap B_\t)
\leq
\hn(\tg{} \cap B_{2\rho})
\end{gathered}
\end{equation}
%

\item
Gathering the estimates \eqref{eq:estimate_1_local_balls} and \eqref{eq:estimate_2_local_balls}, we obtain
\begin{equation}\label{eq:estimate_k-homotopic}
\hn(\s_{j,t} \cap B_\t) + a \hn(\g_{j,t} \cap B_\t)
\leq
7 \Big[ \hn(\s_{j} \cap B_{2\rho}) + a \hn(\g_{j} \cap B_{2\rho})
\Big]
\end{equation}
We now recall that $\ct{j}$ converges to $\pair$ ($=\pairk{k}$) which is stationary in $\an$. By Proposition \ref{prop:monotonicity_formula_infty}, we have that, for a suitable constant $c>0$
\begin{equation}
\norm{V}(B_\lambda) + a \norm{W}(B_\lambda)
\leq
c \lambda^2,
\end{equation}
By varifold convergence $\ct{j} \to \pair$, up to choosing the radius $\rho_0$ such that $\norm{V}(\de B_{2\rho_0}) + a \norm{W}(\de B_{2\rho_0}) = 0$, there exists $j_0 \in \N$ such that for all $j >j_0$ and $\rho< \rho_0$
\begin{equation}
\begin{split}
\hn(\s_j \cap B_{2\rho}) + a \hn(\g_j \cap B_{2\rho})
&\leq
\hn(\s_j \cap B_{2\rho_0}) + a \hn(\g_j \cap B_{2\rho_0})
\\
& \leq
2 \Big[ \norm{V}(B_{2\rho_0}) + a \norm{W}(B_{2\rho_0}) \Big]
\\
&\leq
2^{3} c \rho_0^2.
\end{split}
\end{equation}

If $\rho_0$ is chosen such that $7 \cdot 2^{3} c \rho_0^2 < \frac{\e_k}{8}$, we obtain
\begin{equation}
\fa \ct{j,t}
\leq
\fa \ct{j} + \frac{\e_k}{8}
\qquad
\forall t \in \I;
\end{equation}
Thus, with these choices of $j_0$ and $\rho_0$, $\tc \in \hpp$.

\item
If $\pM$ is not flat in a neighborhood of $y$, there exists a ball $B:=B_R(y)$, a neighbourhood $W$ of the origin and a diffeomorphism $\Phi: B \to W$ such that
\begin{equation}
\Phi(y)=0
\qquad
\Phi(B \cap \M)
=
\{(z_1,\dots,z_{3}) \in W \mid z_{3}\geq 0\}.
\end{equation}
By regularity of $\Phi$, there exists a constant $b>0$ such that
\begin{gather}
\frac{1}{b}\haus{k}(A)
\leq
\haus{k}(\Phi(A))
\leq
b \haus{k}(A)
\qquad
\forall A \subset B_R(y),\,\, k\in\{1,\dots,3\}
\label{eq:measure_estimate_boundary}
\\
\frac{1}{b}|z-w|
\leq
|\Phi(z) - \Phi(w)|
\leq
b |z-w|
\qquad
\forall z,w \in B_R(y).
\label{eq:lenght_estimate_boundary}
\end{gather}
We can repeat all the arguments made above and, up to adjusting the constants, obtain the thesis.
\end{enumerate}
\end{proof}

\subsubsection{Step 3: $\tok{k}$ locally minimizes $\fa$ in $\an$}
\begin{lemma}\label{lmm:minimality_tom_boundary}
Let us assume $y \in \M$ and $0<\rho <\rho_0$, where $\rho_0$ is given by Lemma \ref{lmm:local_balls_competitors}; then, for every open set $\Xi$ of finite perimeter such that $\Xi \- B_{\rho/2}(y)= \tom \- B_{\rho/2}(y)$, we have
\begin{equation}
\fa (\tom)
\leq
\fa(\Xi).
\end{equation}
\end{lemma}
\begin{proof}
If $y \in \interior{\M}$, then the lemma is proved by showing that $\tom$ locally minimizes the perimeter in a small ball with center $y$; the proof is the same as the one of \cite[Lemma 4.2]{delellis2009existence}, then we can assume that $y \in \pM$.

Let us assume by contradiction that there exists $\eta >0$ and an open set $\Xi$ of finite perimeter such that $\Xi \- B_{\rho/2}(y)= \tilde{\om} \- B_{\rho/2}(y)$ and
\begin{equation}\label{eq:assumption_contradiction_minimality}
\fa(\Xi) < \fa(\tilde{\om}) - \eta.
\end{equation}
We want to use this assumption to construct a sequence of competitors for $\fa$ which is better than $\om_j$.

Since $\ind_{\om_j} \to \ind_{\tilde{\om}}$ in $L^1$, by coarea formula the sequence of function
\begin{equation}
\psi_j
\colon
t
\mapsto
\hn \bigl( (\om_j \Delta \tilde{\om}) \cap \partial B_t\bigr)
\end{equation}
converge to $0$ in $L^1([0,\rho])$ as $j \to \infty$. Hence, up to subsequences, we have
\begin{equation}\label{eq:fubini_sphere_l1}
\hn \bigl( (\om_j \Delta \tilde{\om}) \cap \partial B_t\bigr)
\to
0
\quad
\mbox{ for a.e. } t \in [0,\rho].
\end{equation}
We define, for each $j \in \N$,
\begin{equation}
\Xi_j
=
\bigr( \om_j \setminus B_\tau \bigr)
\cup
\bigr( \Xi \cap B_\tau \bigr),
\end{equation}
where $\t$ has to be chosen in $\Big(\frac{\rho}{2}, \rho \Big)$.
%
%
%
%
We have
\begin{equation}
\begin{split}
\fa( \Xi_j)
\leq &
\hn\big((\partial \om_j \setminus B_\tau) \cap \interior{\M} \big) +
\per(\Xi, \interior{B_\tau} \cap \interior{\M})+
\hn \bigl( (\om_j \Delta \tilde{\om}) \cap \partial B_\tau \bigr)
\\
& +
a \hn\big((\partial \om_j \setminus B_\tau) \cap \pM \big)
+
a \hn \big( \des \Xi \cap B_\t \cap \pM \big)
\end{split}
\end{equation}
If we choose $\tau$ such that \eqref{eq:fubini_sphere_l1} holds true and such that the perimeter measure of $\om_j$ and $\tom$ do not charge $\de B_\t$, we have that
\begin{equation}\label{eq:contradiction_minimality_Xi}
\begin{split}
\limsup_{j \to \infty}
\,
\bigl(
\fa(\Xi_j)-\fa(\om_j)
\bigr)
\leq &
\limsup_{j \to \infty}
\,
\Bigl[
\per(\Xi, \interior{B_\tau} \cap \interior{\M})+
\hn \bigl( (\om_j \Delta \tilde{\om}) \cap \partial B_\tau \bigr)
\\
& \quad
 + a \hn \big( \des \Xi \cap B_\t \cap \pM \big)
\\
& \quad
 - \per(\om_j, \interior{B_\tau} \cap \interior{\M})
 - a \hn \big( \des \om_j \cap B_\t \cap \pM \big)
\Bigr]
\\
\leq &
\per(\Xi, \interior{B_\tau} \cap \interior{\M})
+ a \hn \big( \des \Xi \cap B_\t \cap \pM \big)
\\
& \quad
- \liminf_{j \to \infty}
\Bigl[
\per(\om_j, \interior{B_\tau} \cap \interior{\M})
 + a \hn \big( \des \om_j \cap B_\t \cap \pM \big)
\Bigr]
\\
\leq &
\per(\Xi, \interior{B_\tau} \cap \interior{\M})
+ a \hn \big( \des \Xi \cap B_\t \cap \pM \big)
\\
& \quad
- \Bigl[
\per(\tom, \interior{B_\tau} \cap \interior{\M})
+ a \hn \big( \des \tom \cap B_\t \cap \pM \big)
\Bigr]
\\
\leq &
-\eta,
\end{split}
\end{equation}
where the second inequality holds by \eqref{eq:fubini_sphere_l1}, the third by lower semicontinuity of the capillarity functional under $L^1$ convergence of sets, see \cite[Proposition 19.1]{maggi_2012} and the last inequality is given by the assumption \eqref{eq:assumption_contradiction_minimality}.

$\Xi_j$ is not an admissible candidate, since it is not necessarily smooth in $B_{\rho/2}$, hence we have to smooth up these set without increasing $\fa$ too much.

By minimality of the sequence $\{\partial \om_j\}_j$ for the homotopic Plateau problem, we choose $j_0 \in \N$ such that Lemma \ref{lmm:local_balls_competitors} holds and so that
\begin{equation}\label{eq:minimality_homotopic_competitor}
\fa(\om_j)
\leq
\inf \{ \fa\ct{} \mid \ct{} \in \hpp\}
+ \frac{\eta}{8}
\qquad
\forall j \geq j_0.
\end{equation}
We choose an interval $(a,b) \cc (\tau, \rho)$ such that $\bi \om_j$ and $\bb \om_j$ are smooth in $\an(y,a,b)$.
We can choose $V \cc \an(y,a,b)$ so that there exists $d>0$, $(a',b'$ such that
\begin{equation}
V
=
\{(\xi,\zeta) \mid |\xi|\in (a',b'), \zeta \in [0,d]\},
\end{equation}
where $\xi, \zeta$ are normal coordinates in a tubular neighborhood of $\pM$; roughly speacking, in the normal coordinates, $V$ is a ``tube" with base on $\partial \M$.

We call $\gamma = \partial \om_j \cap \de\M \cap V$.
With a small deformation supported in $\an(y,a,b)$ we can ``push-in" $\gamma$ in $\inter(\M)$ in $V$, that is we can assume that, in the normal coordinates $\xi, \zeta$,
\begin{equation}
\partial \om_j \cap V
=
\gamma \times [0,d].
\end{equation}
Moreover $d$ and $(a',b')$ can be chosen small enough so that the possible gain of $\fa$ is as small as we want, and then we can assume, without loss of generality, that $\om_j$ is of this form.

We now extend every $\ind_{\Xi_j}$ to a function $f_j \in BV(\R^{3})$ such that $|Df_j|(\pM)=0$.
We now call $\varphi_\varepsilon$ the standard mollifier and
\begin{equation}
g_{j,\varepsilon} := \ind_{\Xi_j}*\varphi_\varepsilon.
\end{equation}
For every $j > j_0$, by standard approximation procedure \cite[Theorem 13.8]{maggi_2012}, for almost every $t \in (0,1)$ we have that the set $\Delta_{j,\e} := \{g_{j,\e} > t\} \cap \M$ is an open subset of $\M$ which is smooth in $\M$ and satisfies
\begin{equation}
\Delta_{j, \e} \tto{L^1_{\text{loc}}(\M)}{\e \to 0} \ind_{\Xi_j},
\qquad
\per \big(\Delta_{j,\e}, \interior{\M}\big) \tto{}{\e \to 0} \per\big( \Xi_j, \interior{\M} \big).
\end{equation}
Since $\fa$ is continuous under strict convergence of $BV$-functions, we obtain
\begin{equation}
\lim_{\e \to 0} \fa(\Delta_{j,\e})
=
\fa(\Xi_j).
\end{equation}
Hence for every $j>j_0$ there exists $\bar{\e}_j >0$ such that the set $\Delta_{j,\e_j}$ satisfies
\begin{equation}
\fa(\Delta_{j,\e}) \leq \fa(\Xi_j) + \frac{\eta}{4}
\qquad
\forall \e \in (0,\bar{\e}_j)
\forall j > j_0.
\end{equation}
Since $\Delta_{j,\e}$ does not coincide with $\om_j$ outside $B_a(y)$, we have to construct a new sequence of sets $\Psi_j$ that ``connect" $\Delta_{j,\e}$ inside $B_a(y)$ with $\om_j$ outside $B_b(y)$ with a small possible increase of $\fa$. This is possible since both $\om_j$ and $\Delta_{j,\e}$ are smooth within $\an(y,a,b)$.

We fix $j>j_0$ a tubular neighbourhood $U$ of $\partial \om_j$ with normal coordinates $(w,z)$;
since by mollification $\de \Delta_{j,\e}$ converges smoothly to $\de \om_j$ in $\an(y,a,b)$ as $\e \to 0$, there exist functions $\psi_{j,\varepsilon}$ such that
\begin{equation}
\om_j \cap U= \{ (w,s) \mid s < 0\},
\qquad
\Delta_{j,\varepsilon} \cap U = \{(w,s) \mid s < \psi_{j,\varepsilon}(w)\}
\quad
\mbox{ in } \an(y,a,b)
\end{equation}
and $\psi_{j,\varepsilon} \to 0$ smoothly as $\varepsilon \to 0$.

To join smoothly the two sets we choose a function $\chi_A \in C_c^{\infty}(B_b)$ such that $0 \leq \chi_A \leq 1$ and $\chi_A \equiv 1$ in a neighbourhood of $B_a$ and we construct a new family of open sets $\T_{j,\varepsilon}$:
\begin{itemize}
\item $\T_{j,\varepsilon}\cap B_a = \Delta_{j,\varepsilon} \cap B_a$;
\item $\T_{j,\varepsilon} \cap \an(y,a,b) = \{(w,s) \mid s < \chi_A(w) \psi_{j,\varepsilon(w)}\}$;
\item $\T_{j,\varepsilon} \setminus B_b = \om_j \setminus B_b$.
\end{itemize}
By construction $\bi \T_{j,\varepsilon}, \bb \T_{j,\varepsilon}$ are smooth outside of a finite set.

By smooth convergence $\psi_{j,\varepsilon} \tto{}{\e \to 0} 0$, for every $j >j_0$ there exists $\e_j \in (0,\bar{\e}_j)$ such that
\begin{equation}
\fa(\T_{j,\e_j})
\leq
\fa(\Xi_j) + \frac{\eta}{2}
\leq
\fa(\om_j) - \frac{\eta}{4}
\end{equation}
if $j$ is sufficiently large, by \eqref{eq:contradiction_minimality_Xi}. Since by Lemma \ref{lmm:local_balls_competitors} $\T_{j,\e_j} \in \hpp$, this is a contradiction with the minimality of the sequence $\om_j$ in $\hpp$ for $\fa$.
%
%
%
%
%

\end{proof}

\subsubsection{Step 4: $\pairk{k}$ is compatible with $\tom^k$}

We have now to show that the pair $\pair$ is induced by $\bi \tom$ and $\bb \tom$. By \cite[Proposition A.1]{delellis2009existence}, we have only to prove that
\begin{equation}
\lim_{j \to \infty} \hn(\s_j)
=
\hn( \bi \tom),
\qquad
\lim_{j \to \infty} \hn( \g_j)
=
\hn(\bb \tom).
\end{equation}
We first observe that, since $\fa$ is continuous under varifolds convergence of pairs, we have
\begin{equation}
\lim_{j \to \infty} \fa \ct{j}
=
\fa \pair
\end{equation}
Moreover, since $\fa$ is lower semicontinuous under $L^1$-convergence of open sets, it holds
\begin{equation}
\qquad
\fa(\tom)
\leq
\lim_{j \to \infty} \fa \ct{j}.
\end{equation}
The above inequality cannot be strict, because otherwise there would exist a sufficiently small ball $B$ such that the same inequality holds and we can argue as the proof of Lemma \ref{lmm:minimality_tom_boundary}, defining new competitors $\Xi_j = \bigr( \om_j \setminus B_\tau \bigr) \cup \bigr( \tom \cap B_\tau \bigr)$ which would contradict the minimality of sequence $\om_j$. Hence
\begin{equation}\label{eq:fa_tom=fa_pair}
\fa( \tom)
=
\lim_{j \to \infty} \fa \ct{j}
=
\fa \pair.
\end{equation}
Since $\ind_{\om_j} \to \ind_{\tom}$ in $L^1(\ru)$, we have
\begin{equation}\label{eq:semicontinuity_perimeter_hpp}
\begin{gathered}
\per(\tom,\M)
\leq
\liminf_{j \to \infty} \per(\om_j, \M);
\\
\hn(\bi \tom)
=
\per(\tom, \interior{\M})
\leq
\liminf_{j \to \infty} \per(\om_j, \interior{\M})
=
\norm{V};
\\
\hn(\bb \tom)
=
\per(\tom, \pM)
\geq
\limsup_{j \to \infty} \per(\om_j, \pM)
=
\norm{W}.
\end{gathered}
\end{equation}
Let us call
\begin{equation}
\eta:= \norm{V} - \hn(\bi \tom),
\qquad
\z := \hn(\bb \tom) - \norm{W}.
\end{equation}
\eqref{eq:semicontinuity_perimeter_hpp} yields $\eta, \z \geq 0$. The statement is proved if we show that $\eta=\z=0$. To do so, we first compute
\begin{equation}
\begin{split}
\eta - a \z
= &
\norm{V} + a \norm{W}
-\big(
\hn(\bi \tom) + a \hn(\bb \tom)
\big)
\\
= &
\fa \pair - \fa(\tom)
\\
\stackrel{\eqref{eq:fa_tom=fa_pair}}{=} &
\,0,
\end{split}
\end{equation}
On the other hand, one has
\begin{equation}
\begin{split}
\eta - \z
=&
\norm{V}+ \norm{W}
- \big(
\hn(\bi \tom) + \hn(\bb \tom)
\big)
\\
=&
\lim_{j \to \infty} \Big( \hn(\s_j) + \hn(\g_j) \Big)
-
\per(\tom, \M)
\\
=&
\lim_{j \to \infty} \per(\om_j, \M) - \per(\tom, \M)
\\
\geq &
0
\end{split}
\end{equation}
gathering the previous two equations we get that $\z = 0$, hence $\eta=0$, as desired.

\subsubsection{Step 5: $\bi\tom^k$ and $\bb \tom^k$ are regular surfaces in $\an$}

By the minimality of $\tom$ for $\fa$ in small balls, we use use the regularity theorems for minimizers of capillarity functional\cite[Theorem 1.2 and Corollary 1.4]{DePMag14} to obtain that $V$ and $W$, which are induced by $\bi \tom$ and $\bb \tom$, coincide in $\an$  with surfaces
$\ts{}, \tg{}$ that are smooth.

\subsection{Construction of replacements: proof of theorem \ref{prop:existence_replacements} }

\begin{proof}[Proof of the theorem \ref{prop:existence_replacements}]
\begin{enumerate}
\item
Since every $\pairk{k}$ is a stationary stable pair in $\an$, by Theorem \ref{thm:compactness_stable_capillarity} we have that the sequence of pairs $\pairk{k}$ converges, up to subsequences, to a pair $\tpair$ which is induced by a pair of regular surfaces in $\an$ which are stationary and stable for $\fa$ in $\an$.

We select a diagonal sequence $\tk{k}^k = (\bi \om_{j(k)}^k, \bb \om_{j(k)}^k)$ (where for every $k \in \N$ $\{\om_j^k\}_j$ is a minimizing sequence for the $\frac{\e_k}{8}$-homotopic Plateau problem $\hpp$) which satisfies
\begin{equation}
\fa \tk{k}
\leq
\fa \ck{k}
\qquad
\forall k \in \N
\end{equation}
and such that $\tk{k} \to \tpair$.
Since every $\tk{k} \in \hpp$, we have
\begin{equation}\label{eq:new_sequence_min-max}
\fa \ck{k} - \e_k
\leq
\fa \tk{k}
\leq
\fa \ck{k},
\end{equation}
we have that $\tk{k}^k$ is a new min-max sequence for $\fa$.

\item
We fix $\eta>0$ and $\an'= \an(x,r-\eta,s + \eta)$ such that $\an' \in \anul{r(x)}{x}$.

\begin{itemize}
\item We set $\tilde{r}(x) = r(x)$; $\tk{k}$ is $\e_k$-almost minimizing in every annulus in $\anul{r(x)}{x}$; indeed, if there exists $\an'' \in \anul{r(x)}{x}$ such that $\tk{k}$ is not $\e_k$-almost minimizing in $\an''$, since for every $k$, $\tk{k}$ is obtained by deforming $\ck{k}$ in $\an$ without passing through a pair with $\fa$ larger than $\fa\ck{k} + \e_k/8$, this would imply that also $\ck{k}$ is not almost minimizing in an annulus $\an''' \in \anul{r(x)}{x}$ such that $\an''' \supset \an \cup \an''$ and this is a contradiction.

\item if $y \in \an'$, we set $\tilde{r}(y) = \min \{r(y), \dist(y,\partial \an')\}$; In this case $\tk{k}$ is almost minimizing in every annulus belonging to $\anul{\tilde{r}(y)}{y}$ because such an annulus is contained in $\an'$ and $\tk{k}$ is almost minimizing in $\an'$.

\item if $y \notin \an'$, we set $\tilde{r}(y) = \min \{r(y), \dist(y,\an)\}$. If $\an'' \in \anul{\tilde{r}(y)}{y}$ then $\an'' \subset \M \- \an$; since $\tk{k}$ and $\ck{k}$ coincide outside $\an$ and since $\ck{k}$ is almost minimizing in every annulus in $\anul{r(y)}{y}$, we obtain that $\tk{k}$ is a.m. in $\an''$.
\end{itemize}

\item We are left to prove that $\tpair$ is a replacement for $\pair$ in $\an=\an(x,r,s)$. Since $\tk{k}^k$ is a min-max sequence, we have
\begin{equation}
\fa \tpair
=
\fa \pair.
\end{equation}
By construction, it follows that $\tpair$
is stationary in $\an$ and in $\M \- \an$. We have only to prove that it is stationary in $\M$.
Let us assume, by contradiction, that $\tpair$ is not stationary. We fix $\eta>0$, a vector field $X \in \Xt$ such that $\dfa \tpair [X] \leq - \eta$
and we choose a partition of unity for the cover $\{\an', \M\- \an\}$ of $\M$: $\varphi \in C_c^{\infty}(\an')$ and $\psi \in C_c^{\infty}(\M\- \an)$. Hence
\begin{equation}
\dfa \tpair [X]
=
\dfa \tpair [\varphi X]
+
\dfa \tpair [\psi X].
\end{equation}
Since $\tpair$ is stationary in $\M\- \an$, it follows that $\dfa \tpair [\psi X]=0$. Thus
\begin{equation}
\dfa \tpair [\varphi X]
\leq
- \eta
\end{equation}
By continuity of first variation there exists $T>0$ such that, if $\Phi_t$ is the flow of $\varphi X$, then
\begin{equation}
\dfa (\Phi_t)_\sharp \tpair [\varphi X]
\leq
- \frac{\eta}{2}
\qquad
\forall t \in [0,T].
\end{equation}
Since $\tk{k} \tto{}{k} \tpair$, by continuity of first variation again, there exists $\bar{k} \in \N$ such that
\begin{equation}
\dfa (\Phi_t)_\sharp \tk{k} [\varphi X]
\leq
- \frac{\eta}{4}
\qquad
\forall t \in [0,T],
\forall k > \bar{k}.
\end{equation}
Hence
\begin{equation}
\fa \bigl((\Phi_t)_\sharp \tk{k} \bigr)
\leq
\fa \tk{k}
-t \,\frac{\eta}{4}
\qquad
\forall t \in [0,T],
\,\,\,
\forall k > \bar{k}.
\end{equation}
Thus we have a smooth deformation of $\tk{k}$ with support in $\an'$ which lowers $\fa$ of a fixed amount.
But this contradicts, for $k$ large enough, the almost minimizing property of $\tk{k}$ in $\an'$.
\end{enumerate}
\end{proof}

\begin{rem}
Since $\tk{k}$ is a new min-max sequence which is almost minimizing in each annulus in $\anul{\tilde{r}(y)}{y}$, we can repeat the previous arguments to construct a new replacement of $\tpair$ in each annulus in $\anul{\tilde{r}(y)}{y}$. In particular, since $\tilde{r}(x)= r(x)$, we can construct a again a new min-max sequence and a new replacement of $\tpair$ in every annulus in $\anul{r(x)}{x}$.
This procedure can be in fact repeated an infinite number of times.
\end{rem}

\section{Regularity of $V$}\label{sec8}
In this final section we prove that  \(\pair\) is indeed induced by a smooth pair \((\Sigma, \Gamma)\), thus competing the proof of Theorem \ref{thm:main}
\subsection{Integer rectifiability of $V$}

\begin{proposition}
Let $\pair$ be the pair of varifolds given by Proposition \ref{prop:almost_minimizing_min-max_sequences}. Then $V$ and $W$ are rectifiable and $V$ is integer rectifiable. Moreover, for any $x \in \pM \cap \supp\norm{V}$, any tangent pair $\cpair$ to $\pair$ at $x$ is such that $C$ and $K$ are half-planes and $C$ meets $T_x \pM$ with angle $\theta$.
\end{proposition}

\begin{proof}
We divide the proof in few steps:

\medskip
\noindent
\(\bullet\)\emph{$V$ and $W$ are rectifiable}

If $p \in \interior{\M}$, then there exists a neighborhood $U$ of $p$ such that $V$ is integer rectifiable in $U$, see \cite[Lemma 5.2]{delellis2009existence}.
So we fix a point $x \in \partial \M$ such that $\T_2(\sumpair,x)>0$ and choose a sequence $r_j \downarrow 0 $; the rescaled pairs $\paira{j}:=(\dil{x,r_j})_{\sharp} \pair$ converge to a pair $\cpair$ such that
\begin{enumerate}
\item $\supp C \subset T_x\M$ and $\supp K \subset T_x \pM$;
\item $\dfa \cpair [X]=0$ for every $X \in \X_t(T_x \M)$;
\item $\rho^{-2} \big(\norm{C}(B_\rho) + a \norm{K}(B_\rho)\big) = \sigma^{-2} \big(\norm{C}(B_\sigma) + a \norm{K}(B_\sigma)\big) = \td(\sumpair,x)$ for every $\rho, \sigma >0$.
\end{enumerate}
All these statements follow by the fact that, by definition, the varifold $V + aW$ has free boundary (Definition \ref{def:varifold_free_boundary}) and it  is indeed  stationary for the mass with respect to all variations tangent to $T_x \pM$, see for example the proof of \cite[Lemma 5.4]{demasi2021rectifiability}.


Since $\pair$ has replacement $\tpairk{j}$ in every annulus $\an(x,r_j,2r_j)$ for $j$ sufficiently large, by Theorem \ref{thm:compactness_stable_capillarity} the rescaled pairs $(\dil{x,r_j})_\sharp\tpairk{j}$ converge to a pair $\tcpair$ which is  a replacement for $\cpair$ in $\an(0,1,2)$.
By the properties of replacements, we have
\begin{equation}
\begin{split}
\frac{\norm{\tilde{C}}(B_{1/2}) + a \norm{\tilde{K}}(B_{1/2})}{(1/2)^2}
= &
\frac{\norm{{C}}(B_{1/2}) + a \norm{{K}}(B_{1/2})}{(1/2)^2}
\\
= &
\frac{\norm{{C}}(B_{5/2}) + a \norm{{K}}(B_{5/2})}{(5/2)^2}
\\
= &
\frac{\norm{\tilde{C}}(B_{5/2}) + a \norm{\tilde{K}}(B_{5/2})}{(5/2)^2}.
\end{split}
\end{equation}
By the monotonicity identity for varifolds with free boundary (see e.g. Step 5 in the proof of \cite[Lemma 5.4]{demasi2021rectifiability}) one gets
\begin{equation}\label{eq:equality_density_ratios_replacements}
\frac{ \norm{\tilde{C}}(B_\rho) + a \norm{\tilde{K}}(B_\rho)}{\rho^{2} }
=
\frac{\norm{\tilde{C}}(B_\sigma) + a \norm{\tilde{K}}(B_\sigma)}{\sigma^{2} }
=
\td(\sumpair,x)
\qquad
\forall \rho, \sigma >0.
\end{equation}
This means that for every $\rho>0$ we have $\supp (\norm{\tilde{C}}+a\norm{\tilde{K}}) \cap \de B_\rho \neq \emptyset$, because otherwise \eqref{eq:equality_density_ratios_replacements} does not hold for $\rho-\e, \rho+\e$ for some small $\e>0$.

For any $y \in \de B_{3/2} \cap \supp (\norm{\tilde{C}}+a\norm{\tilde{K}})$, since $\tilde{C}$ and $\tilde{K}$ are regular in $\an(0,1,2)$, it follows that $\td(\norm{\tilde{C}}+a\norm{\tilde{K}},y) \geq 1/2$. By monotonicity of density ratios for $\tilde{C}+a\tilde{K}$ given by Proposition \ref{prop:monotonicity_formula_infty}, we have that there exists an universal constant $c>0$ such that
\begin{equation}
{\norm{\tilde{C}}(B_{1/2}(y))+ a \norm{\tilde{K}}(B_{1/2}(y))}
\geq
c.
\end{equation}
\eqref{eq:equality_density_ratios_replacements} and $B_{1/2}(y) \subset B_2$ give the existence of an universal constant $c'>0$ such that
\begin{equation}
\td(\sumpair,x) > c'
\qquad
\forall
x \in \supp(\norm{V} + a\norm{W}).
\end{equation}
Since $V + a W$ is stationary with respect to $\Xt$, \cite[Theorem 1.1, (1.6)]{demasi2021rectifiability} implies that it has bounded first variation in the whole space $\ru$ with a uniform bound on the first variation given by the mass $(\norm{V}+a\norm{W})(\M)$. This, together with the above uniform lower bound on the density and the Allard Rectifiability Theorem \cite[Theorem 42.4]{simon:lectures}, proves that $V+aW$ is a rectifiable varifold.
Since $W$ is rectifiable because it is the weak-* limit of a sequence of $L^{\infty}$ functions on $\pM$, we have that $V$ is rectifiable.

\medskip
\noindent
\(\bullet\)\emph{$C+aK$ is a rectifiable cone}

We claim that every tangent pair $\cpair$ to $\pair$ is such that $C+aK$ is a rectifiable cone. Indeed $C+aK$ is obtained as limit of rescalings $(\dil{x,r_j})_\sharp (V+aW)$ which are stationary for $\X_t\big(\dil{x,r_j}(\M)\big)$; by \cite[Theorem 1.1 and (1.6)]{demasi2021rectifiability}, we have that $(\dil{x,r_j})_\sharp (V+aW)$ have uniformly bounded first variation, then we can apply the Compactness Theorem for rectifiable varifolds \cite[Theorem 42.7]{simon:lectures}, to get the rectifiability of $C+aK$. The fact that $C+aK$ is a cone follows by the same arguments in Step 5 in the proof of \cite[Lemma 5.4]{demasi2021rectifiability}.

\medskip
\noindent
\(\bullet\)\emph{$C$ and $K$ are half-planes and $C$ intersects $T_x \pM$ with angle $\theta$.}

Let $\cpair$ a tangent pair to $\pair$ at $x$.
For any $\lambda \in (0,1)$, $\pair$ has replacement $\tpairk{j}$ in every annulus $\an(x,\lambda r_j,\lambda^{-1}r_j)$ for $j$ sufficiently large and the rescaled pairs $(\dil{x,r_j})_\sharp\tpairk{j}$ converge to a pair $\tcpair$ which is  a replacement for $\cpair$ in $\an(0,\lambda,\lambda^{-1})$. Moreover, by the above discussion, $\tilde{C}+a\tilde{K}$ is a rectifiable varifold.

As above, for any $\sigma \in (0,\lambda)$ and for any $\rho > \lambda^{-1}$, we have
\begin{equation}
\begin{split}
\frac{\norm{\tilde{C}}(B_{\sigma}) + a \norm{\tilde{K}}(B_{\sigma})}{\sigma^2}
= &
\frac{\norm{{C}}(B_{\sigma}) + a \norm{{K}}(B_{\sigma})}{\sigma^2}
\\
= &
\frac{\norm{{C}}(B_{\rho}) + a \norm{{K}}(B_{\rho})}{\rho^2}
\\
= &
\frac{\norm{\tilde{C}}(B_{\rho}) + a \norm{\tilde{K}}(B_{\rho})}{\rho^2}.
\end{split}
\end{equation}
Again this and the same arguments in Step 5 in the proof of \cite[Lemma 5.4]{demasi2021rectifiability} prove that $\tilde{C}+a\tilde{K}$ is a rectifiable cone. The same arguments as above show that both $\tilde{C}, \tilde{K}$ are rectifiable varifolds.

Since $\tilde{C}+a\tilde{K}$ is a cone and $\tilde{C}, \tilde{K}$ are smooth surfaces in $\an(0,\lambda,\lambda^{-1})$ with integer multiplicity, it follows that both $\tilde{C}$ and $\tilde{K}$ are cones in $\an(0,\lambda,\lambda^{-1})$; indeed, the density of $\tilde{C}+a\tilde{K}$ takes values in $\N \cup (a+\N)$ (since $\tilde{K}$ has multiplicity 1) and has to be constant on half lines starting from the origin (because $\tilde{C}+a\tilde{K}$ is a cone), then the densities of $\tilde{K}$ and $\tilde{C}$ must be constant on lines.

Since $\tcpair$ is a replacement of $\cpair$ in $\an(0,\lambda,\lambda^{-1})$ and $\tilde{C}$ and $\tilde{K}$ are cones in $\an(0,\lambda,\lambda^{-1})$, if we consider the pair of smooth surfaces which induce $\tcpair$ in $\an(0,\lambda,\lambda^{-1})$ and we extend them in a conical way up to the origin and the infinity, the pair of surfaces $\ts{}$ and $\tg{}$ are stationary and stable in $\ru \- \{0\}$, thus we can apply the Bernstein Theorem (Theorem \ref{thm:Bernstein_cap}) to conclude that $\ts{}$ and $\tg{}$ are half planes and $\ts{}$ meets $T_x \pM$ with contact angle $\theta$.

By the fact that $C + aK$ is a cone which coincides with $\ts{} + a \tg{}$ in $\an(0,\lambda,\lambda^{-1})$, it follows that it coincides with $\ts{} + a \tg{}$ on the whole $\ru$. This in particular implies, since $\tg{}$ and $\supp \norm{K}$ are contained in $T_x \pM$, that
\begin{equation}\label{eq:cone_coincides_replacement}
C \cap \interior{(T_x \M)}
=
\ts{} \cap \interior{(T_x \M)}.
\end{equation}
This in turn implies that any replacement in an annulus for $\cpair$ coincide with $\ts{}$ in the annulus and in  $\interior{(T_x \M)}$, that is the surface $\ts{}$ is independent on the choice of $\lambda$.

%

We claim that $\norm{C}(T_x \pM)=0$. Indeed, if were not, there would exist $y \in T_x \pM \- \{0\}$ and a small ball $B_r(y)$ such that $0 \notin B_r(y)$ and $\norm{C}\big( B_r(y) \cap \pM \big) >0$. Let us consider $\lambda$ so small so that $B_r(y) \subset \an(0,\lambda,\lambda^{-1})$. Let us consider the replacement $\tcpair$ of $\cpair$ in $\an(0,\lambda,\lambda^{-1})$. We have
\begin{equation}
\begin{gathered}
\norm{C} \big(B_{\lambda^{-1}}\big)
=
\norm{\tilde{C}}\big(B_{\lambda^{-1}})
\\
C \cap B_\lambda = \tilde{C} \cap B_\lambda
\end{gathered}
\end{equation}
This implies that $\norm{C}\big( \an(0, \lambda, \lambda^{-1})\big) = \norm{\tilde{C}}\big( \an(0, \lambda, \lambda^{-1})\big)$ but the fact that
\begin{equation}
\norm{\tilde{C}}\big( T_x \pM) \cap \an(0,\lambda,\lambda^{-1}) \big)
=
\hn\big(\ts{} \cap T_x \pM \cap \an(0,\lambda,\lambda^{-1}) \big)
=
0
\end{equation}
contradicts \eqref{eq:cone_coincides_replacement}.
Hence $\norm{C}(T_x \pM)=0$ and $C = \ts{}$. In turn, by $C + a K = \ts{}+a \tg{}$, it follows that $K = \tg{}$. Hence $C$ and $K$ are half-planes and $C$ intersects $T_x \pM$ with angle $\theta$.

\medskip
\noindent
\(\bullet\)\emph{$V$ is integer rectifiable.}

We have proved that, for every tangent pair $\cpair$ to $\pair$ at $x$, it holds $\norm{C}(T_x \pM) =0$. Since $(\dil{x,r_j})_\sharp V \wto C$, we have that
\begin{equation}
\limsup_{j \to \infty} \frac{\norm{V}\big(B_{r_j}(x) \cap \pM \big)}{r_j^2}
=
\limsup_{j \to \infty} \norm{(\dil{x,r_j})_\sharp V}\big(B_{1} \cap \dil{x,r_j}(\pM) \big)
\leq
\norm{C} \big( B_1 \cap \pM \big)
=
0.
\end{equation}
A straightforward covering argument yields $\norm{V}(\pM)=0$. Since $V$ is integer rectifiable in the interior of $\M$, it is integer rectifiable in all $\M$.
\end{proof}

\subsection{Regularity of $V$}


We first note  that \(V\) is regular and stable outside a finite number of point
\begin{proposition}\label{prop:regularity_punctured}
Let $\pair$  be the pair of varifolds given by Proposition \ref{prop:almost_minimizing_min-max_sequences} and let us fix $x \in \pM$. Then there exists $\rho>0$ such that $V$ is induced in $B_\rho(x) \- \{x\}$ by a regular minimal surface \(\Sigma\)  which meets $\pM$ with angle $\theta$. Moreover \(\Sigma\) is stable for the capillarity  energy.
\end{proposition}
\begin{proof}
The proof is exactly the same as \cite[10.3]{delellis2017minmax} but we sketch here the argument for the reader convenience.

Let us consider $x \in \pM$, $r(x)$  so that $\pair$ has replacements in any annulus $\an \in \anul{r(x)}{x}$. Let us fix $\rho >0$ such that $4\rho < r(x)$. By the results in \cite{delellis2009existence}, we have that $V$ is regular in $\interior{\M}$. Let us call $\s_1,\s_2, \dots$ the connected components of $V$ in $\an(x,\rho,4\rho) \cap \interior{\M}$ and let us fix one of them $\s_i$ with a regular point $y \in \s_i$. Since $y$ is regular for $\s_i$, if $s=|x-y|$ we can find $\bar{\sigma}>0$ such that
\begin{itemize}
\item $\s_i$ is regular in $B_{\bar{\sigma}}(y)$;
\item $\s_i \cap \de B_{s +\sigma}(x) \neq \emptyset$ (by maximum principle \cite[Theorem 5.1 $(ii)$]{delellis2009existence}) for every $\sigma \in (0,\bar{\sigma})$;
\item the intersection between $\s_i$ and $\de B_{s + \sigma}(x)$ is transversal (by Sard's Lemma) for almost every $\sigma \in (0,\bar{\sigma})$.
\end{itemize}
We consider such a $\sigma$, call $t = s + \sigma$ and fix a point $z \in \s_i \cap \de B_t(x) \cap \interior{\M}$. We consider a replacement $\tpair$ for $\pair$ in $\an(x,\rho,t)$. Again by the interior theory in \cite{delellis2009existence}, $\tilde{V}$ is regular in $\interior{\M}$.
Since $V$ and $\tilde{V}$ coincide outside $B_t(x)$, any tangent cone $C$ to $\tilde{V}$ at $z$ must contain a half-plane $\pi$ (since $z$ is regular for $\s_i$); since $C$ is stationary and stable, it is a plane, hence $C$ coincides with the whole plane $\pi$. This implies that $z$ is regular for $\tilde{V}$.

If we call $\tilde{\s}$ the connected component of $\tilde{V}$ that contains $z$, by unique continuation, $\tilde{\s}$ and $\s_i$ coincide and, since $\tilde{\s}$ is regular up to $\pM$ in $\an(x,\rho,t)$ (because $\tpair$ is a replacement for $\pair$), also $\s_i$ is regular up to $\pM$ in $\an(x,\rho,t)$. Since the choice of $y$ is arbitrary, we have that $\s_i$ is regular up to $\pM$ in the whole annulus $\an(x,\rho,4\rho)$.

By monotonicity formula, it follows that there are only finitely many connected components of $V$ in $\an(x,2\rho,3\rho)$: $\clos{\s}_1, \dots, \clos{\s}_N$ and each of them is a regular minimal surface which meets $\pM$ with angle $\theta$; let us call $S_i$ ($i=1,\dots, N$) the singular sets of the surfaces $\clos{\s}_i$.

Let us consider a point $y \in \clos{\s}_i \cap \an(x,2\rho,3\rho) \- \bigcup_{j=1}^N S_j$.
We claim that $y$ is a regular point for $V$; of course, since $y$ is regular for $\clos{\s}_i$, to prove this it is sufficient to show that $y$ does not belong to any of the others $\clos{\s}_j$; if this were the case for a $\clos{\s}_j$, $y$ would be regular also for $\clos{\s}_j$, and since $\s_i \cap \s_j \cap \interior{\M} = \emptyset$, then $T_y \clos{\s}_i = T_y \clos{\s}_j$, but this contradicts the maximum principle.

On the other hand, since \(\partial M\) is convex, $\nv(\pM)=0$, and any point $y \in \supp \nv \cap \an(x,2\rho,3\rho)$ belongs to some $\clos{\s}_i$.

This shows that $V$ is regular up to $\pM$ in $\an(x,\rho,4\rho)$; since the radius $\rho$ is arbitrary (except for the fact that $4\rho < r(x)$), this argument proves that $V$ is regular in $B_{r(x)/4}(x) \- \{x\}$.

Stability of \(\Sigma\) in  $B_{r(x)/4}(x) \- \{x\}$ follows again by checking the regularity proof of \cite[10.3]{delellis2017minmax}, since there it is shown that \(V\) coincides with each of its replacements in any $\an \in \anul{r(x)}{x}$
\end{proof}

In order to conclude the proof we only need to remove the final (possible) singular point \(x\). Here, the stability of \(V\) in a in a (small) punctuerd ball plays a crucial role. The following Proposition is indeed a generalization to capillarity surfaces of  of the ``removable'' singularity theorem for two dimensional stable minimal surfaces first established by Lawson and Gulliver,~\cite{GulliverLawson86}, see also \cite{simon1982isolated} and~\cite{Meeks07}. Although the proof of Lawson and Gulliver can be probably generalized to this setting , we follow here a different route. To explain the idea, let us  assume that \(\M=\{x_1\ge0\}\) and that \(\Sigma\) is globally stable. Then it is well known that 
\[
v={\nu_{\Sigma}} \cdot{e_1}
\]
is a jacobi field which is negative at the boundary due to the contact angle condition. By the stability inequality \(v\) should be negative also inside since otherwise its positive part would induce a negative variation. This implies that \(\Sigma\) can be expressed as a union of graph with respect to \(\{x_1=0\}\). Each of this graph will be a  stationary point  point of the functional 
\[
\mathcal{F} (g)=\int \sqrt{1+|\nabla g|^2}+a|\{g>0\}|
\]
and the regularity will follow by the regularity theory of~\cite{AltCaffarelliFriedman84} (see also~\cite{CaffarelliFriedman85}) and more precisely the implementation in~\cite{De-Silva11}). In order to make the proof  rigorous we shall first localize the above argument and then precise in which sense \(g\) is solving a free boundary problem.

\begin{proposition}
Let $\pair$  be the pair of varifolds given by Proposition \ref{prop:almost_minimizing_min-max_sequences}, let us fix $x \in \pM$ and let $\rho>0$ be the radius given by Proposition \ref{prop:regularity_punctured}. Then $V$ is regular in the whole $B_\rho(x)$, that is is induced in $B_\rho(x)$ by a regular minimal surface which meets $\pM$ with angle $\theta$.
\end{proposition}
\begin{proof} The proof is divided in few steps:

\medskip
\noindent
\emph{Step 1}: First by arguing exactly as in  \cite[10.4]{delellis2017minmax}, we note that $V\llcorner B_{\rho}(x)$ consists by finitely many connected components $\s_i$ with integer multiplicities $h_i$ which are homeomorphic to $2$-balls without the center $x$.

\medskip
\noindent
\emph{Step 2}: We now fix such a surface $\s^i = \s$, remove the multiplicity $h_i$ and we aim to show that it can be extended smoothly through \(x\). We first  claim that for \(y \in \Sigma \cap \pM \cap B_{\rho}(x)\) the 
sequence \(\s^y_j := (\dil{y,r_j})_\sharp \s\) converges (up to subsequences) to  a plane \(\pi_y\) which forms the correct angle with \(T_{y} \pM\). Note that this is obvious if \(y\ne x\) (since the surface is regular there) and we only need to show it for \(y=x\). For, we note that   the curvature estimates (Theorem \ref{thm:curvature_estimates}) yield the existence of a constant $c>0$ such that
\begin{equation}
|\as(y)|
\leq
\frac{c}{|y -x|}
\qquad
\forall y \in B_\rho(x).
\end{equation}
so that \(\s^x_j\) converges locally smoothly in \(\ru \setminus \{0\}\) to an entire surface \(\Xi\) which is smooth and stable outside the origin. By  Theorem \ref{thm:Bernstein_cap}, \(\Xi\) is a plane through the origin.

\medskip
\noindent
\emph{Step 2}: We now show that \(\Sigma\) is a Lipschitz graph with respect to \(T_{x} \pM\). To this end we can assume that\[
N_{T_x \pM}
=
-e_1.
\]
We now  let \(r_j\) be a sequence such that \(\s_j := (\dil{x,r_j})_\sharp \s\) converges to a plane \(\Xi\) locally smoothly. Note that the plane might depend on the subsequence, but in any case we can rotate the coordinates so that 
\[
\nu_\pi
=
-a e_1 + \sqrt{1-a^2}e_2.
\]
Since the convergence of $\s_j \to \pi$ is in the $C^1$-topology on annuli, for every $\e >0$ there exists $J \in \N$ such that
\begin{equation}
\sup_{y \in B_{r_j}(x) \- B_{r_j/4}(x)} |\nu_\s(y) - \nu_\pi|
<
\e
\qquad
\forall j>J,
\end{equation}
where $\nu_\s$ is the unit normal vector to $\s$. Hence, for any $k > j > J$, we have
\begin{equation}\label{eq:normal_annuli_close_plane}
\sup_{y \in \s \cap(\de B_{r_j}(x) \cup \de B_{r_k}(x))} |\nu_\s(y) - \nu_\pi|
<
\e.
\end{equation}
Let us consider for \(e\in \mathbb S^2\) and \(|e-e_1|\le \delta\) to be fixed ,the function
\begin{equation}
v(y)
=
\nu_\s(y) \cdot e
\quad
\forall y \in \s.
\end{equation}
By \eqref{eq:normal_annuli_close_plane}, it follows that for every $\e, \delta>0$ there exists $J \in \N$ such that
\begin{equation}
\sup_{y \in \s \cap(\de B_{r_j}(x) \cup \de B_{r_k}(x))}
v(y)
\leq
-a + \e+\delta
\qquad
\forall k > j \geq J.
\end{equation}
Moreover, by the contact angle condition and the regularity of $\s$ up to $\pM$, we get
$v(y) = -a$ for every on $y \in \s \cap \pM \cap \an(x,r_k,r_j)$. Hence
\begin{equation}
\sup_{y \in \s \cap\big(\pM \cup \de B_{r_j}(x) \cup \de B_{r_k}(x)\big)}
v(y)
\leq
-a + \e+\delta<0
\qquad
\forall k > j \geq J.
\end{equation}
provided \(\delta \) and \(\e\) are chosen sufficiently small but depending only on \(a\).

It is a classical fact that $v$ is a Jacobi field, that is 
\[
\Delta_\s v + |\as|^2v = 0.
\]
We now claim that $v$ has constant non-positive sign also in the interior of the annulus $\an(x,r_k,r_j)$. Indeed, the positive part $v_+ := \max\{v,0\}$ is a function with compact support in the interior of $\an(x,r_k,r_j)$ and we can test the weak form of the Jacobi equation with the function $v_+$ as test function, that is
\begin{equation}
\begin{split}
0
= &
\int_\s \nabla v_+ \cdot \nabla v - |\as|^2 v v_+ \dif \hn
\\
= &
\int_\s |\nabla v_+|^2 - |\as|^2 (v_+)^2 \dif \hn
\\
= &
Q(v_+)
\end{split}
\end{equation}
where $Q$ is the stability operator defined in \eqref{eq:definition_stability_operator}.
Since $Q \geq 0$ (by stability of $\s$), it follows that $v_+$ also is a non negative solution of the Jacobi equation, a contradiction with the strong minimum principle. Hence for each \(e\) sufficiently close to \(e_1\), 
\[
\nu_\s \cdot e<0 \qquad\text{on \(\Sigma \cap \an(x,r_k,r_j)\)}
\]
which is easily seen to imply that \(\Sigma \cap \an(x,r_k,r_j)\) can be represented as the graph of a Lipschitz function \(g: \R^2\to \R\) 
\[
\Sigma \cap \an(x,r_k,r_j)=\{x_1=g(x_2,x_3)\}\cap \an(x,r_k,r_j)
\]
and \(|\nabla g|\le 1/\delta\) for $\delta >0$ depending only on $a$. Letting \(k \to \infty\) we deduce that \(\Sigma\cap B_{r_j}\) is the graph of a Lipschitz function.

\medskip
\noindent
\emph{Step 3}: We now claim that the function \(g\) in the above step is a solution of a free boundary problem. To this end note first that, up to taking \(j\) large and a  scaling (and translation) of  the coordinates that we can assume that 
\[
\Sigma\cap B_1=\{x_1=g\} \cap B_1,
\qquad
\pM\cap B_1=\{x_1=h\} \cap B_1
\]
for a suitable function \(h\) with \(h(0)=|\nabla h (0)|=0\). Note that \(g \ge h\). It is now easy to check that \(g\) is a viscosity solution of the equation 
\[
\begin{cases}
-\Div\Bigl(\dfrac{\nabla g}{\sqrt{1+|\nabla g|^2} }\Bigr)=0\qquad &\text{on \(\{g>h\}\)}
\\
-\nabla g \cdot {\nabla h}+1=a\sqrt{1+|\nabla g|^2}  \sqrt{1+|\nabla h|^2} \qquad &\text{on \(\{g=h\}\)}.
\end{cases}
\]
Indeed the first equation is clearly satisfied classically. As for the boundary condition we first note that  for \(x_0\in \{g=h\}\), the point \((g(x_0), x_0)\in\Sigma\cap \pM\)  and thus, thanks to Step \(1\), the functions 
\[
g_r(x)=\frac{g(x_0+rx)-g(x_0)}{r}
\]
converges (up to subsequences) to a linear function \(\ell\) with the property that 
\[
-\nabla \ell \cdot {\nabla h(x_0)}+1=a\sqrt{1+|\nabla \ell|^2}  \sqrt{1+|\nabla h(x_0)|^2}. 
\]
It is easy to verify that this implies that the condition is satisfied in the viscosity sense, see for instance the computations in~\cite{De-PhilippisSpolaorVelichkov21}. Note that this also implies that, up to choose \(r_j\) smaller, we can assume that the free boundary \(\{g = h\}\) 
\[
 \|g-\ell\|_{L^\infty)B'_1(0)}<\e
\]
where \(\ell(x)=x_2 \sqrt{(1-a^2)/a^2}\) (recall the normalizations at \(x\)). One can then apply some straightforward modifications of the arguments in~\cite{De-Silva11} to show that the free boundary a \(C^1\) curve. Note here that  the Lipschitz assumption on \(g\)  ensures the  uniform ellipticity of the minimal surface operator.

\medskip
\noindent
\emph{Step 4}: The previous step implies that each of the \(\Sigma^j\) are smooth graphs near \(x_0\) and that they satisfy the minimal surface equation with constant  capillarity conditions. By Hopf Lemma, they have to coincide, so that \(x\) is a regular point as well.

\end{proof}

\appendix
\section{Proof of the results of section \ref{subsec:monotonicity}}\label{sec:appendix_monotonicity}
\begin{proof}[Proof of Proposition \ref{prop:bv_angle_varifolds}]
If $X \in \Xm$, there exist $X^T, X^\perp$ such that $X=X^T+X^\perp$, with $X^T \in \Xt$ and $X^\perp \in \Xp$. We have
\begin{equation*}
\dive_S X(x)
=
\dive_S X^T(x)
+
\dive_S X^\perp(x).
\end{equation*}
By Definition \ref{eq:def_contact_angle}, both $V$ and $V+aW$ have generalized mean curvature $H \ind_{\interior{\M}}$ with respect to $\Xo$; thus by \cite[Theorem 1.1]{demasi2021rectifiability}, there exists a positive Radon measure $\muv$ on $\pM$ and a $\nv$-measurable vector field $\tilde{H}$ such that
\begin{gather}\label{eq:normal_variation_V}
\begin{split}
\int_{G_2(\M)} \dive_S X^\perp(x) \dif V(x,S)
= &
- \int_\M {X^\perp}\cdot{H \ind_{\interior{\M}} + \tilde{H}} \dif \nv
+ \int_{\pM} {X}\cdot{N} \dif \muv
\\
=&
-
\int_\M {X^\perp}\cdot{H + \tilde{H}} \dif \nv
+ \int_{\pM} {X}\cdot{N} \dif \muv,
\end{split}
\end{gather}
where the last equality follows because $H$ is assumed to be tangent to $\pM$ on $\pM$ (see last part of Definition \ref{def:contact_angle_condition}). Moreover by the definition of $\tilde{H}$ given in \cite[(3.12)]{demasi2021rectifiability}, it follows that in our case $\tilde{H} = - N(x) \dive_{\pM} N(x)$ (since on $\pM$, $V$ charges only planes which are tangent to $\pM$, by \cite[Lemma 3.1]{demasi2021rectifiability}). The estimates
\eqref{eq:estimate_muv_global} and \eqref{eq:estimate_measure_boundary} on $\muv$ are given by \cite[Theorem 1.1]{demasi2021rectifiability}.

For the same reason, also $V+aW$ has generalized mean curvature $H \ind_{\interior{\M}}$ with respect to $\Xo$ and satisfies the properties as above with the same $\tilde{H}$ and $\muv$ (because $\muv$ depends only of the behavior of the varifold in the interior of $\M$, see its definition in \cite{demasi2021rectifiability}).

Thus $W$ has generalized mean curvature $0$ with respect to $\Xo$ and
\begin{equation}
\int_{G_2(\M)} \dive_{\pM} X^\perp(x) \dif \nw
=
- \int_{\pM} {X^\perp}\cdot{\tilde{H}} \dif \nw
\end{equation}
for the same $\tilde{H}$ as above.
For what concerns the tangent part $X^T$, Definition \ref{def:contact_angle_condition} yields
\begin{equation*}
\begin{split}
\int_{G_2(\M)} \dive_S X^T(x) \dif V(x,S)
+
a \int_{G_2(\M)} \dive_S X^T(x) \dif \nw
=
- \int_\M {X^T}\cdot{H} \dif \nv.
\end{split}
\end{equation*}
This shows the conclusion.
\end{proof}

In order to prove Proposition \ref{prop:monotonicity_formula_p}, we first derive a monotonicity inequality for the masses.
\begin{lemma}[Monotonicity inequality]\label{lmm:monotonicity_inequality_s}
Suppose $\pair \in \cu$ satisfy the contact angle condition $\theta$, with $H \in L^p(\M,\nv)$ for some $p \in [1,+\infty)$. Then there exists a constant $c>0$ that depends only on $p$ and on the second fundamental form of $\pM$ such that, for all $x_0 \in \pM$, the following inequality holds:
\begin{equation}\label{eq:monotonicity_inequality_s}
\begin{split}
(1+c\rho)
\frac{\dif}{\dif \rho}
\bigg(
\frac{1}{\rho^2}
&
\int_{\M} \gamma\Big( \frac{|x|}{\rho}\Big)
\dif (\nv + a \nw)
\bigg)^{\frac{1}{p}}
\\
\geq &
-
\rho^{-\frac{2}{p}}
\left[
\frac{1+\rho}{p}
\left(
\int_{\M} \gamma\Big( \frac{|x|}{\rho}\Big)|H +\tilde{H}|^p \dif \nv
\right)^{\frac{1}{p}}
-
\frac{a}{p}
\left(
\int_{\M} \gamma\Big( \frac{|x|}{\rho}\Big)|\tilde{H}|^p \dif \nw
\right)^{\frac{1}{p}}
\right]
\\
& -
c(1+\rho)
\left( \frac{1}{\rho^2}
\int_{\M} \gamma\Big( \frac{|x|}{\rho}\Big) \dif (\nv + a \nw)
\right)^{\frac{1}{p}}
\end{split}
\end{equation}
\end{lemma}
\begin{proof}
Without loss of generality we can suppose $x_0 = 0$. Since for large $\rho$ the statement is obvious, we have to prove it only for $0<\rho< R(\M)$, where $R(\M)$ is the minimum radius of curvature of $\M$. We want to bound from below the following derivative:
\begin{equation}\label{eq:monotonicity_derivative_initial}
\begin{split}
\frac{\dif}{\dif \rho}
&
\left(
\frac{1}{\rho^2}
\int_{\M} \gamma\Big( \frac{|x|}{\rho}\Big)
\dif (\nv + a \nw)
\right)^{\frac{1}{p}}
\\
\quad
&
=
\frac{1}{p}
\frac{\dif}{\dif \rho}
\left(
\frac{1}{\rho^2}
\int_{\M} \gamma\Big( \frac{|x|}{\rho}\Big)
\dif (\nv + a \nw)
\right)
\left(
\frac{1}{\rho^2}
\int_{\M} \gamma\Big( \frac{|x|}{\rho}\Big)
\dif (\nv + a \nw)
\right)^{\frac{1-p}{p}}.
\end{split}
\end{equation}
To do so, we want to bound from below the derivative in the right-hand side to get a differential inequality. We have
\begin{equation*}
\begin{split}
\frac{\dif}{\dif \rho}
&
\left(
\frac{1}{\rho^2}
\int_{\M} \gamma\Big( \frac{|x|}{\rho}\Big)
\big( \dif \nv(x) + a \nw(x) \big)
\right)
\\
&
\quad
=
-\frac{1}{\rho^{3}}
\int_{G_2(\M)}
\left( 2 \gamma\Big( \frac{|x|}{\rho}\Big)
+
\frac{|x|}{\rho} \gamma'\Big( \frac{|x|}{\rho}\Big)
\right) \dif (V(x,S) + a W(x,S)).
\end{split}
\end{equation*}
Let us choose $X(x)=\gamma\Big( \frac{|x|}{\rho}\Big) x$. Then
\begin{equation*}
\begin{split}
\dive_S X(x)
=
2 \gamma\Big( \frac{|x|}{\rho}\Big)
+
\frac{|x|}{\rho} \gamma'\Big( \frac{|x|}{\rho}\Big)
\left| P_S\frac{x}{|x|}\right|^2 .
\end{split}
\end{equation*}
We use Proposition \ref{prop:bv_angle_varifolds}: by testing \eqref{eq:total_fvf_varifold_angle} with $X$ we get the following \emph{monotonicity identity}:
\begin{equation}\label{eq:monotonicity_identity_fb}
\begin{split}
\frac{\dif}{\dif \rho}
\left(
\frac{1}{\rho^2}
\int_{\M} \gamma\Big( \frac{|x|}{\rho}\Big)
\dif (\nv + a \nw)
\right)
= &
-
\frac{1}{\rho^{3}} \int_{G_2(\M)} \dive_S X(x) \dif \big(V(x,S) + a W(x,S) \big)
\\
& -
\frac{1}{\rho^{3}} \int_{G_2(\M)} \frac{|x|}{\rho} \gamma'\Big( \frac{|x|}{\rho}\Big) \left| P_{S^\perp}\frac{x}{|x|}\right|^2 \dif \big(V(x,S) + a W(x,S) \big)
\\
=
& \frac{1}{\rho^{3}} \int_\M {X}\cdot{H+\tilde{H}} \dif \nv
- \frac{1}{\rho^{3}} \int_{\pM} {X}\cdot{N} \dif \muv
\\
&+
\frac{a}{\rho^{3}} \int_\M {X}\cdot{\tilde{H}} \dif \nw
\\
&+
-
\frac{1}{\rho^{3}} \int_{G_2(\M)} \frac{|x|}{\rho} \gamma'\Big( \frac{|x|}{\rho}\Big) \left| P_{S^\perp}\frac{x}{|x|}\right|^2 \dif \big(V(x,S) + a W(x,S) \big).
\end{split}
\end{equation}
We have to estimate from below the last member of the above identity.

Since $\gamma' \leq 0$, we can neglect the last integral and, since $|x| \leq \rho$ by cut-off, we obtain
\begin{equation}\label{eq:monotonicity_estimate_fb}
\begin{split}
\frac{\dif}{\dif \rho}
\left(
\frac{1}{\rho^2}
\int_{\M} \gamma\Big( \frac{|x|}{\rho}\Big)
\dif (\nv + a \nw)
\right)
\geq &
- \frac{1}{\rho^{2}} \int_\M \gamma\Big( \frac{|x|}{\rho}\Big)|H +\tilde{H}| \dif \nv
- \frac{a}{\rho^{2}} \int_\M \gamma\Big( \frac{|x|}{\rho}\Big)|\tilde{H}| \dif \nw
\\
& -
\frac{1}{\rho^{2}} \int_{\pM} \gamma\Big( \frac{|x|}{\rho}\Big)
\Bigl| {\frac{x}{|x|}}\cdot{N} \Bigr|
\dif \muv.
\end{split}
\end{equation}
We have now to bound the three terms in the right-hand side of \eqref{eq:monotonicity_estimate_fb}.
\begin{itemize}
\item
For the first one, since $H \in L^p(\M,\nv)$ and $\tilde{H} \in L^{\infty}(\pM,\nv)$, also $H + \tilde{H} \in L^p(\M,\nv)$. Therefore, by H\"older inequality we get
\begin{equation}\label{eq:estimate_curvature_lp}
\begin{split}
\frac{1}{\rho^2}
\int_\M \gamma\Big( \frac{|x|}{\rho}\Big)|H+\tilde{H}| \dif \nv
\leq &
\frac{1}{\rho^2}
\left(
\int_{\M} \gamma\Big( \frac{|x|}{\rho}\Big)|H +\tilde{H}|^p \dif \nv
\right)^{\frac{1}{p}}
\left(
\int_\M \gamma\Big( \frac{|x|}{\rho}\Big) \dif \nv
\right)^{1-\frac{1}{p}}.
\end{split}
\end{equation}
\item
For the second one, since $|\tilde{H}| \in L^{\infty}(\pM)$, by H\"older inequality and $\gamma \leq 1$ again, we obtain
\begin{equation}\label{eq:estimate_curvature_lp_W}
\begin{split}
\frac{1}{\rho^2}
\int_\M \gamma\Big( \frac{|x|}{\rho}\Big)|\tilde{H}| \dif \nw
\leq &
\frac{1}{\rho^2}
\left(
\int_{\M} \gamma\Big( \frac{|x|}{\rho}\Big)|\tilde{H}|^p \dif \nw
\right)^{\frac{1}{p}}
\left(
\int_\M \gamma\Big( \frac{|x|}{\rho}\Big) \dif \nw
\right)^{1-\frac{1}{p}}.
\end{split}
\end{equation}
\item
We now move on the estimate of the third integral in the right-hand side of \eqref{eq:monotonicity_estimate_fb}. Since $\pM$ is of class $C^2$ and since $0 \in \pM$, there exists a constant $c$ such that
\begin{equation}\label{eq:estimate_scalar_normal_radial}
|{\frac{x}{|x|}}\cdot{N(x)}|
\leq
c |x|.
\end{equation}
This yields
\begin{equation}\label{eq:first_estimate_measure_monotonicity}
\frac{1}{\rho^{2}} \int_{\pM} \gamma\Big( \frac{|x|}{\rho}\Big)
\big|{\frac{x}{|x|}}\cdot{N}\big|
\dif \muv
\leq                                  						
\frac{c}{\rho} \int \gamma\Big( \frac{|x|}{\rho}\Big) \dif \muv.
\end{equation}
We have to further estimate the righ-hand side of this inequality. This is done by testing \eqref{eq:normal_variation_V} with $X(x)=-\gamma\Big( \frac{|x|}{\rho}\Big) \nabla d(x)$; by standard computation we get
\begin{equation*}
\begin{split}                          						
\frac{1}{\rho} \int \gamma\Big( \frac{|x|}{\rho}\Big) \dif \muv
= &																
-\frac{1}{\rho} \int_{G_2(\M)} \gamma'\Big(\frac{|x|}{\rho}\Big) \frac{1}{\rho} {P_S \frac{x}{|x|}}\cdot{\nabla d(x)} \dif V(x,S)
\\
& -
\frac{1}{\rho}
\int_{G_2(\M)} \gamma\Big( \frac{|x|}{\rho}\Big) \dive_S \nabla d(x) \dif V(x,S)
\\
& -
\frac{1}{\rho}
\int_\M \gamma\Big( \frac{|x|}{\rho}\Big) {\nabla d(x)}\cdot{H +\tilde{H}} \dif \nv(x).
\end{split}
\end{equation*}
Since $\gamma'(s)=0$ if $s \in (0,1/2)$, we have $\frac{|x|}{\rho} \geq \frac{1}{2}$; Moreover, using $|\dive_S \nabla d|\leq c$ (because $\rho<R(\M)$ and $d$ is of class $C^2$ in $\clos{U_{R(\M)}(\pM)}$ ) and \eqref{eq:estimate_curvature_lp}, we obtain
\begin{equation}\label{eq:penultima_stima_monotonicity}
\begin{split}
\frac{1}{\rho} \int \gamma\Big( \frac{|x|}{\rho}\Big) \dif \muv
\leq &																
-\frac{2}{\rho} \int_{\M} \gamma'\Big(\frac{|x|}{\rho}\Big) \frac{|x|}{\rho^2}
\dif (\nv + a \nw)
+ \frac{c}{\rho}  \int_\M \gamma\Big( \frac{|x|}{\rho}\Big) \dif (\nv + a \nw)
\\
& +
\rho^{1-\frac{2}{p}}
\left(
\int_{\M} \gamma\Big( \frac{|x|}{\rho}\Big)|H +\tilde{H}|^p \dif \nv
\right)^{\frac{1}{p}}
\left( \frac{1}{\rho^2}
\int_\M \gamma\Big( \frac{|x|}{\rho}\Big) \dif (\nv + a \nw)
\right)^{1-\frac{1}{p}}
\\
= & 														
\frac{2}{\rho}
\frac{\dif}{\dif \rho}
\left(
\int_{\M} \gamma\Big( \frac{|x|}{\rho}\Big) \dif (\nv + a \nw)
\right)
-\frac{2}{\rho^2} \int_{\M} \gamma\Big( \frac{|x|}{\rho}\Big) \dif (\nv + a \nw)
\\
& +
\frac{2 +c\rho}{\rho^{2}}  \int_\M \gamma\Big( \frac{|x|}{\rho}\Big) \dif (\nv + a \nw)
\\
& +
\rho^{1-\frac{2}{p}}
\left(
\int_{\M} \gamma\Big( \frac{|x|}{\rho}\Big)|H +\tilde{H}|^p \dif \nv
\right)^{\frac{1}{p}}
\left( \frac{1}{\rho^2}
\int_\M \gamma\Big( \frac{|x|}{\rho}\Big) \dif (\nv + a \nw)
\right)^{1-\frac{1}{p}}.
\end{split}
\end{equation}
Concerning the last member, we now observe that
\begin{equation*}
\begin{split}
\frac{2}{\rho}
\frac{\dif}{\dif \rho}
\left(
\int_{\M} \gamma\Big( \frac{|x|}{\rho}\Big) \dif (\nv + a \nw)
\right)
-&
\frac{2}{\rho^2} \int_{\M} \gamma\Big( \frac{|x|}{\rho}\Big) \dif (\nv + a \nw)
\\
= &
2 \frac{\dif}{\dif \rho}
\left(
\frac{1}{\rho}
\int_{\M} \gamma\Big( \frac{|x|}{\rho}\Big) \dif (\nv + a \nw)
\right).
\end{split}
\end{equation*}
Substituting in \eqref{eq:penultima_stima_monotonicity} we get
\begin{equation*}
\begin{split}
\frac{1}{\rho} \int \gamma\Big( \frac{|x|}{\rho}\Big) \dif \muv
\leq &	\,												
\, 2 \frac{\dif}{\dif \rho}
\left(
\frac{1}{\rho}
\int_{\M} \gamma\Big( \frac{|x|}{\rho}\Big) \dif (\nv + a \nw)
\right)
\\
& +
\frac{2 +c\rho}{\rho^{2}}  \int_\M \gamma\Big( \frac{|x|}{\rho}\Big) \dif (\nv + a \nw)
\\
& +
\rho^{1-\frac{2}{p}}
\left(
\int_{\M} \gamma\Big( \frac{|x|}{\rho}\Big)|H +\tilde{H}|^p \dif \nv
\right)^{\frac{1}{p}}
\left( \frac{1}{\rho^2}
\int_\M \gamma\Big( \frac{|x|}{\rho}\Big) \dif (\nv + a \nw)
\right)^{1-\frac{1}{p}}.
\end{split}
\end{equation*}
Taking into account that
\begin{equation*}
\begin{split}
\frac{\dif}{\dif \rho}
&
\left(
\frac{1}{\rho}
\int_{\M} \gamma\Big( \frac{|x|}{\rho}\Big) \dif (\nv + a \nw)
\right)
\\
& \qquad =
\frac{1}{\rho^2}
\int_{\M} \gamma\Big( \frac{|x|}{\rho}\Big) \dif (\nv + a \nw)
+
\rho
\frac{\dif}{\dif \rho}
\left(
\frac{1}{\rho^{2}}
\int_{\M} \gamma\Big( \frac{|x|}{\rho}\Big) \dif (\nv + a \nw)
\right),
\end{split}
\end{equation*}
we obtain
\begin{equation}\label{eq:last_estimate_measure_monotonicity}
\begin{split}
\frac{1}{\rho} \int \gamma\Big( \frac{|x|}{\rho}\Big) \dif \muv
\leq &	\,	
\frac{4 + c\rho}{\rho^{2}}
\int_{\M} \gamma\Big( \frac{|x|}{\rho}\Big) \dif (\nv + a \nw)
+
2\rho
\frac{\dif}{\dif \rho}
\left(
\frac{1}{\rho^{2}}
\int_{\M} \gamma\Big( \frac{|x|}{\rho}\Big) \dif (\nv + a \nw)
\right)
\\
& +
\rho^{1-\frac{2}{p}}
\left(
\int_{\M} \gamma\Big( \frac{|x|}{\rho}\Big)|H +\tilde{H}|^p \dif \nv
\right)^{\frac{1}{p}}
\left( \frac{1}{\rho^2}
\int_\M \gamma\Big( \frac{|x|}{\rho}\Big) \dif (\nv + a \nw)
\right)^{1-\frac{1}{p}}.
\end{split}
\end{equation}
This complete the estimate of the third term in the right-hand side of \eqref{eq:monotonicity_estimate_fb}.
\end{itemize}
Gathering \eqref{eq:estimate_curvature_lp}, \eqref{eq:estimate_curvature_lp_W}, \eqref{eq:first_estimate_measure_monotonicity} and \eqref{eq:last_estimate_measure_monotonicity} in \eqref{eq:monotonicity_estimate_fb}, we can estimate \eqref{eq:monotonicity_derivative_initial} as follows:
\begin{equation*}
\begin{split}
\frac{\dif}{\dif \rho}
\bigg(
\frac{1}{\rho^2}
&
\int_{\M} \gamma\Big( \frac{|x|}{\rho}\Big)
\dif (\nv + a \nw)
\bigg)^{\frac{1}{p}}
\\
=
&
\frac{1}{p}
\frac{\dif}{\dif \rho}
\left(
\frac{1}{\rho^2}
\int_{\M} \gamma\Big( \frac{|x|}{\rho}\Big)
\dif (\nv + a \nw)
\right)
\left(
\frac{1}{\rho^2}
\int_{\M} \gamma\Big( \frac{|x|}{\rho}\Big)
\dif (\nv + a \nw)
\right)^{\frac{1-p}{p}}
\\
\geq &
-
\frac{1+\rho}{p}
\rho^{-\frac{2}{p}}
\left(
\int_{\M} \gamma\Big( \frac{|x|}{\rho}\Big)|H +\tilde{H}|^p \dif \nv
\right)^{\frac{1}{p}}
-
\frac{a \rho^{-\frac{2}{p}}}{p}
\left(
\int_{\M} \gamma\Big( \frac{|x|}{\rho}\Big)|\tilde{H}|^p \dif \nw
\right)^{\frac{1}{p}}
\\
& -
c(1+\rho)
\left( \frac{1}{\rho^2}
\int_{\M} \gamma\Big( \frac{|x|}{\rho}\Big) \dif (\nv + a \nw)
\right)^{\frac{1}{p}}
\\
& -
c \rho
\frac{\dif}{\dif \rho}
\left(
\frac{1}{\rho^{2}}
\int_{\M} \gamma\Big( \frac{|x|}{\rho}\Big) \dif (\nv + a \nw)
\right)^{\frac{1}{p}},
\end{split}
\end{equation*}
which is the desired inequality.
\end{proof}

\begin{proof}[Proof of Proposition \ref{prop:monotonicity_formula_p}]
Without loss of generality we can assume $x_0=0 \in \pM$. We first notice that we have to prove the result only for small $\rho$, since it is clearly true for large value of $\rho$.

If we call
\begin{equation}
f(\rho)
:=
\bigg(
\frac{1}{\rho^2}
\int_{\M} \gamma\Big( \frac{|x|}{\rho}\Big)
\dif (\nv + a \nw)
\bigg)^{\frac{1}{p}},
\end{equation}
by Lemma \ref{lmm:monotonicity_inequality_s}, there exists a constant $\Lambda = \Lambda\big(c,\norm{H}_{L^p(\M)},\M, (\nv+a\nw)(\M)\big) > 0$ (where $c$ is the constant given by Lemma \ref{lmm:monotonicity_inequality_s}) such that
\begin{equation}
f'(\rho) + \Lambda f(\rho) + \Lambda \rho^{-\frac{2}{p}}
\geq
0.
\end{equation}
Hence, by $p > 2$,
\begin{equation}
f'(\rho) + \Lambda f(\rho) + \Lambda \rho^{-\frac{2}{p}}
+ \frac{\Lambda^2}{1- \frac{2}{p}}\rho^{1 -\frac{2}{p}}
\geq
0.
\end{equation}
Multiplying the above inequality by $e^{\Lambda \rho}$, we get
\begin{equation}
\frac{\dif}{\dif \rho}
e^{\Lambda \rho}
\left[
\left(
\frac{1}{\rho^2}
\int_{\M} \gamma\Big( \frac{|x|}{\rho}\Big)
\dif (\nv + a \nw)
\right)^{\frac{1}{p}}
+ \Lambda \rho^{1- \frac{2}{p}}
\right]
\geq
0.
\end{equation}
Since the constant $c$ of Lemma \ref{lmm:monotonicity_inequality_s} does not depend on the choice of $\gamma$ (the estimates are independent on the choice of $\gamma$ unless that $\gamma'(s)=0$ for $s \in (0,1/2)$), letting $\gamma$ increase to $\ind_{[0,1)}$ we have that the function
\begin{equation*}
\rho
\longmapsto
e^{\Lambda \rho}
\left[
\left(
\frac{1}{\rho^2}
\int_{\M} \gamma\Big( \frac{|x|}{\rho}\Big)
\dif (\nv + a \nw)
\right)^{\frac{1}{p}}
+ \Lambda \rho^{1- \frac{2}{p}}
\right]
\end{equation*}
is monotone increasing.
\end{proof}

\begin{proof}[Proof of Proposition  \ref{prop:monotonicity_formula_infty}]
Since $H \in L^{\infty}(\M, \nv)$, in \eqref{eq:estimate_curvature_lp} and \eqref{eq:estimate_curvature_lp_W} we simply estimate
\begin{equation}
\begin{gathered}
\frac{1}{\rho^2}
\int_\M \gamma\Big( \frac{|x|}{\rho}\Big)|H +\tilde{H}| \dif \nv
\leq
\frac{\normi{H + \tilde{H}}}{\rho^2}
\int_\M \gamma \Big( \frac{|x|}{\rho}\Big) \dif \nv,
\\
\frac{1}{\rho^2}
\int_\M \gamma\Big( \frac{|x|}{\rho}\Big)|\tilde{H}| \dif \nw
\leq
\frac{\normi{\tilde{H}}}{\rho^2}
\int_\M \gamma \Big( \frac{|x|}{\rho}\Big) \dif \nw.
\end{gathered}
\end{equation}
For $\rho^{-1} \int \gamma \dif \muv$, we have
\begin{equation}\label{eq:last_estimate_measure_monotonicity_infty}
\begin{split}
\frac{1}{\rho} \int \gamma\Big( \frac{|x|}{\rho}\Big) \dif \muv
\leq &	\,	
\frac{4 + c\rho}{\rho^{2}}
\int_{\M} \gamma\Big( \frac{|x|}{\rho}\Big) \dif (\nv + a \nw)
+
2\rho
\frac{\dif}{\dif \rho}
\left(
\frac{1}{\rho^{2}}
\int_{\M} \gamma\Big( \frac{|x|}{\rho}\Big) \dif (\nv + a \nw)
\right)
\\
& +
\frac{\normi{H + \tilde{H}}}{\rho}
\int_\M \gamma \Big( \frac{|x|}{\rho}\Big) \dif \nv.
\end{split}
\end{equation}
Hence, computations similar to the proof of Lemma \ref{lmm:monotonicity_inequality_s} and of Proposition \ref{prop:monotonicity_formula_p} yield
\begin{equation}
\frac{\dif}{\dif \rho}
\bigg(
\frac{1}{\rho^2}
\int_{\M} \gamma\Big( \frac{|x|}{\rho}\Big)
\dif (\nv + a \nw)
\bigg)
\geq
- \frac{\Lambda}{\rho^2}
\int_{\M} \gamma\Big( \frac{|x|}{\rho}\Big)
\dif (\nv + a \nw)
\end{equation}
for some $\Lambda>0$ which depends only on $\M, \normi{H}, \normi{\tilde{H}}$ (which depends on the second fundamental form of $\pM$). By multiplying by $e^{\Lambda \rho}$ and arguing as in the proof of Proposition \ref{prop:monotonicity_formula_p}, we obtain the conclusion.
\end{proof}

%

\end{document}